\documentclass[11pt]{amsart}
\usepackage{amsthm,amsmath,amssymb}

\usepackage{geometry}
\usepackage{graphicx}

\numberwithin{equation}{section}
\numberwithin{figure}{section}
\theoremstyle{plain}
\newtheorem{thm}{Theorem}[section]
\newtheorem{lem}[thm]{Lemma}

\newtheorem{claim}[thm]{Claim}

\theoremstyle{remark}
\newtheorem{rmk}[thm]{Remark}

\newcommand{\M}{\operatorname{M}}
\newcommand{\Hf}{\operatorname{H}}

\newcommand{\Od}{\operatorname{O}}
\newcommand{\E}{\operatorname{E}}
\newcommand{\od}{\operatorname{\textbf{o}}}
\newcommand{\e}{\operatorname{\textbf{e}}}
\newcommand{\s}{\operatorname{\textbf{s}}}
\newcommand{\T}{\operatorname{T}}
\newcommand{\V}{\operatorname{V}}
\newcommand{\Pn}{\operatorname{P}}
\newcommand{\Q}{\operatorname{Q}}
\newcommand{\K}{\operatorname{K}}
\begin{document}

\title[Lozenge tilings of a halved hexagon II]{Lozenge tilings of a halved hexagon with an array of triangles removed from the boundary, part II}

\author{Tri Lai}
\address{Department of Mathematics, University of Nebraska -- Lincoln, Lincoln, NE 68588}
\email{tlai3@unl.edu}
%\thanks{This work was supported in part  by the Institute for Mathematics and its Applications with funds provided by NSF grant DMS-0931945}



\date{\today}

\dedicatory{}

\begin{abstract}
Proctor's work on staircase plane partitions yields an enumeration of lozenge tilings of a halved hexagon on the triangular lattice. Rohatgi later extended this tiling enumeration to a halved hexagon with a triangle cut off from the boundary. In the previous paper the author proved  a common generalization of Proctor's and Rohatgi's results  by enumerating lozenge tilings of a halved hexagon in the case an array of an arbitrary number of triangles has been removed from a non-staircase side. In this paper we consider the other case when the array of triangles has been removed from the staircase side of the halved hexagon. Our result also implies an explicit formula for the number of tilings of a hexagon with an array of triangles missing on the symmetry axis.
\end{abstract}

\maketitle
\section{Introduction}\label{Intro}
A \emph{plane partition} is an array of positive integers $p_{i,j}$ (called ``\emph{parts}") so that $p_{i,j}\geq \max(p_{i+1,j},p_{i,j+1})$. R. Proctor \cite{Proc} proved a simple product formula for the number of a class of staircase plane partitions. The plane partitions in Proctor's result are in bijection with lozenge tilings of a hexagon of side-lengths $a,b,c,a,b,c$ (in the  counter clockwise order, starting from the northwestern side)  on the triangular lattice with a `maximal staircase' cut off, denoted by $\mathcal{P}_{a,b,c}$ (see Figure \ref{halfhex6}(a)). When $a=b$, the region $\mathcal{P}_{a,b,c}$ can be viewed as  a half of a symmetry hexagon cut by a zigzag lattice line along the symmetry axis. In this point of view, we usually call the region $\mathcal{P}_{a,b,c}$ a \emph{halved hexagon} (with defects).  Here a \emph{lozenge} (or \emph{unit rhombus}) is the union of two unit equilateral triangles sharing an edge, and a \emph{lozenge tiling} of a region is a covering of the region by lozenges so that there are no gaps or overlaps. This way Proctor's result yields the  following tiling enumeration.

\begin{thm}\label{Proctiling}For any non-negative integers $a,$ $b$, and $c$ with $a\leq b$, we have
\begin{equation}
\M(\mathcal{P}_{a,b,c})=\prod_{i=1}^{a}\left[\prod_{j=1}^{b-a+1}\frac{c+i+j-1}{i+j-1}\prod_{j=b-a+2}^{b-a+i}\frac{2c+i+j-1}{i+j-1}\right],
\end{equation}
where $\M(R)$ denotes the number of lozenge tilings of a region $R$, and where empty products are taken to be 1.
\end{thm}

Tiling enumerations of halved hexagons with certain defects have been investigated by a number of authors (see e.g. \cite{Ciucu1}, \cite{Cutoff}, \cite{Ranjan}, \cite{Lai}, \cite{Lai3}).  In particular, the author \cite{Lai, Lai3} proved  a simple product formula for the tiling number of a halved hexagon in which an arbitrary number of triangles have been removed from the base (see the regions in Figure \ref{halfhex3c} and the tiling formulas in Lemma \ref{QAR}). This result has Proctor's theorem above as a special case.

We note that when $a=b$, Proctor's Theorem \ref{Proctiling} implies an exact enumeration for one of the ten symmetry classes of plane partitions, the \emph{transposed-complementary plane partitions} (see e.g.  \cite{Stanley}).

\begin{figure}
  \centering
  \includegraphics[width=10cm]{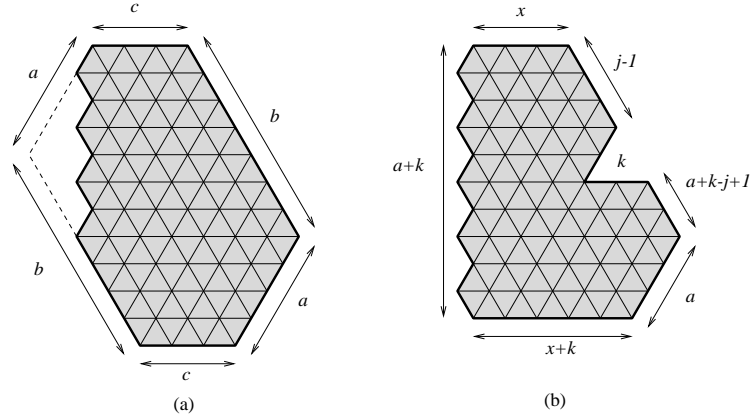}
  \caption{(a) Halved hexagon (with defects) $\mathcal{P}_{4,7,3}$. (b) The region in Rohatgi's paper \cite{Ranjan}.}\label{halfhex6}
\end{figure}

Lozenges in a region can carry weights. In the weighted case, we use the notation $\M(R)$ for the sum of weights of all lozenge tilings of $R$, where the \emph{weight} of a lozenge tiling is the weight product of its constituent lozenges. We still call $\M(R)$ the \emph{(weighted) tiling number} of $R$. We also consider the weighted counterpart $\mathcal{P}'_{a,b,c}$  of $\mathcal{P}_{a,b,c}$,  where all the lozenges along the staircase cut are weighted by $1/2$.  M. Ciucu \cite{Ciucu1} proved the following weighted version of Theorem \ref{Proctiling}.

\begin{thm} For any non-negative integers $a,$ $b$, and $c$ with $a\leq b$
\begin{equation}
\M(\mathcal{P}'_{a,b,c})=2^{-a}\prod_{i=1}\frac{2c+b-a+i}{c+b-a+i}\prod_{i=1}^{a}\left[\prod_{j=1}^{b-a+1}\frac{c+i+j-1}{i+j-1}\prod_{j=b-a+2}^{b-a+i}\frac{2c+i+j-1}{i+j-1}\right].
\end{equation}
\end{thm}
From now on, we use the notations $\Pn_{a,b,c}$ and $\Pn'_{a,b,c}$ for the numbers of tilings of the regions $\mathcal{P}_{a,b,c}$ and $\mathcal{P}'_{a,b,c}$, respectively.

In the previous paper \cite{Lai4}, the author generalized the tiling enumerations of a halved hexagon by Proctor and Rohatgi to a halved hexagon in which an array of \emph{an arbitrary number} of adjacent triangles has been removed from the northeastern side (see Figure \ref{halfhex13} for examples of the regions). In this paper, we investigate the other case when the array of triangles has been removed from the \emph{western} side of the halved hexagon as shown in  Figures \ref{middlehole1},  \ref{middlehole2}, and  \ref{middlehole3}. Based on the positions of the array of removed triangles and the weight assignments of the lozenges along the staircase side, we have \emph{eight} families of defected halved hexagons that will be described in detail in Section \ref{Statement}. We will show that the numbers of tilings of these regions are \emph{all} given by simple product formulas.

Our explicit enumerations for halved hexagons also imply an interesting tiling formula for a symmetric hexagon with an array of triangles missing on the symmetry axis (see Theorem \ref{main9}).

\begin{figure}\centering
\setlength{\unitlength}{3947sp}%
\begingroup\makeatletter\ifx\SetFigFont\undefined%
\gdef\SetFigFont#1#2#3#4#5{%
  \reset@font\fontsize{#1}{#2pt}%
  \fontfamily{#3}\fontseries{#4}\fontshape{#5}%
  \selectfont}%
\fi\endgroup%
\resizebox{12cm}{!}{
\begin{picture}(0,0)%
\includegraphics{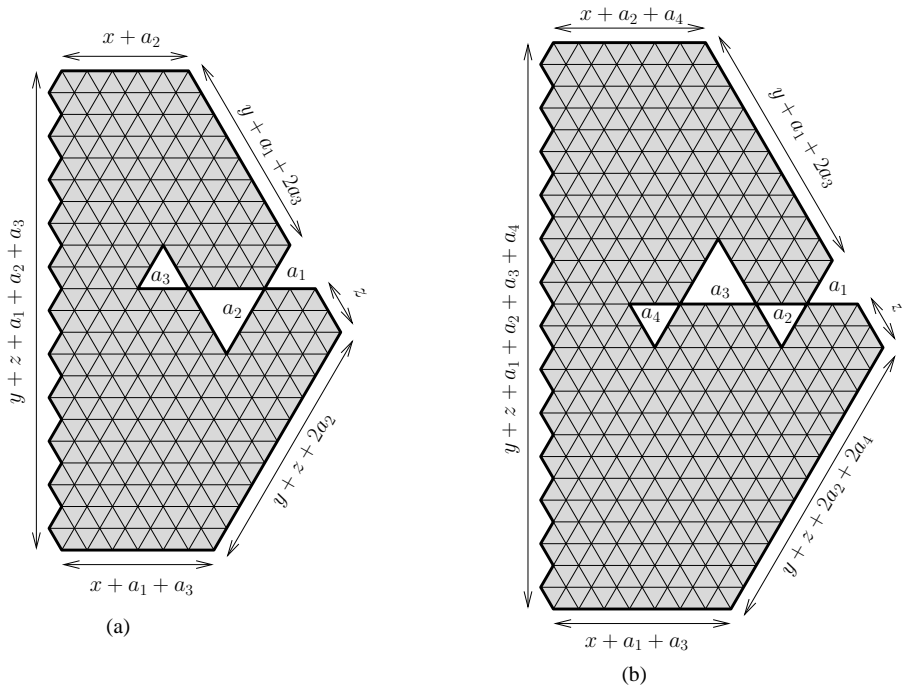}%
\end{picture}%

\begin{picture}(13959,10649)(1761,-9926)
%  METADATA <id>184</id>
\put(3256, 74){\makebox(0,0)[lb]{\smash{{\SetFigFont{20}{16.8}{\rmdefault}{\mddefault}{\itdefault}{$x+a_2$}%
}}}}
%  METADATA <id>185</id>
\put(5371,-931){\rotatebox{300.0}{\makebox(0,0)[lb]{\smash{{\SetFigFont{20}{16.8}{\rmdefault}{\mddefault}{\itdefault}{$y+a_1+2a_3$}%
}}}}}
%  METADATA <id>186</id>
\put(6091,-3541){\makebox(0,0)[lb]{\smash{{\SetFigFont{20}{16.8}{\rmdefault}{\mddefault}{\itdefault}{$a_1$}%
}}}}
%  METADATA <id>187</id>
\put(5100,-4126){\makebox(0,0)[lb]{\smash{{\SetFigFont{20}{16.8}{\rmdefault}{\mddefault}{\itdefault}{$a_2$}%
}}}}
%  METADATA <id>188</id>
\put(4000,-3601){\makebox(0,0)[lb]{\smash{{\SetFigFont{20}{16.8}{\rmdefault}{\mddefault}{\itdefault}{$a_3$}%
}}}}
%  METADATA <id>189</id>
\put(7111,-3751){\rotatebox{300.0}{\makebox(0,0)[lb]{\smash{{\SetFigFont{20}{16.8}{\rmdefault}{\mddefault}{\itdefault}{$z$}%
}}}}}
%  METADATA <id>190</id>
\put(6001,-7111){\rotatebox{60.0}{\makebox(0,0)[lb]{\smash{{\SetFigFont{20}{16.8}{\rmdefault}{\mddefault}{\itdefault}{$y+z+2a_2$}%
}}}}}
%  METADATA <id>191</id>
\put(3061,-8416){\makebox(0,0)[lb]{\smash{{\SetFigFont{20}{16.8}{\rmdefault}{\mddefault}{\itdefault}{$x+a_1+a_3$}%
}}}}
%  METADATA <id>192</id>
\put(1996,-5506){\rotatebox{90.0}{\makebox(0,0)[lb]{\smash{{\SetFigFont{20}{16.8}{\rmdefault}{\mddefault}{\itdefault}{$y+z+a_1+a_2+a_3$}%
}}}}}
%  METADATA <id>193</id>
\put(9601,-6286){\rotatebox{90.0}{\makebox(0,0)[lb]{\smash{{\SetFigFont{20}{16.8}{\rmdefault}{\mddefault}{\itdefault}{$y+z+a_1+a_2+a_3+a_4$}%
}}}}}
%  METADATA <id>194</id>
\put(10666,-9271){\makebox(0,0)[lb]{\smash{{\SetFigFont{20}{16.8}{\rmdefault}{\mddefault}{\itdefault}{$x+a_1+a_3$}%
}}}}
%  METADATA <id>195</id>
\put(13831,-8176){\rotatebox{60.0}{\makebox(0,0)[lb]{\smash{{\SetFigFont{20}{16.8}{\rmdefault}{\mddefault}{\itdefault}{$y+z+2a_2+2a_4$}%
}}}}}
%  METADATA <id>197</id>
\put(13533,-661){\rotatebox{300.0}{\makebox(0,0)[lb]{\smash{{\SetFigFont{20}{16.8}{\rmdefault}{\mddefault}{\itdefault}{$y+a_1+2a_3$}%
}}}}}
%  METADATA <id>198</id>
\put(14416,-3766){\makebox(0,0)[lb]{\smash{{\SetFigFont{20}{16.8}{\rmdefault}{\mddefault}{\itdefault}{$a_1$}%
}}}}
%  METADATA <id>199</id>
\put(13561,-4201){\makebox(0,0)[lb]{\smash{{\SetFigFont{20}{16.8}{\rmdefault}{\mddefault}{\itdefault}{$a_2$}%
}}}}
%  METADATA <id>200</id>
\put(12600,-3826){\makebox(0,0)[lb]{\smash{{\SetFigFont{20}{16.8}{\rmdefault}{\mddefault}{\itdefault}{$a_3$}%
}}}}
%  METADATA <id>201</id>
\put(11536,-4186){\makebox(0,0)[lb]{\smash{{\SetFigFont{20}{16.8}{\rmdefault}{\mddefault}{\itdefault}{$a_4$}%
}}}}
%  METADATA <id>202</id>
\put(10591,434){\makebox(0,0)[lb]{\smash{{\SetFigFont{20}{16.8}{\rmdefault}{\mddefault}{\itdefault}{$x+a_2+a_4$}%
}}}}
\end{picture}}
\caption{Halved hexagons with an array of triangles removed from the non-staircase boundary in \cite{Lai4}.} \label{halfhex13}
\end{figure}

It is worth noticing that Ciucu \cite{Ciucu2}, in his effort on finding duals of the MacMahon classical theorem on plane partitions \cite{Mac}, gave a closed form product formula for the tiling number of a hexagon (not necessarily symmetric) with an array of triangular holes  in the center. In Ciucu's result, if the array of triangular holes moves away from the center, the tiling number is \emph{not} a simple product anymore. In contrast, our result (Theorem \ref{main9} in Section \ref{Statement}) shows that we have  a nice tiling formula for \emph{any} positions of the array of triangular holes along the symmetric axis of the hexagon.

The rest of the paper is organized as follows. Due to the large number of  new regions needed to define,  we leave the precise statement of our main theorems (Theorems \ref{main1}--\ref{main9}) to Section \ref{Statement}. In Section \ref{Background}, we quote several fundamental results in the enumeration of tilings and introduce the particular version of Kuo condensation \cite{Kuo} that we will employ in our proofs. Section \ref{Mainproofs} are devoted to the proofs of our main theorems. %Finally, we investigate the symmetric tilings of a hexagon with three aligned arrays of holes in Section \ref{Threearrays}.

\section{Precise statement of the main results} \label{Statement}

\begin{figure}\centering
\setlength{\unitlength}{3947sp}%
\begingroup\makeatletter\ifx\SetFigFont\undefined%
\gdef\SetFigFont#1#2#3#4#5{%
  \reset@font\fontsize{#1}{#2pt}%
  \fontfamily{#3}\fontseries{#4}\fontshape{#5}%
  \selectfont}%
\fi\endgroup%
\resizebox{13cm}{!}{
\begin{picture}(0,0)%
\includegraphics{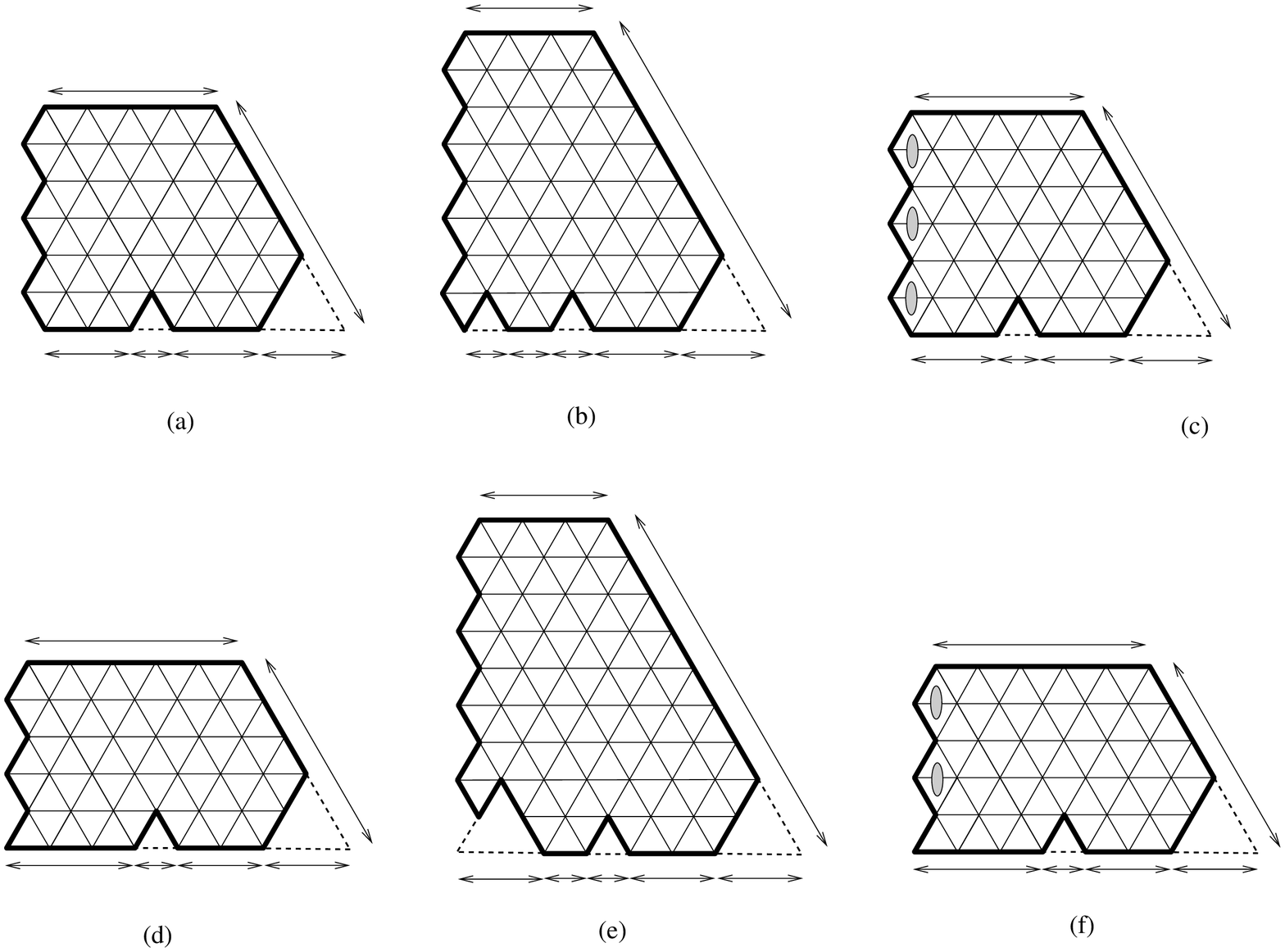}%
\end{picture}
\begin{picture}(11857,8988)(3679,-9536)
%  METADATA <id>102</id>
\put(4501,-1553){\makebox(0,0)[lb]{\smash{{\SetFigFont{16}{16.8}{\rmdefault}{\mddefault}{\updefault}{$t_1+t_3$}%
}}}}
%  METADATA <id>106</id>
\put(4148,-4338){\makebox(0,0)[lb]{\smash{{\SetFigFont{16}{16.8}{\rmdefault}{\mddefault}{\updefault}{$t_1$}%
}}}}
%  METADATA <id>107</id>
\put(4863,-4358){\makebox(0,0)[lb]{\smash{{\SetFigFont{16}{16.8}{\rmdefault}{\mddefault}{\updefault}{$t_2$}%
}}}}
%  METADATA <id>108</id>
\put(5513,-4368){\makebox(0,0)[lb]{\smash{{\SetFigFont{16}{16.8}{\rmdefault}{\mddefault}{\updefault}{$t_3$}%
}}}}
%  METADATA <id>116</id>
\put(8020,-4331){\makebox(0,0)[lb]{\smash{{\SetFigFont{16}{16.8}{\rmdefault}{\mddefault}{\updefault}{$t_2$}%
}}}}
%  METADATA <id>117</id>
\put(8387,-4331){\makebox(0,0)[lb]{\smash{{\SetFigFont{16}{16.8}{\rmdefault}{\mddefault}{\updefault}{$t_3$}%
}}}}
%  METADATA <id>185</id>
\put(4058,-9001){\makebox(0,0)[lb]{\smash{{\SetFigFont{16}{16.8}{\rmdefault}{\mddefault}{\updefault}{$t_1$}%
}}}}
%  METADATA <id>186</id>
\put(4938,-9011){\makebox(0,0)[lb]{\smash{{\SetFigFont{16}{16.8}{\rmdefault}{\mddefault}{\updefault}{$t_2$}%
}}}}
%  METADATA <id>187</id>
\put(5548,-9026){\makebox(0,0)[lb]{\smash{{\SetFigFont{16}{16.8}{\rmdefault}{\mddefault}{\updefault}{$t_3$}%
}}}}
%  METADATA <id>189</id>
\put(4583,-6566){\makebox(0,0)[lb]{\smash{{\SetFigFont{16}{16.8}{\rmdefault}{\mddefault}{\updefault}{$t_1+t_3$}%
}}}}
%  METADATA <id>279</id>
\put(6287,-9022){\makebox(0,0)[lb]{\smash{{\SetFigFont{16}{16.8}{\rmdefault}{\mddefault}{\updefault}{$t_4$}%
}}}}
%  METADATA <id>295</id>
\put(8785,-4331){\makebox(0,0)[lb]{\smash{{\SetFigFont{16}{16.8}{\rmdefault}{\mddefault}{\updefault}{$t_4$}%
}}}}
%  METADATA <id>296</id>
\put(9295,-4331){\makebox(0,0)[lb]{\smash{{\SetFigFont{16}{16.8}{\rmdefault}{\mddefault}{\updefault}{$t_5$}%
}}}}
%  METADATA <id>297</id>
\put(10045,-4331){\makebox(0,0)[lb]{\smash{{\SetFigFont{16}{16.8}{\rmdefault}{\mddefault}{\updefault}{$t_6$}%
}}}}
%  METADATA <id>303</id>
\put(8089,-9145){\makebox(0,0)[lb]{\smash{{\SetFigFont{16}{16.8}{\rmdefault}{\mddefault}{\updefault}{$t_2$}%
}}}}
%  METADATA <id>304</id>
\put(8681,-9130){\makebox(0,0)[lb]{\smash{{\SetFigFont{16}{16.8}{\rmdefault}{\mddefault}{\updefault}{$t_3$}%
}}}}
%  METADATA <id>305</id>
\put(9094,-9137){\makebox(0,0)[lb]{\smash{{\SetFigFont{16}{16.8}{\rmdefault}{\mddefault}{\updefault}{$t_4$}%
}}}}
%  METADATA <id>306</id>
\put(9611,-9130){\makebox(0,0)[lb]{\smash{{\SetFigFont{16}{16.8}{\rmdefault}{\mddefault}{\updefault}{$t_5$}%
}}}}
%  METADATA <id>307</id>
\put(10406,-9115){\makebox(0,0)[lb]{\smash{{\SetFigFont{16}{16.8}{\rmdefault}{\mddefault}{\updefault}{$t_6$}%
}}}}
%  METADATA <id>309</id>
\put(6256,-4355){\makebox(0,0)[lb]{\smash{{\SetFigFont{16}{16.8}{\rmdefault}{\mddefault}{\updefault}{$t_4$}%
}}}}
%  METADATA <id>310</id>
\put(6204,-2123){\rotatebox{300.0}{\makebox(0,0)[lb]{\smash{{\SetFigFont{16}{16.8}{\rmdefault}{\mddefault}{\updefault}{$2t_2+2t_4$}%
}}}}}
%  METADATA <id>313</id>
\put(6347,-7064){\rotatebox{300.0}{\makebox(0,0)[lb]{\smash{{\SetFigFont{16}{16.8}{\rmdefault}{\mddefault}{\updefault}{$2t_2+2t_4-1$}%
}}}}}
%  METADATA <id>314</id>
\put(8200,-798){\makebox(0,0)[lb]{\smash{{\SetFigFont{16}{16.8}{\rmdefault}{\mddefault}{\updefault}{$t_3+t_5$}%
}}}}
%  METADATA <id>316</id>
\put(8228,-5242){\makebox(0,0)[lb]{\smash{{\SetFigFont{16}{16.8}{\rmdefault}{\mddefault}{\updefault}{$t_3+t_5$}%
}}}}
%  METADATA <id>317</id>
\put(9685,-1481){\rotatebox{300.0}{\makebox(0,0)[lb]{\smash{{\SetFigFont{16}{16.8}{\rmdefault}{\mddefault}{\updefault}{$2t_2+2t_4+2t_6$}%
}}}}}
%  METADATA <id>318</id>
\put(9881,-6092){\rotatebox{300.0}{\makebox(0,0)[lb]{\smash{{\SetFigFont{16}{16.8}{\rmdefault}{\mddefault}{\updefault}{$2t_2+2t_4+2t_6-1$}%
}}}}}
%  METADATA <id>402</id>
\put(12392,-1604){\makebox(0,0)[lb]{\smash{{\SetFigFont{16}{16.8}{\rmdefault}{\mddefault}{\updefault}{$t_1+t_3$}%
}}}}
%  METADATA <id>403</id>
\put(12039,-4389){\makebox(0,0)[lb]{\smash{{\SetFigFont{16}{16.8}{\rmdefault}{\mddefault}{\updefault}{$t_1$}%
}}}}
%  METADATA <id>404</id>
\put(12754,-4409){\makebox(0,0)[lb]{\smash{{\SetFigFont{16}{16.8}{\rmdefault}{\mddefault}{\updefault}{$t_2$}%
}}}}
%  METADATA <id>405</id>
\put(13404,-4419){\makebox(0,0)[lb]{\smash{{\SetFigFont{16}{16.8}{\rmdefault}{\mddefault}{\updefault}{$t_3$}%
}}}}
%  METADATA <id>411</id>
\put(14147,-4406){\makebox(0,0)[lb]{\smash{{\SetFigFont{16}{16.8}{\rmdefault}{\mddefault}{\updefault}{$t_4$}%
}}}}
%  METADATA <id>412</id>
\put(14095,-2174){\rotatebox{300.0}{\makebox(0,0)[lb]{\smash{{\SetFigFont{16}{16.8}{\rmdefault}{\mddefault}{\updefault}{$2t_2+2t_4$}%
}}}}}
%  METADATA <id>498</id>
\put(12326,-9036){\makebox(0,0)[lb]{\smash{{\SetFigFont{16}{16.8}{\rmdefault}{\mddefault}{\updefault}{$t_1$}%
}}}}
%  METADATA <id>499</id>
\put(13206,-9046){\makebox(0,0)[lb]{\smash{{\SetFigFont{16}{16.8}{\rmdefault}{\mddefault}{\updefault}{$t_2$}%
}}}}
%  METADATA <id>500</id>
\put(13816,-9061){\makebox(0,0)[lb]{\smash{{\SetFigFont{16}{16.8}{\rmdefault}{\mddefault}{\updefault}{$t_3$}%
}}}}
%  METADATA <id>501</id>
\put(12851,-6601){\makebox(0,0)[lb]{\smash{{\SetFigFont{16}{16.8}{\rmdefault}{\mddefault}{\updefault}{$t_1+t_3$}%
}}}}
%  METADATA <id>506</id>
\put(14555,-9057){\makebox(0,0)[lb]{\smash{{\SetFigFont{16}{16.8}{\rmdefault}{\mddefault}{\updefault}{$t_4$}%
}}}}
%  METADATA <id>508</id>
\put(14615,-7099){\rotatebox{300.0}{\makebox(0,0)[lb]{\smash{{\SetFigFont{16}{16.8}{\rmdefault}{\mddefault}{\updefault}{$2t_2+2t_4-1$}%
}}}}}
\end{picture}}
\caption{ (a) The region $\mathcal{Q}(2,1,2,2)$. (b) The region $\mathcal{Q}(0,1,1,1,2,2)$.  (c) The region $\mathcal{Q}'(2,1,2,2)$. (d) The region $\mathcal{K}(3,1,2,2)$. (e) The region $\mathcal{K}(0,2,1,1,2,2)$. (f) The region $\mathcal{K}'(3,1,2,2)$. The lozenges with shaded cores are weighted by $\frac{1}{2}$. This figure first appeared in \cite{Lai4}.}\label{halfhex3c}
\end{figure}

We define the \emph{Pochhammer symbol} $(x)_n$ by
\begin{equation}
(x)_n=\begin{cases}
x(x+1)\dotsc(x+n-1) & \text{if $n>0$;}\\
1 &\text{if $n=0$;}\\
\frac{1}{(x-1)(x-2)\dotsc(x+n)}&\text{if $n<0$.}
\end{cases}
\end{equation}
We also use the `skipping' version of the Pochhammer symbol:
\begin{equation}
[x]_n=\begin{cases}
x(x+2)\dotsc(x+2(n-1)) & \text{if $n>0$;}\\
1 &\text{if $n=0$;}\\
\frac{1}{(x-2)(x-4)\dotsc(x+2n)}&\text{if $n<0$.}
\end{cases}
\end{equation}
We define two `trapezoidal' products as follows:
\begin{equation}
\T(x,n,m)=\prod_{i=0}^{m-1}(x+i)_{n-2i}
\end{equation}
and
\begin{equation}
V(x,n,m)=\prod_{i=0}^{m-1}[x+2i]_{n-2i}.
\end{equation}
Let $\textbf{a}=(a_1,a_2,\dotsc,a_n)$ be a sequence. We define several sequence operations as follows:
\begin{equation}
\Od(\textbf{a})=\sum_{\text{ $i$ odd}}a_i,
\end{equation}
\begin{equation}
\E(\textbf{a})=\sum_{\text{ $i$ even}}a_i,
\end{equation}
\begin{equation}
\s_k(\textbf{a})=\sum_{i=1}^{k}a_i,
\end{equation}
\begin{equation}
\od_k(\textbf{a})=\sum_{i\geq k}a_{2i-1},
\end{equation}
and
\begin{equation}
\e_k(\textbf{a})=\sum_{i\geq k}a_{2i}.
\end{equation}

Before going to the statement of our main results we quote here the tiling formulas of halved hexagons with triangles removed on the base (see Lemma  1.4 in  \cite{Lai4}).

Assume that $\textbf{t}=(t_1,t_2,\dots,t_{2l})$ is a sequence of non-negative integers. Consider a trapezoidal region whose northern, northeastern, and southern sides have lengths $\Od(\textbf{t}),$ $ 2\E(\textbf{t}),$ and $\E(\textbf{t})+\Od(\textbf{t})$, respectively, and whose western side follows a vertical zigzag lattice paths with $\E(\textbf{t})$ steps.  We remove the triangles of sides $t_{2i}$'s from the base of the latter region  so that the distances between two consecutive triangles are $t_{2i-1}$'s. Denote the resulting region by $\mathcal{Q}(\textbf{t})=\mathcal{Q}(t_1,t_2,\dotsc,t_{2l})$ (see the regions in Figure \ref{halfhex3c}(a)  for the case when $t_{1}>0$ and Figure \ref{halfhex3c}(b) for the  case when $t_{1}=0$).   Inspired by the weighted region $\mathcal{P}_{a,b,c}$, we consider the weighted counterpart  $\mathcal{Q}'(\textbf{t})$ of the newly defined region, where the vertical lozenges on the western side are weighted by $\frac{1}{2}$ (see Figure \ref{halfhex3c}(c); the vertical lozenges with shaded cores are weighted by $\frac{1}{2}$).

We are also interested in a variation of the above $\mathcal{Q}$-type regions as follows. Consider the  trapezoidal region whose northern, northeastern, and southern sides have lengths $\Od(\textbf{t}), 2\E(\textbf{t})-1, \E(\textbf{t})+\Od(\textbf{t}),$ respectively, and whose western side follows the vertical zigzag lattice path with $\E(\textbf{t})-\frac{1}{2}$ steps (i.e. the western side has $\E(\textbf{t})-1$ and a half `bumps'). Next, we  also remove the triangles of sides $t_{2i}$'s from the base so that the distances between two consecutive ones are $t_{2i-1}$'s. Denote by $\mathcal{K}(\textbf{t})=\mathcal{K}(t_1,t_2,\dotsc,t_{2l})$ the resulting regions (see the regions in Figure \ref{halfhex3c}(d)  for the case when $t_1>0$ and Figure \ref{halfhex3c}(e) for the  case when $t_1=0$). Similar to the case of $\mathcal{Q}'$-type regions, we also define a weighted version $\mathcal{K}'(\textbf{t})$ of the $\mathcal{K}(\textbf{t})$  by assigning to each vertical lozenge on its western side a weight $\frac{1}{2}$ (see Figure \ref{halfhex3c}(f)).

We adopt here the notations $\Q(\textbf{t})$, $\Q'(\textbf{t})$, $\K(\textbf{t})$, and  $\K'(\textbf{t})$ from \cite{Lai4} for the numbers of tilings of the regions $\mathcal{Q}(\textbf{t})$, $\mathcal{Q}'(\textbf{t})$, $\mathcal{K}(\textbf{t})$, and  $\mathcal{K}'(\textbf{t})$, respectively.

We define the  \emph{hyperfactorial} $\Hf(n)$ by
\begin{equation*}\Hf(n):=0!\cdot1!\cdot2!\dotsc(n-1)!,
\end{equation*}
 and the \emph{`skipping' hyperfactorial } $\Hf_2(n)$ by
\begin{equation*}
\Hf_2(n)=
\begin{cases}
 0!\cdot2!\cdot4!\dots(n-2)! &\text{if $n$ is even;}\\
1!\cdot2!\cdot 3!\dots(n-2)! &\text{if $n$ is odd.}
 \end{cases}
 \end{equation*}

\begin{lem}[Lemma 1.4 in \cite{Lai4}]\label{QAR}
For any sequence of non-negative integers $\textbf{t}=(t_1,t_2,\dotsc,t_{2l})$
\begin{align}\label{QARa}
\Q(\textbf{t})&=\dfrac{\prod_{i=1}^{l}\frac{(\s_{2i}(\textbf{t}))!}{(\s_{2i-1}(\textbf{t}))!}}{\Hf_2(2\E(\textbf{t})+1)} \prod_{i=1}^{l}\frac{\Hf_2(2\s_{2i}(\textbf{t})+1)\Hf(2\s_{2i-1}(\textbf{t})+2)}{\Hf_2(2\s_{2i-1}(\textbf{t})+3)}\notag\\
                  &\times \displaystyle {\prod_{\substack{1\leq i< j\leq 2l\\
                  \text{$j-i$ odd}}}}\dfrac{\Hf(\s_j(\textbf{t})-\s_{i}(\textbf{t}))}{\Hf(\s_j(\textbf{t})+\s_{i}(\textbf{t})+1)}\displaystyle {\prod_{\substack{1\leq i<j\leq 2l\\
                  \text{$j-i$ even }}}}\dfrac{\Hf(\s_j(\textbf{t})+\s_{i}(\textbf{t})+1)}{\Hf(\s_j(\textbf{t})-\s_{i}(\textbf{t}))},
                  %\times \prod_{0\leq i< j\leq 2l} \frac{\Hf(\s_{2i}(\textbf{t})+\s_{2j}(\textbf{t})+1)\Hf(\s_{2i-1}(\textbf{t})+\s_{2j-1}(\textbf{t})+1)}{\Hf(\s_{2i-1}(\textbf{t})+\s_{2j}(\textbf{t})+1)\Hf(\s_{2i}(\textbf{t})+\s_{2j-1}(\textbf{t})+1)}\notag\\
\end{align}
\begin{align}\label{QARb}
\Q'(\textbf{t})&=\dfrac{2^{-\E(\textbf{t})}}{\Hf_2(2\E(\textbf{t})+1)}  \prod_{i=1}^{l}\frac{\Hf_2(2\s_{2i}(\textbf{t})+1)\Hf(2\s_{2i-1}(\textbf{t}))}{\Hf_2(2\s_{2i-1}(\textbf{t})+1)}\notag\\
                  & \times \displaystyle {\prod_{\substack{1\leq i< j\leq 2l\\
                  \text{$j-i$ odd}}}}\dfrac{\Hf(\s_j(\textbf{t})-\s_{i}(\textbf{t}))}{\Hf(\s_j(\textbf{t})+\s_{i}(\textbf{t}))}\displaystyle {\prod_{\substack{1\leq i<j\leq 2l\\
                  \text{$j-i$ even }}}}\dfrac{\Hf(\s_j(\textbf{t})+\s_{i}(\textbf{t}))}{\Hf(\s_j(\textbf{t})-\s_{i}(\textbf{t}))},
                  %\times \prod_{0\leq i< j\leq 2l} \frac{\Hf(\s_{2i}(\textbf{t})+\s_{2j}(\textbf{t}))\Hf(\s_{2i-1}(\textbf{t})+\s_{2j-1}(\textbf{t}))}{\Hf(\s_{2i-1}(\textbf{t})+\s_{2j}(\textbf{t}))\Hf(\s_{2i}(\textbf{t})+\s_{2j-1}(\textbf{t}))}\notag\\
\end{align}
\begin{align}\label{QARc}
\K(\textbf{t})&=\dfrac{1}{\Hf_2(2\E(\textbf{t}))}   \prod_{i=1}^{l}\frac{\Hf_2(2\s_{2i}(\textbf{t}))\Hf(2\s_{2i-1}(\textbf{t})+1)}{\Hf_2(2\s_{2i-1}(\textbf{t})+2)} \notag\\
                  &\times   \displaystyle {\prod_{\substack{1\leq i< j\leq 2l\\
                  \text{$j-i$ odd}}}}\dfrac{\Hf(\s_j(\textbf{t})-\s_{i}(\textbf{t}))}{\Hf(\s_j(\textbf{t})+\s_{i}(\textbf{t}))}\displaystyle {\prod_{\substack{1\leq i<j\leq 2l\\
                  \text{$j-i$ even }}}}\dfrac{\Hf(\s_j(\textbf{t})+\s_{i}(\textbf{t}))}{\Hf(\s_j(\textbf{t})-\s_{i}(\textbf{t}))},
                  %\times \prod_{0\leq i< j\leq 2l} \frac{\Hf(\s_{2i}(\textbf{t})+\s_{2j}(\textbf{t}))\Hf(\s_{2i-1}(\textbf{t})+\s_{2j-1}(\textbf{t}))}{\Hf(\s_{2i-1}(\textbf{t})+\s_{2j}(\textbf{t}))\Hf(\s_{2i}(\textbf{t})+\s_{2j-1}(\textbf{t}))}\notag\\
\end{align}
and
\begin{align}\label{QARd}
\K'(\textbf{t})&=\dfrac{1}{\Hf_2(2\E(\textbf{t}))} \prod_{i=1}^{l}\frac{\Hf_2(2\s_{2i}(\textbf{t})-1)\Hf(2\s_{2i-1}(\textbf{t}))}{\Hf_2(2\s_{2i-1}(\textbf{t})+1)}\notag\\
                  &\times   \displaystyle {\prod_{\substack{1\leq i< j\leq 2l\\
                  \text{$j-i$ odd}}}}\dfrac{\Hf(\s_j(\textbf{t})-\s_{i}(\textbf{t}))}{\Hf(\s_j(\textbf{t})+\s_{i}(\textbf{t})-1)}\displaystyle {\prod_{\substack{1\leq i<j\leq 2l\\
                  \text{$j-i$ even }}}}\dfrac{\Hf(\s_j(\textbf{t})+\s_{i}(\textbf{t})-1)}{\Hf(\s_j(\textbf{t})-\s_{i}(\textbf{t}))}.
                  %\times \prod_{0\leq i< j\leq 2l} \frac{\Hf(\s_{2i}(\textbf{t})+\s_{2j}(\textbf{t})-1)\Hf(\s_{2i-1}(\textbf{t})+\s_{2j-1}(\textbf{t})-1)}{\Hf(\s_{2i-1}(\textbf{t})+\s_{2j}(\textbf{t})-1)\Hf(\s_{2i}(\textbf{t})+\s_{2j-1}(\textbf{t})-1)}\notag\\
\end{align}
\end{lem}

\begin{figure}\centering
\setlength{\unitlength}{3947sp}%
\begingroup\makeatletter\ifx\SetFigFont\undefined%
\gdef\SetFigFont#1#2#3#4#5{%
  \reset@font\fontsize{#1}{#2pt}%
  \fontfamily{#3}\fontseries{#4}\fontshape{#5}%
  \selectfont}%
\fi\endgroup%
\resizebox{13cm}{!}{
\begin{picture}(0,0)%
\includegraphics{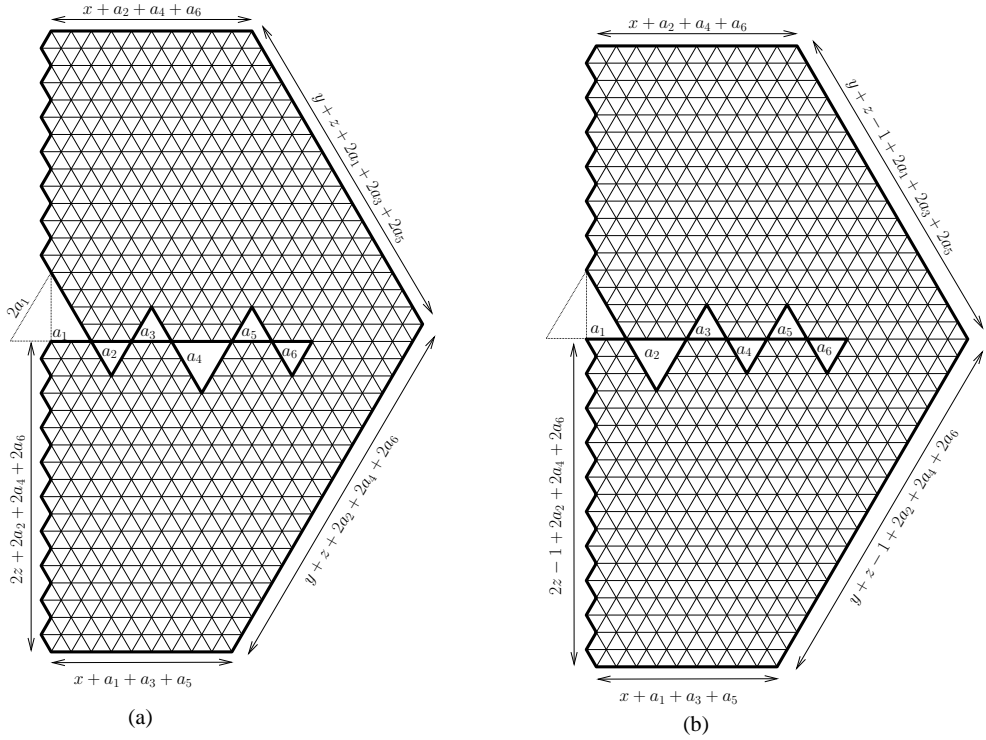}%
\end{picture}%
%
%  Created by WinFIG version 6.2
%  METADATA <version>1.0</version>
%

\begin{picture}(19122,14547)(2375,-13913)
%  METADATA <id>157</id>
\put(4363,-6421){\makebox(0,0)[lb]{\smash{{\SetFigFont{20}{24.0}{\rmdefault}{\mddefault}{\updefault}{$a_2$}%
}}}}
%  METADATA <id>158</id>
\put(5113,-6046){\makebox(0,0)[lb]{\smash{{\SetFigFont{20}{24.0}{\rmdefault}{\mddefault}{\updefault}{$a_3$}%
}}}}
%  METADATA <id>159</id>
\put(6013,-6519){\makebox(0,0)[lb]{\smash{{\SetFigFont{20}{24.0}{\rmdefault}{\mddefault}{\updefault}{$a_4$}%
}}}}
%  METADATA <id>160</id>
\put(7078,-6046){\makebox(0,0)[lb]{\smash{{\SetFigFont{20}{24.0}{\rmdefault}{\mddefault}{\updefault}{$a_5$}%
}}}}
%  METADATA <id>161</id>
\put(7858,-6421){\makebox(0,0)[lb]{\smash{{\SetFigFont{20}{24.0}{\rmdefault}{\mddefault}{\updefault}{$a_6$}%
}}}}
%  METADATA <id>162</id>
\put(3999,269){\makebox(0,0)[lb]{\smash{{\SetFigFont{20}{24.0}{\rmdefault}{\mddefault}{\updefault}{$x+a_2+a_4+a_6$}%
}}}}
%  METADATA <id>163</id>
\put(8484,-1336){\rotatebox{300.0}{\makebox(0,0)[lb]{\smash{{\SetFigFont{20}{24.0}{\rmdefault}{\mddefault}{\updefault}{$y+z+2a_1+2a_3+2a_5$}%
}}}}}
%  METADATA <id>164</id>
\put(8431,-10914){\rotatebox{60.0}{\makebox(0,0)[lb]{\smash{{\SetFigFont{20}{24.0}{\rmdefault}{\mddefault}{\updefault}{$y+z+2a_2+2a_4+2a_6$}%
}}}}}
%  METADATA <id>165</id>
\put(3819,-12856){\makebox(0,0)[lb]{\smash{{\SetFigFont{20}{24.0}{\rmdefault}{\mddefault}{\updefault}{$x+a_1+a_3+a_5$}%
}}}}
%  METADATA <id>169</id>
\put(2694,-5709){\rotatebox{60.0}{\makebox(0,0)[lb]{\smash{{\SetFigFont{20}{24.0}{\rmdefault}{\mddefault}{\updefault}{$2a_1$}%
}}}}}
%  METADATA <id>170</id>
\put(2821,-10891){\rotatebox{90.0}{\makebox(0,0)[lb]{\smash{{\SetFigFont{20}{24.0}{\rmdefault}{\mddefault}{\updefault}{$2z+2a_2+2a_4+2a_6$}%
}}}}}
%  METADATA <id>342</id>
\put(3406,-6061){\makebox(0,0)[lb]{\smash{{\SetFigFont{20}{24.0}{\rmdefault}{\mddefault}{\updefault}{$a_1$}%
}}}}
%  METADATA <id>350</id>
\put(13076,-5679){\rotatebox{60.0}{\makebox(0,0)[lb]{\smash{{\SetFigFont{20}{24.0}{\rmdefault}{\mddefault}{${0,0,0}2a_1$}%
}}}}}
%  METADATA <id>354</id>
\put(13846,-5979){\makebox(0,0)[lb]{\smash{{\SetFigFont{20}{24.0}{\rmdefault}{\mddefault}{\updefault}{$a_1$}%
}}}}
%  METADATA <id>356</id>
\put(14896,-6444){\makebox(0,0)[lb]{\smash{{\SetFigFont{20}{24.0}{\rmdefault}{\mddefault}{\updefault}{$a_2$}%
}}}}
%  METADATA <id>358</id>
\put(15886,-5994){\makebox(0,0)[lb]{\smash{{\SetFigFont{20}{24.0}{\rmdefault}{\mddefault}{\updefault}{$a_3$}%
}}}}
%  METADATA <id>360</id>
\put(16681,-6384){\makebox(0,0)[lb]{\smash{{\SetFigFont{20}{24.0}{\rmdefault}{\mddefault}{\updefault}{$a_4$}%
}}}}
%  METADATA <id>362</id>
\put(17476,-5994){\makebox(0,0)[lb]{\smash{{\SetFigFont{20}{24.0}{\rmdefault}{\mddefault}{\updefault}{$a_5$}%
}}}}
%  METADATA <id>364</id>
\put(18241,-6354){\makebox(0,0)[lb]{\smash{{\SetFigFont{20}{24.0}{\rmdefault}{\mddefault}{\updefault}{$a_6$}%
}}}}
%  METADATA <id>368</id>
\put(14409,-13179){\makebox(0,0)[lb]{\smash{{\SetFigFont{20}{24.0}{\rmdefault}{\mddefault}{\updefault}{$x+a_1+a_3+a_5$}%
}}}}
%  METADATA <id>372</id>
\put(14588, -9){\makebox(0,0)[lb]{\smash{{\SetFigFont{20}{24.0}{\rmdefault}{\mddefault}{\updefault}{$x+a_2+a_4+a_6$}%
}}}}
%  METADATA <id>379</id>
\put(18841,-1066){\rotatebox{300.0}{\makebox(0,0)[lb]{\smash{{\SetFigFont{20}{24.0}{\rmdefault}{\mddefault}{\updefault}{$y+z-1+2a_1+2a_3+2a_5$}%
}}}}}
%  METADATA <id>380</id>
\put(19006,-11296){\rotatebox{60.0}{\makebox(0,0)[lb]{\smash{{\SetFigFont{20}{24.0}{\rmdefault}{\mddefault}{\updefault}{$y+z-1+2a_2+2a_4+2a_6$}%
}}}}}
%  METADATA <id>381</id>
\put(13276,-11101){\rotatebox{90.0}{\makebox(0,0)[lb]{\smash{{\SetFigFont{20}{24.0}{\rmdefault}{\mddefault}{\updefault}{$2z-1+2a_2+2a_4+2a_6$}%
}}}}}
\end{picture}}
\caption{The halved hexagons with an array of holes on the staircase boundary: (a) $H^{(1)}_{3,3,2}(2,2,2,3,2,2)$ and (b) $H^{(2)}_{3,3,3}(2,3,2,2,2,2)$.}\label{middlehole1}
\end{figure}

Assume that $x,y,z$ are non-negative integers and that $\textbf{a}=(a_1,a_2,\dotsc,a_k)$ is a sequence of $k$ non-negative integers ($\textbf{a}$ may be an empty sequence). We consider a pentagonal region whose northern, northeastern, southeastern, and south sides have respectively lengths $x+\E(\textbf{a}),y+z+2\Od(\textbf{a}),y+z+2\E(\textbf{a}),x+\Od(\textbf{a})$, and the western side follows a vertical zigzag lattice path with $y+z+\E(\textbf{a})+\Od(\textbf{a})$ steps. Next, we remove an array of alternating up-pointing and down-pointing triangles at level $2z+2\E(\textbf{a})$ from the bottom. The array of triangles starts with an up-pointing half triangle of side $2a_1$, then the next triangles have sides $a_2,a_3,\dotsc,a_k$, ordered from left to right (see Figure \ref{middlehole1}(a) for the case $x=3,y=3,z=2,k=6,a_1=2,a_2=2,a_3=2,a_4=3,a_5=2,a_6=2$). Denote by $H^{(1)}_{x,y,z}(\textbf{a})=H^{(1)}_{x,y,z}(a_1,a_2,\dotsc,a_k)$ the resulting region.

Next, we investigate a variation $H^{(2)}_{x,y,z}(\textbf{a})$ of $H^{(1)}_{x,y,z}(\textbf{a})$ as follows. We start with a pentagonal region of side-lengths $x+\E(\textbf{a}),y+z-1+2\Od(\textbf{a}),y+z-1+2\E(\textbf{a}),x+\Od(\textbf{a}),y+z-1\E(\textbf{a})+\Od(\textbf{a})$. The array of triangles is now removed at an odd level, $2z+2\E(\textbf{a})-1$, from the bottom (instead of the distance $2z+2\E(\textbf{a})$ in the case of  the $H^{(1)}$-regions). Figure \ref{middlehole1}(b) illustrates the region for the case $x=3,y=3,z=3,k=6,a_1=2,a_2=3,a_3=2,a_4=2,a_5=2,a_6=2$. Lozenge tilings of the two `defected' halved hexagons $H^{(1)}_{x,y,z}(\textbf{a})$ and $H^{(2)}_{x,y,z}(\textbf{a})$ are always enumerated by simple product formulas.

\begin{thm}\label{main1}
Assume that $x,y,z,k$ are non-negative integers and  $\textbf{a}=(a_1,a_2,\dotsc,a_k)$ is a sequence of $k$ non-negative integers. The number of tilings of the defected halved hexagon  $H^{(1)}_{x,y,z}(a_1,a_2,\dotsc,a_k)$ is given by
\begin{align}\label{main1eq1}
\M&(H^{(1)}_{x,y,z}(a,b))=\frac{\Pn_{y,y+2a,b}\Pn_{z+b,z+b,a}\Q(a,b,x,y+z)}{\Pn_{y+z+b,y+z+b,a}}\notag\\
&\frac{\T(x+b+1,y+a-1,a)\T(x+z+a+b+2,y+a-1,a)\T(2a+b+2,y+b-1,b)\T(z+1,y+b-1,b)}{\T(b+1,y+a-1,a)\T(z+a+b+2,y+a-1,a)\T(x+2a+b+2,y+b-1,b)\T(x+z+1,y+b-1,b)},
\end{align}
for $k\geq 2$
\begin{align}\label{main1eq2}
\M&(H^{(1)}_{x,y,z}(a_1,a_2,\dotsc,a_{2k}))=\M(H^{(1)}_{x,y,z}(\Od(\textbf{a}),\E(\textbf{a})))\frac{\Q(0,a_1,\dotsc,a_{2k},y)\Q(a_1,\dotsc,a_{2k}+z)}{\Pn_{y,y+2\Od(\textbf{a}),\E(\textbf{a})}\Pn_{z+\E(\textbf{a}),z+\E(\textbf{a}),\Od(\textbf{a})}}\notag\\
&\times \prod_{i=2}^{k}\frac{\T(x+z+\e_{i}(\textbf{a})+1,a_{2i-2}+\od_{i}(\textbf{a})-1,\od_{i}(\textbf{a}))}{\T(x+y+\e_{i}(\textbf{a})+1,a_{2i-2}+\od_{i}(\textbf{a})-1,\od_{i}(\textbf{a}))}\frac{\T(y+\e_{i}(\textbf{a})+1,a_{2i-2}+\od_{i}(\textbf{a})-1,\od_{i}(\textbf{a}))}{\T(z+\e_{i}(\textbf{a})+1,a_{2i-2}+\od_{i}(\textbf{a})-1,\od_{i}(\textbf{a}))}\notag\\
&\times \prod_{i=2}^{k}\frac{\T(x+y+\s_{2i-1}(\textbf{a})+\s_{2k}(\textbf{a})+2,a_{2i-2}+\od_{i}(\textbf{a})-1,\od_{i}(\textbf{a}))}{\T(x+z+\s_{2i-1}(\textbf{a})+\s_{2k}(\textbf{a})+2,a_{2i-2}+\od_{i}(\textbf{a})-1,\od_{i}(\textbf{a}))}\notag\\
&\times \prod_{i=2}^{k}\frac{\T(z+\s_{2i-1}(\textbf{a})+\s_{2k}(\textbf{a})+2,a_{2i-2}+\od_{i}(\textbf{a})-1,\od_{i}(\textbf{a}))}{\T(y+\s_{2i-1}(\textbf{a})+\s_{2k}(\textbf{a})+2,a_{2i-2}+\od_{i}(\textbf{a})-1,\od_{i}(\textbf{a}))},
\end{align}
and
\begin{equation}\label{main1eqx}
\M(H^{(1)}_{x,y,z}(a_1,a_2,\dotsc,a_{2k-1}))=\M(H^{(1)}_{x,y,z}(a_1,a_2,\dotsc,a_{2k-1},0)).
\end{equation}
\end{thm}

We can view a halved hexagon with an odd number of triangular holes as a special case of the one with an even-number of holes, when the rightmost hole has size $0$, as illustrated in the equation (\ref{main1eqx}). For the sake of simplicity,  in our next theorems (Theorems \ref{main2}--\ref{main8}), we only show the tiling formulas of regions with an even number of holes.

\begin{thm}\label{main2}
Assume that $x,y,z,k$ are non-negative integers and  $\textbf{a}=(a_1,a_2,\dotsc,a_k)$ is a sequence of $k$ non-negative integers. The number of tilings of the defected halved hexagon  $H^{(2)}_{x,y,z}(a_1,a_2,\dotsc,a_k)$ is given by
\begin{align}
\M&(H^{(2)}_{x,y,z}(a,b))=\frac{\Pn_{y,y+2a,b}\Pn_{z+b-1,z+b-1,a}\K(a,b,x,y+z)}{\Pn_{y+z+b-1,y+z+b-1,a}}\notag\\
&\frac{\T(x+b+1,y+a-1,a)\T(x+z+a+b+1,y+a-1,a)\T(2a+b+1,y+b-1,b)\T(z+1,y+b-1,b)}{\T(b+1,y+a-1,a)\T(z+a+b+1,y+a-1,a)\T(x+2a+b+1,y+b-1,b)\T(x+z+1,y+b-1,b)},
\end{align}
for $k\geq 2$
\begin{align}
\M&(H^{(2)}_{x,y,z}(a_1,a_2,\dotsc,a_{2k}))=\M(H^{(2)}_{x,y,z}(\Od(\textbf{a}),\E(\textbf{a})))\frac{\K(0,a_1,\dotsc,a_{2k},y)\K(a_1,\dotsc,a_{2k}+z)}{\Pn_{y,y+2\Od(\textbf{a}),\E(\textbf{a})}\Pn_{z+\E(\textbf{a})-1,z+\E(\textbf{a})-1,\Od(\textbf{a})}}\notag\\
&\times \prod_{i=2}^{k}\frac{\T(x+z+\e_{i}(\textbf{a})+1,a_{2i-2}+\od_{i}(\textbf{a})-1,\od_{i}(\textbf{a}))}{\T(x+y+\e_{i}(\textbf{a})+1,a_{2i-2}+\od_{i}(\textbf{a})-1,\od_{i}(\textbf{a}))}\frac{\T(y+\e_{i}(\textbf{a})+1,a_{2i-2}+\od_{i}(\textbf{a})-1,\od_{i}(\textbf{a}))}{\T(z+\e_{i}(\textbf{a})+1,a_{2i-2}+\od_{i}(\textbf{a})-1,\od_{i}(\textbf{a}))}\notag\\
&\times \prod_{i=2}^{k}\frac{\T(x+y+\s_{2i-1}(\textbf{a})+\s_{2k}(\textbf{a})+1,a_{2i-2}+\od_{i}(\textbf{a})-1,\od_{i}(\textbf{a}))}{\T(x+z+\s_{2i-1}(\textbf{a})+\s_{2k}(\textbf{a})+1,a_{2i-2}+\od_{i}(\textbf{a})-1,\od_{i}(\textbf{a}))}\notag\\
&\times \prod_{i=2}^{k}\frac{\T(z+\s_{2i-1}(\textbf{a})+\s_{2k}(\textbf{a})+1,a_{2i-2}+\od_{i}(\textbf{a})-1,\od_{i}(\textbf{a}))}{\T(y+\s_{2i-1}(\textbf{a})+\s_{2k}(\textbf{a})+1,a_{2i-2}+\od_{i}(\textbf{a})-1,\od_{i}(\textbf{a}))}.
\end{align}
%and
%\begin{equation}
%\M(H^{(2)}_{x,y,z}(a_1,a_2,\dotsc,a_{2k-1}))=\M(H^{(2)}_{x,y,z}(a_1,a_2,\dotsc,a_{2k-1},0)).
%\end{equation}
\end{thm}

\begin{figure}\centering
\setlength{\unitlength}{3947sp}%
\begingroup\makeatletter\ifx\SetFigFont\undefined%
\gdef\SetFigFont#1#2#3#4#5{%
  \reset@font\fontsize{#1}{#2pt}%
  \fontfamily{#3}\fontseries{#4}\fontshape{#5}%
  \selectfont}%
\fi\endgroup%
\resizebox{13cm}{!}{
\begin{picture}(0,0)%
\includegraphics{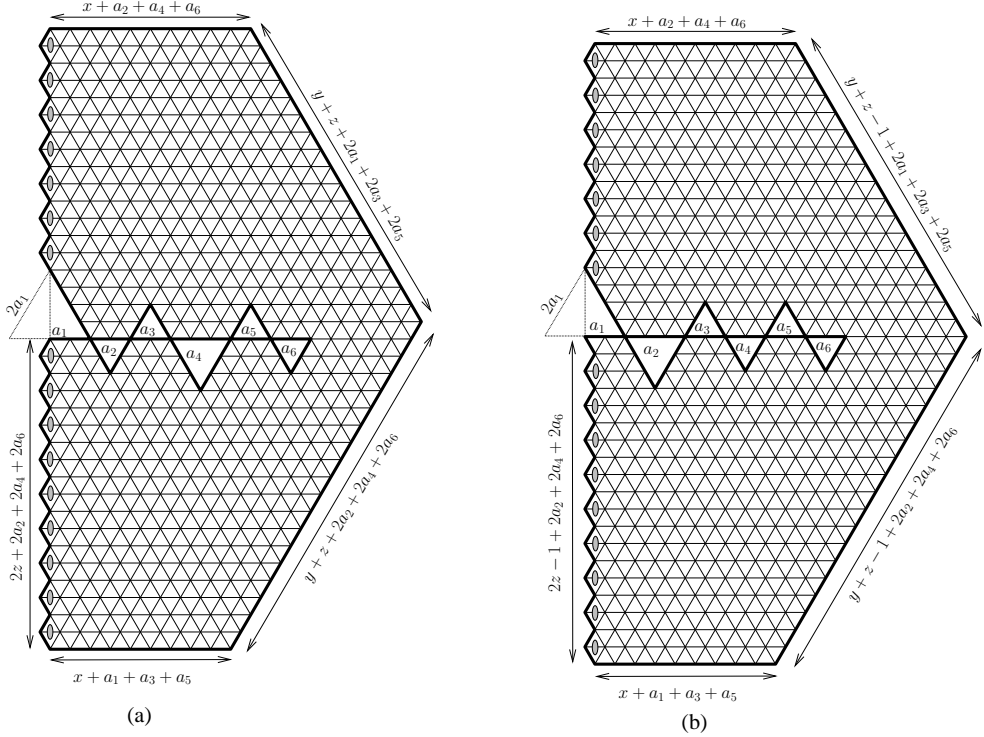}%
\end{picture}%
%
%  Created by WinFIG version 6.2
%  METADATA <version>1.0</version>
%

\begin{picture}(19133,14562)(1001,-14641)
%  METADATA <id>494</id>
\put(3000,-7147){\makebox(0,0)[lb]{\smash{{\SetFigFont{20}{24.0}{\rmdefault}{\mddefault}{\updefault}{$a_2$}%
}}}}
%  METADATA <id>495</id>
\put(3750,-6772){\makebox(0,0)[lb]{\smash{{\SetFigFont{20}{24.0}{\rmdefault}{\mddefault}{\updefault}{$a_3$}%
}}}}
%  METADATA <id>496</id>
\put(4650,-7245){\makebox(0,0)[lb]{\smash{{\SetFigFont{20}{24.0}{\rmdefault}{\mddefault}{\updefault}{$a_4$}%
}}}}
%  METADATA <id>497</id>
\put(5715,-6772){\makebox(0,0)[lb]{\smash{{\SetFigFont{20}{24.0}{\rmdefault}{\mddefault}{\updefault}{$a_5$}%
}}}}
%  METADATA <id>498</id>
\put(6495,-7147){\makebox(0,0)[lb]{\smash{{\SetFigFont{20}{24.0}{\rmdefault}{\mddefault}{\updefault}{$a_6$}%
}}}}
%  METADATA <id>499</id>
\put(2636,-457){\makebox(0,0)[lb]{\smash{{\SetFigFont{20}{24.0}{\rmdefault}{\mddefault}{\updefault}{$x+a_2+a_4+a_6$}%
}}}}
%  METADATA <id>500</id>
\put(7121,-2062){\rotatebox{300.0}{\makebox(0,0)[lb]{\smash{{\SetFigFont{20}{24.0}{\rmdefault}{\mddefault}{\updefault}{$y+z+2a_1+2a_3+2a_5$}%
}}}}}
%  METADATA <id>501</id>
\put(7068,-11640){\rotatebox{60.0}{\makebox(0,0)[lb]{\smash{{\SetFigFont{20}{24.0}{\rmdefault}{\mddefault}{\updefault}{$y+z+2a_2+2a_4+2a_6$}%
}}}}}
%  METADATA <id>502</id>
\put(2456,-13582){\makebox(0,0)[lb]{\smash{{\SetFigFont{20}{24.0}{\rmdefault}{\mddefault}{\updefault}{$x+a_1+a_3+a_5$}%
}}}}
%  METADATA <id>504</id>
\put(1331,-6435){\rotatebox{60.0}{\makebox(0,0)[lb]{\smash{{\SetFigFont{20}{24.0}{\rmdefault}{\mddefault}{\updefault}{$2a_1$}%
}}}}}
%  METADATA <id>505</id>
\put(1458,-11617){\rotatebox{90.0}{\makebox(0,0)[lb]{\smash{{\SetFigFont{20}{24.0}{\rmdefault}{\mddefault}{\updefault}{$2z+2a_2+2a_4+2a_6$}%
}}}}}
%  METADATA <id>672</id>
\put(2043,-6787){\makebox(0,0)[lb]{\smash{{\SetFigFont{20}{24.0}{\rmdefault}{\mddefault}{\updefault}{$a_1$}%
}}}}
%  METADATA <id>676</id>
\put(11713,-6405){\rotatebox{60.0}{\makebox(0,0)[lb]{\smash{{\SetFigFont{20}{24.0}{\rmdefault}{\mddefault}{\updefault}{$2a_1$}%
}}}}}
%  METADATA <id>677</id>
\put(12483,-6705){\makebox(0,0)[lb]{\smash{{\SetFigFont{20}{24.0}{\rmdefault}{\mddefault}{\updefault}{$a_1$}%
}}}}
%  METADATA <id>678</id>
\put(13533,-7170){\makebox(0,0)[lb]{\smash{{\SetFigFont{20}{24.0}{\rmdefault}{\mddefault}{\updefault}{$a_2$}%
}}}}
%  METADATA <id>679</id>
\put(14523,-6720){\makebox(0,0)[lb]{\smash{{\SetFigFont{20}{24.0}{\rmdefault}{\mddefault}{\updefault}{$a_3$}%
}}}}
%  METADATA <id>680</id>
\put(15318,-7110){\makebox(0,0)[lb]{\smash{{\SetFigFont{20}{24.0}{\rmdefault}{\mddefault}{\updefault}{$a_4$}%
}}}}
%  METADATA <id>681</id>
\put(16113,-6720){\makebox(0,0)[lb]{\smash{{\SetFigFont{20}{24.0}{\rmdefault}{\mddefault}{\updefault}{$a_5$}%
}}}}
%  METADATA <id>682</id>
\put(16878,-7080){\makebox(0,0)[lb]{\smash{{\SetFigFont{20}{24.0}{\rmdefault}{\mddefault}{\updefault}{$a_6$}%
}}}}
%  METADATA <id>684</id>
\put(13046,-13905){\makebox(0,0)[lb]{\smash{{\SetFigFont{20}{24.0}{\rmdefault}{\mddefault}{\updefault}{$x+a_1+a_3+a_5$}%
}}}}
%  METADATA <id>686</id>
\put(13225,-735){\makebox(0,0)[lb]{\smash{{\SetFigFont{20}{24.0}{\rmdefault}{\mddefault}{\updefault}{$x+a_2+a_4+a_6$}%
}}}}
%  METADATA <id>690</id>
\put(17478,-1792){\rotatebox{300.0}{\makebox(0,0)[lb]{\smash{{\SetFigFont{20}{24.0}{\rmdefault}{\mddefault}{\updefault}{$y+z-1+2a_1+2a_3+2a_5$}%
}}}}}
%  METADATA <id>691</id>
\put(17643,-12022){\rotatebox{60.0}{\makebox(0,0)[lb]{\smash{{\SetFigFont{20}{24.0}{\rmdefault}{\mddefault}{\updefault}{$y+z-1+2a_2+2a_4+2a_6$}%
}}}}}
%  METADATA <id>692</id>
\put(11913,-11827){\rotatebox{90.0}{\makebox(0,0)[lb]{\smash{{\SetFigFont{20}{24.0}{\rmdefault}{\mddefault}{\updefault}{$2z-1+2a_2+2a_4+2a_6$}%
}}}}}
\end{picture}}
\caption{The weighted halved hexagons with an array of holes on the staircase boundary: (a) $H^{(3)}_{3,3,2}(2,2,2,3,2,2)$  and (b) $H^{(4)}_{3,3,3}(2,3,2,2,2,2)$. The lozenges with shaded cores are weighted by $1/2$.}\label{middlehole2}
\end{figure}

Similar to the case of original halved hexagons $\mathcal{P}_{a,b,c}$, we are interested in the weighted versions $H^{(3)}_{x,y,z}(\textbf{a})$ and $H^{(4)}_{x,y,z}(\textbf{a})$ of the above regions $H^{(1)}_{x,y,z}(\textbf{a})$ and $H^{(2)}_{x,y,z}(\textbf{a})$, respectively, where the vertical lozenges along the western side are weighted by $1/2$ (see Figures \ref{middlehole2}(a) and (b), respectively; the lozenges with shaded cores are weighted by $1/2$). The tiling numbers of these weighted regions are also given by simple product formulas.

\begin{thm}\label{main3}
Assume that $x,y,z,k$ are non-negative integers and  $\textbf{a}=(a_1,a_2,\dotsc,a_k)$ is a sequence of $k$ non-negative integers. The number of tilings of the defected halved hexagon  $H^{(3)}_{x,y,z}(a_1,a_2,\dotsc,a_k)$ is given by
\begin{align}
\M&(H^{(3)}_{x,y,z}(a,b))=\frac{\Pn'_{y,y+2a,b}\Pn'_{z+b,z+b,a}\Q'(0,a,b,x,y+z)}{\Pn'_{y+z+b,y+z+b,a}}\notag\\
&\frac{\T(x+b+1,y+a-1,a)\T(x+z+a+b+1,y+a-1,a)\T(2a+b+1,y+b-1,b)\T(z+1,y+b-1,b)}{\T(b+1,y+a-1,a)\T(z+a+b+1,y+a-1,a)\T(x+2a+b+1,y+b-1,b)\T(x+z+1,y+b-1,b)},
\end{align}
and for $k\geq 2$
\begin{align}
\M&(H^{(3)}_{x,y,z}(a_1,a_2,\dotsc,a_{2k}))=2^{a_1}\M(H^{(3)}_{x,y,z}(\Od(\textbf{a}),\E(\textbf{a})))\frac{\Q'(a_1,a_2,\dotsc,a_{2k},z)\Q'(0,a_1,\dotsc,a_{2k}+z)}{\Pn'_{y,y+2\Od(\textbf{a}),\E(\textbf{a})}\Pn'_{z+\E(\textbf{a}),z+\E(\textbf{a}),\Od(\textbf{a})}}\notag\\
&\times \prod_{i=2}^{k}\frac{\T(x+z+\e_{i}(\textbf{a})+1,a_{2i-2}+\od_{i}(\textbf{a})-1,\od_{i}(\textbf{a}))}{\T(x+y+\e_{i}(\textbf{a})+1,a_{2i-2}+\od_{i}(\textbf{a})-1,\od_{i}(\textbf{a}))}\frac{\T(y+\e_{i}(\textbf{a})+1,a_{2i-2}+\od_{i}(\textbf{a})-1,\od_{i}(\textbf{a}))}{\T(z+\e_{i}(\textbf{a})+1,a_{2i-2}+\od_{i}(\textbf{a})-1,\od_{i}(\textbf{a}))}\notag\\
&\times \prod_{i=2}^{k}\frac{\T(x+y+\s_{2i-1}(\textbf{a})+\s_{2k}(\textbf{a})+1,a_{2i-2}+\od_{i}(\textbf{a})-1,\od_{i}(\textbf{a}))}{\T(x+z+\s_{2i-1}(\textbf{a})+\s_{2k}(\textbf{a})+1,a_{2i-2}+\od_{i}(\textbf{a})-1,\od_{i}(\textbf{a}))}\notag\\
&\times \prod_{i=2}^{k}\frac{\T(z+\s_{2i-1}(\textbf{a})+\s_{2k}(\textbf{a})+1,a_{2i-2}+\od_{i}(\textbf{a})-1,\od_{i}(\textbf{a}))}{\T(y+\s_{2i-1}(\textbf{a})+\s_{2k}(\textbf{a})+1,a_{2i-2}+\od_{i}(\textbf{a})-1,\od_{i}(\textbf{a}))}.
\end{align}
%\begin{equation}
%\M(H^{(3)}_{x,y,z}(a_1,a_2,\dotsc,a_{2k-1}))=\M(H^{(3)}_{x,y,z}(a_1,a_2,\dotsc,a_{2k-1},0)).
%\end{equation}
\end{thm}

\begin{thm}\label{main4}
Assume that $x,y,z,k$ are non-negative integers and  $\textbf{a}=(a_1,a_2,\dotsc,a_k)$ is a sequence of $k$ non-negative integers. The number of tilings of the defected halved hexagon  $H^{(4)}_{x,y,z}(a_1,a_2,\dotsc,a_k)$ is given by
\begin{align}
\M&(H^{(4)}_{x,y,z}(a,b))=\frac{\Pn'_{y,y+2a,b}\Pn'_{z+b-1,z+b-1,a}\K'(a,b,x,y+z)}{\Pn'_{y+z+b-1,y+z+b-1,a}}\notag\\
&\frac{\T(x+b+1,y+a-1,a)\T(x+z+a+b,y+a-1,a)\T(2a+b,y+b-1,b)\T(z+1,y+b-1,b)}{\T(b+1,y+a-1,a)\T(z+a+b,y+a-1,a)\T(x+2a+b,y+b-1,b)\T(x+z+1,y+b-1,b)},
\end{align}
and for $k\geq 2$
\begin{align}
\M&(H^{(4)}_{x,y,z}(a_1,a_2,\dotsc,a_{2k}))=2^{a_1-1}\M(H^{(4)}_{x,y,z}(\Od(\textbf{a}),\E(\textbf{a})))\frac{\K'(0,a_1,\dotsc,a_{2k},y)\K'(a_1,\dotsc,a_{2k}+z)}{\Pn'_{y,y+2\Od(\textbf{a}),\E(\textbf{a})}\Pn'_{z+\E(\textbf{a})-1,z+\E(\textbf{a})-1,\Od(\textbf{a})}}\notag\\
&\times \prod_{i=2}^{k}\frac{\T(x+z+\e_{i}(\textbf{a})+1,a_{2i-2}+\od_{i}(\textbf{a})-1,\od_{i}(\textbf{a}))}{\T(x+y+\e_{i}(\textbf{a})+1,a_{2i-2}+\od_{i}(\textbf{a})-1,\od_{i}(\textbf{a}))}\frac{\T(y+\e_{i}(\textbf{a})+1,a_{2i-2}+\od_{i}(\textbf{a})-1,\od_{i}(\textbf{a}))}{\T(z+\e_{i}(\textbf{a})+1,a_{2i-2}+\od_{i}(\textbf{a})-1,\od_{i}(\textbf{a}))}\notag\\
&\times \prod_{i=2}^{k}\frac{\T(x+y+\s_{2i-1}(\textbf{a})+\s_{2k}(\textbf{a}),a_{2i-2}+\od_{i}(\textbf{a})-1,\od_{i}(\textbf{a}))}{\T(x+z+\s_{2i-1}(\textbf{a})+\s_{2k}(\textbf{a}),a_{2i-2}+\od_{i}(\textbf{a})-1,\od_{i}(\textbf{a}))}\notag\\
&\times \prod_{i=2}^{k}\frac{\T(z+\s_{2i-1}(\textbf{a})+\s_{2k}(\textbf{a}),a_{2i-2}+\od_{i}(\textbf{a})-1,\od_{i}(\textbf{a}))}{\T(y+\s_{2i-1}(\textbf{a})+\s_{2k}(\textbf{a}),a_{2i-2}+\od_{i}(\textbf{a})-1,\od_{i}(\textbf{a}))}.
\end{align}
%and
%\begin{equation}
%\M(H^{(4)}_{x,y,z}(a_1,a_2,\dotsc,a_{2k-1}))=\M(H^{(4)}_{x,y,z}(a_1,a_2,\dotsc,a_{2k-1},0)).
%\end{equation}
\end{thm}

Finally, we focus on four new families of defected halved hexagons with `mixed' western boundary. In these families only a half of the lozenges along the western boundary (the ones above the array of holes or the ones below the array of holes) are weighted by $1/2$.

\begin{figure}\centering
\setlength{\unitlength}{3947sp}%
\begingroup\makeatletter\ifx\SetFigFont\undefined%
\gdef\SetFigFont#1#2#3#4#5{%
  \reset@font\fontsize{#1}{#2pt}%
  \fontfamily{#3}\fontseries{#4}\fontshape{#5}%
  \selectfont}%
\fi\endgroup%
\resizebox{13cm}{!}{
\begin{picture}(0,0)%
\includegraphics{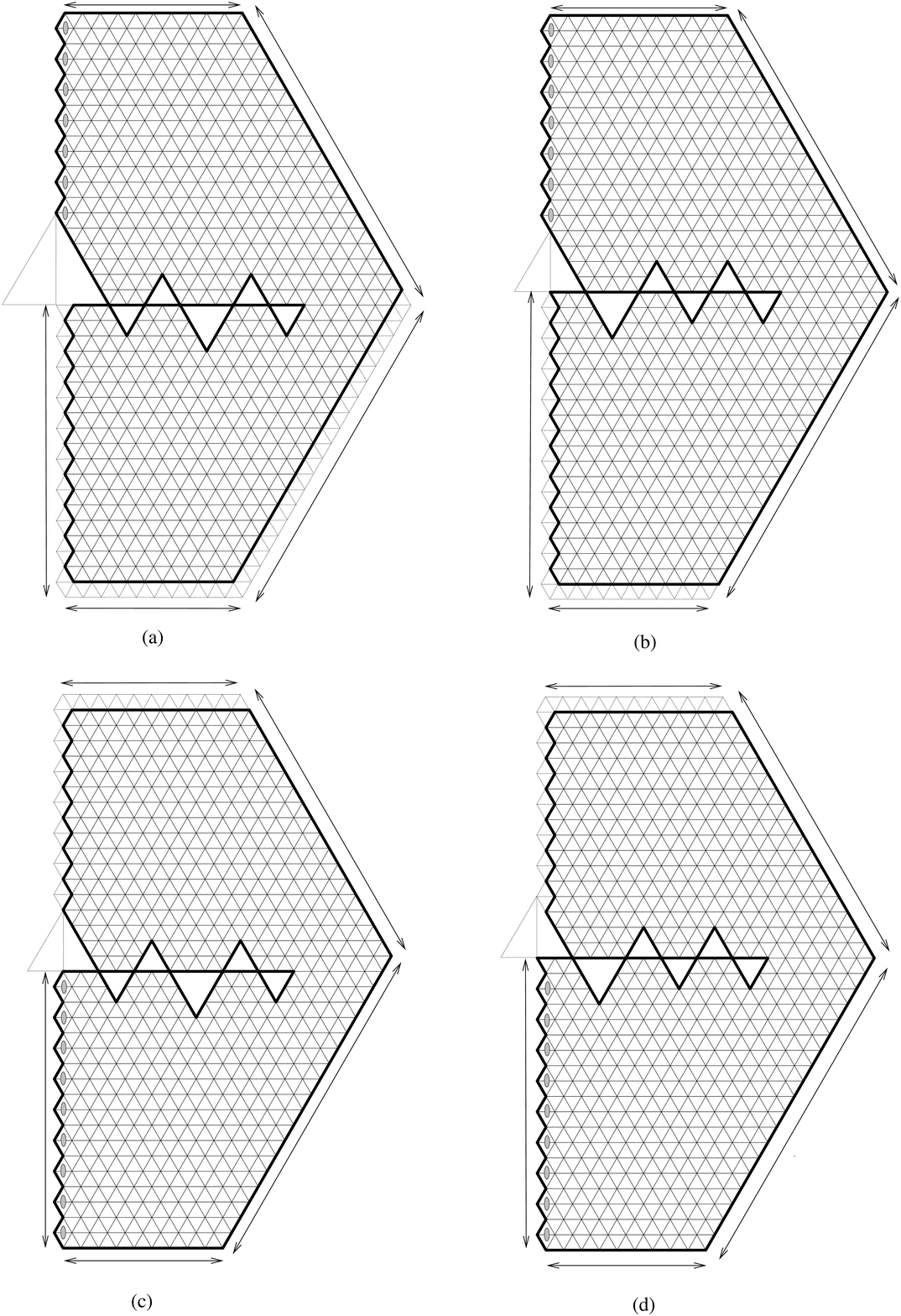}%
\end{picture}%
%
%  Created by WinFIG version 6.2
%  METADATA <version>1.0</version>
%

\begin{picture}(19811,29362)(3845,-28553)
%  METADATA <id>1120</id>
\put(18927,-6364){\makebox(0,0)[lb]{\smash{{\SetFigFont{20}{24.0}{\rmdefault}{\mddefault}{\updefault}{$a_4$}%
}}}}
%  METADATA <id>1121</id>
\put(19677,-5959){\makebox(0,0)[lb]{\smash{{\SetFigFont{20}{24.0}{\rmdefault}{\mddefault}{\updefault}{$a_5$}%
}}}}
%  METADATA <id>1122</id>
\put(20532,-6364){\makebox(0,0)[lb]{\smash{{\SetFigFont{20}{24.0}{\rmdefault}{\mddefault}{\updefault}{$a_6$}%
}}}}
%  METADATA <id>1124</id>
\put(5584,-6205){\makebox(0,0)[lb]{\smash{{\SetFigFont{20}{24.0}{\rmdefault}{\mddefault}{\updefault}{$a_1$}%
}}}}
%  METADATA <id>1126</id>
\put(6499,-6595){\makebox(0,0)[lb]{\smash{{\SetFigFont{20}{24.0}{\rmdefault}{\mddefault}{\updefault}{$a_2$}%
}}}}
%  METADATA <id>1128</id>
\put(7249,-6235){\makebox(0,0)[lb]{\smash{{\SetFigFont{20}{24.0}{\rmdefault}{\mddefault}{\updefault}{$a_3$}%
}}}}
%  METADATA <id>1130</id>
\put(8149,-6730){\makebox(0,0)[lb]{\smash{{\SetFigFont{20}{24.0}{\rmdefault}{\mddefault}{\updefault}{$a_4$}%
}}}}
%  METADATA <id>1132</id>
\put(9214,-6235){\makebox(0,0)[lb]{\smash{{\SetFigFont{20}{24.0}{\rmdefault}{\mddefault}{\updefault}{$a_5$}%
}}}}
%  METADATA <id>1134</id>
\put(10039,-6580){\makebox(0,0)[lb]{\smash{{\SetFigFont{20}{24.0}{\rmdefault}{\mddefault}{\updefault}{$a_6$}%
}}}}
%  METADATA <id>1136</id>
\put(5793,-14451){\makebox(0,0)[lb]{\smash{{\SetFigFont{20}{24.0}{\rmdefault}{\mddefault}{\updefault}{$x+a_2+a_4+a_6$}%
}}}}
%  METADATA <id>1137</id>
\put(16405,-14533){\makebox(0,0)[lb]{\smash{{\SetFigFont{20}{24.0}{\rmdefault}{\mddefault}{\updefault}{$x+a_2+a_4+a_6$}%
}}}}
%  METADATA <id>1139</id>
\put(10332,-15655){\rotatebox{300.0}{\makebox(0,0)[lb]{\smash{{\SetFigFont{20}{24.0}{\rmdefault}{\mddefault}{\updefault}{$y+z+2a_1+2a_3+2a_5$}%
}}}}}
%  METADATA <id>1142</id>
\put(4615,-25948){\rotatebox{90.0}{\makebox(0,0)[lb]{\smash{{\SetFigFont{20}{24.0}{\rmdefault}{\mddefault}{\updefault}{$2z+2a_2+2a_4+2a_6$}%
}}}}}
%  METADATA <id>1143</id>
\put(15235,-25693){\rotatebox{90.0}{\makebox(0,0)[lb]{\smash{{\SetFigFont{20}{24.0}{\rmdefault}{\mddefault}{\updefault}{$2z-1+2a_2+2a_4+2a_6$}%
}}}}}
%  METADATA <id>1144</id>
\put(21069,-15823){\rotatebox{300.0}{\makebox(0,0)[lb]{\smash{{\SetFigFont{20}{24.0}{\rmdefault}{\mddefault}{\updefault}{$y+z-1+2a_1+2a_3+2a_5$}%
}}}}}
%  METADATA <id>1145</id>
\put(20771,-26368){\rotatebox{60.0}{\makebox(0,0)[lb]{\smash{{\SetFigFont{20}{24.0}{\rmdefault}{\mddefault}{\updefault}{$y+z-1+2a2+2a_4+2a_6$}%
}}}}}
%  METADATA <id>1147</id>
\put(16383,-27906){\makebox(0,0)[lb]{\smash{{\SetFigFont{20}{24.0}{\rmdefault}{\mddefault}{\updefault}{$x+a_1+a_3+a5$}%
}}}}
%  METADATA <id>1149</id>
\put(6295,-21246){\makebox(0,0)[lb]{\smash{{\SetFigFont{20}{24.0}{\rmdefault}{\mddefault}{\updefault}{$a_2$}%
}}}}
%  METADATA <id>1150</id>
\put(16780,-21021){\makebox(0,0)[lb]{\smash{{\SetFigFont{20}{24.0}{\rmdefault}{\mddefault}{\updefault}{$a_2$}%
}}}}
%  METADATA <id>1152</id>
\put(7030,-20871){\makebox(0,0)[lb]{\smash{{\SetFigFont{20}{24.0}{\rmdefault}{\mddefault}{\updefault}{$a_3$}%
}}}}
%  METADATA <id>1153</id>
\put(17830,-20571){\makebox(0,0)[lb]{\smash{{\SetFigFont{20}{24.0}{\rmdefault}{\mddefault}{\updefault}{$a_3$}%
}}}}
%  METADATA <id>1155</id>
\put(7930,-21381){\makebox(0,0)[lb]{\smash{{\SetFigFont{20}{24.0}{\rmdefault}{\mddefault}{\updefault}{$a_4$}%
}}}}
%  METADATA <id>1156</id>
\put(18700,-20961){\makebox(0,0)[lb]{\smash{{\SetFigFont{20}{24.0}{\rmdefault}{\mddefault}{\updefault}{$a_4$}%
}}}}
%  METADATA <id>1158</id>
\put(8995,-20856){\makebox(0,0)[lb]{\smash{{\SetFigFont{20}{24.0}{\rmdefault}{\mddefault}{\updefault}{$a_5$}%
}}}}
%  METADATA <id>1159</id>
\put(19375,-20556){\makebox(0,0)[lb]{\smash{{\SetFigFont{20}{24.0}{\rmdefault}{\mddefault}{\updefault}{$a_5$}%
}}}}
%  METADATA <id>1161</id>
\put(20230,-20961){\makebox(0,0)[lb]{\smash{{\SetFigFont{20}{24.0}{\rmdefault}{\mddefault}{\updefault}{$a_6$}%
}}}}
%  METADATA <id>1162</id>
\put(9790,-21246){\makebox(0,0)[lb]{\smash{{\SetFigFont{20}{24.0}{\rmdefault}{\mddefault}{\updefault}{$a_6$}%
}}}}
%  METADATA <id>1164</id>
\put(10051,-26184){\rotatebox{60.0}{\makebox(0,0)[lb]{\smash{{\SetFigFont{20}{24.0}{\rmdefault}{\mddefault}{\updefault}{$y+z+2a_2+2a_4+2a_6$}%
}}}}}
%  METADATA <id>1166</id>
\put(5508,-27779){\makebox(0,0)[lb]{\smash{{\SetFigFont{20}{24.0}{\rmdefault}{\mddefault}{\updefault}{$x+a_1+a_3+a5$}%
}}}}
%  METADATA <id>1044</id>
\put(15290,-5609){\rotatebox{60.0}{\makebox(0,0)[lb]{\smash{{\SetFigFont{20}{24.0}{\rmdefault}{\mddefault}{\updefault}{$2a_1$}%
}}}}}
%  METADATA <id>1047</id>
\put(14991,-20226){\rotatebox{60.0}{\makebox(0,0)[lb]{\smash{{\SetFigFont{20}{24.0}{\rmdefault}{\mddefault}{\updefault}{$2a_1$}%
}}}}}
%  METADATA <id>1050</id>
\put(4596,-20526){\rotatebox{60.0}{\makebox(0,0)[lb]{\smash{{\SetFigFont{20}{24.0}{\rmdefault}{\mddefault}{\updefault}{$2a_1$}%
}}}}}
%  METADATA <id>1056</id>
\put(16486,419){\makebox(0,0)[lb]{\smash{{\SetFigFont{20}{24.0}{\rmdefault}{\mddefault}{\updefault}{$x+a_2+a_4+a_6$}%
}}}}
%  METADATA <id>1057</id>
\put(21102,-1219){\rotatebox{300.0}{\makebox(0,0)[lb]{\smash{{\SetFigFont{20}{24.0}{\rmdefault}{\mddefault}{\updefault}{$y+z+2a_1+2a_3+2a_5$}%
}}}}}
%  METADATA <id>1060</id>
\put(16377,-13384){\makebox(0,0)[lb]{\smash{{\SetFigFont{20}{24.0}{\rmdefault}{\mddefault}{\updefault}{$x+a_1+a_3+a5$}%
}}}}
%  METADATA <id>1064</id>
\put(5608,444){\makebox(0,0)[lb]{\smash{{\SetFigFont{20}{24.0}{\rmdefault}{\mddefault}{\updefault}{$x+a_2+a_4+a_6$}%
}}}}
%  METADATA <id>1067</id>
\put(15245,-11652){\rotatebox{90.0}{\makebox(0,0)[lb]{\smash{{\SetFigFont{20}{24.0}{\rmdefault}{\mddefault}{\updefault}{$2z+2a_2+2a_4+2a_6$}%
}}}}}
%  METADATA <id>1068</id>
\put(21117,-11554){\rotatebox{60.0}{\makebox(0,0)[lb]{\smash{{\SetFigFont{20}{24.0}{\rmdefault}{\mddefault}{\updefault}{$y+z+2a_2+2a_4+2a_6$}%
}}}}}
%  METADATA <id>1107</id>
\put(6031,-13336){\makebox(0,0)[lb]{\smash{{\SetFigFont{20}{24.0}{\rmdefault}{\mddefault}{\updefault}{$x+1+a_1+a_3+a_5$}%
}}}}
%  METADATA <id>1109</id>
\put(10669,-11747){\rotatebox{60.0}{\makebox(0,0)[lb]{\smash{{\SetFigFont{20}{24.0}{\rmdefault}{\mddefault}{\updefault}{$y+z+2a_2+2a_4+2a_6$}%
}}}}}
%  METADATA <id>1110</id>
\put(10473,-932){\rotatebox{300.0}{\makebox(0,0)[lb]{\smash{{\SetFigFont{20}{24.0}{\rmdefault}{\mddefault}{\updefault}{$y+z+1+2a_1+2a_3+2a_5$}%
}}}}}
%  METADATA <id>1115</id>
\put(4699,-11852){\rotatebox{90.0}{\makebox(0,0)[lb]{\smash{{\SetFigFont{20}{24.0}{\rmdefault}{\mddefault}{\updefault}{$2z+1+2a_2+2a_4+2a_6$}%
}}}}}
%  METADATA <id>1116</id>
\put(4164,-5615){\rotatebox{60.0}{\makebox(0,0)[lb]{\smash{{\SetFigFont{20}{24.0}{\rmdefault}{\mddefault}{\updefault}{$2a_1+2$}%
}}}}}
%  METADATA <id>1117</id>
\put(16047,-5944){\makebox(0,0)[lb]{\smash{{\SetFigFont{20}{24.0}{\rmdefault}{\mddefault}{\updefault}{$a_1$}%
}}}}
%  METADATA <id>1118</id>
\put(17067,-6439){\makebox(0,0)[lb]{\smash{{\SetFigFont{20}{24.0}{\rmdefault}{\mddefault}{\updefault}{$a_2$}%
}}}}
%  METADATA <id>1119</id>
\put(18072,-5959){\makebox(0,0)[lb]{\smash{{\SetFigFont{20}{24.0}{\rmdefault}{\mddefault}{\updefault}{$a_3$}%
}}}}
\end{picture}}
\caption{The four mixed boundary halved hexagons: (a) $H^{(5)}_{3,3,2}(2,2,2,3,2,2)$,  (b)  $H^{(6)}_{3,3,3}(2,3,2,2,2,2)$, (c) $H^{(7)}_{3,3,2}(2,2,2,3,2,2)$, and (d) $H^{(8)}_{3,3,3}(2,3,2,2,2,2)$.} \label{middlehole3}
\end{figure}

The first `mixed boundary' region $H^{(5)}_{x,y,z}(a_1,a_2,\dotsc,a_k)$ is obtained from the weighted region\\  $H^{(4)}_{x,y,z+1}(a_1+1,a_2,\dotsc,a_k)$ by removing all unit triangles running along the southern and the southeastern sides, as well as the portion of the western boundary below the array of holes (see Figure \ref{middlehole3}(a) for the case $x=3,y=3,z=2,k=6,a_1=2,a_2=2,a_3=2,a_4=3,a_5=2,a_6=2$; the dotted triangles indicate the triangles removed; the lozenges with shaded cores are weighted by $1/2$ as usual). We can see that that only the lozenges above the array of holes and along the western boundary side are weighted in the region $H^{(5)}_{x,y,z}(a_1,\dotsc,a_k)$. The second mixed boundary region $H^{(6)}_{x,y,z}(a_1,a_2,\dotsc,a_k)$ is created from $H^{(3)}_{x,y,z}(a_1,a_2,\dotsc,a_k)$ by removing all unit triangles running along the southern side and the western side below the array of holes (see Figure \ref{middlehole3}(b) for $x=3,y=3,z=3,k=6,a_1=2,a_2=3,a_3=2,a_4=2,a_5=2,a_6=2$).

Similarly, the region $H^{(7)}_{x,y,z}(a_1,a_2,\dotsc,a_k)$ is obtained from the region $H^{(3)}_{x,y,z}(a_1,a_2,\dotsc,a_k)$ by removing all unit triangles running along the \emph{northern} side and the western side \emph{above} the holes (shown in Figure \ref{middlehole3}(c) when $x=3,y=3,z=2,k=6,a_1=2,a_2=2,a_3=2,a_4=3,a_5=2,a_6=2$).  Finally, the region $H^{(8)}_{x,y,z}(a_1,a_2,\dotsc,a_k)$ is the remaining  by removing the unit triangles running along the \emph{northern} side and the western side \emph{above} the holes in $H^{(4)}_{x,y,z}(a_1,a_2,\dotsc,a_k)$  (illustrated in Figure \ref{middlehole3}(d) for $x=3,y=3,z=3,k=6,a_1=2,a_2=3,a_3=2,a_4=2,a_5=2,a_6=2$).

The tiling numbers of the above four regions with mixed boundary are all given by simple product formulas.

\begin{thm}\label{main5}
Assume that $x,y,z,k$ are non-negative integers and  $\textbf{a}=(a_1,a_2,\dotsc,a_k)$ is a sequence of $k$ non-negative integers. The number of tilings of the defected halved hexagon  $H^{(5)}_{x,y,z}(a_1,a_2,\dotsc,a_k)$ is given by
\begin{align}\label{main5eq1}
\M&(H^{(5)}_{x,y,z}(a,b))=\frac{\Pn'_{y,y+2a+1,b}\Pn_{z+b,z+b,a}\Q(a,b,x,y+z)}{\Pn_{y+z+b,y+z+b,a}}\frac{\V(2a+2b+3,y+z-1,y)}{\V(2x+2a+2b+3,y+z-1,y)}\notag\\
&\frac{\T(x+b+1,y+z+2a,y)\T(2a+b+2,y+b-1,b)\T(z+1,y+b-1,b)}{\T(b+1,y+z+2a,y)\T(x+2a+b+2,y+b-1,b)\T(x+z+1,y+b-1,b)},
\end{align}
and for $k\geq2$
\begin{align}\label{main5eq2}
\M&(H^{(5)}_{x,y,z}(a_1,a_2,\dotsc,a_{2k}))=2^{a_1}\M(H^{(5)}_{x,y,z}(\Od(\textbf{a}),\E(\textbf{a})))\frac{\K'(0,a_1+1,a_2,\dotsc,a_{2k},y)\Q(a_1,\dotsc,a_{2k}+z)}{\Pn'_{y,y+2\Od(\textbf{a})+1,\E(\textbf{a})}\Pn_{z+\E(\textbf{a}),z+\E(\textbf{a}),\Od(\textbf{a})}}\notag\\
&\times \prod_{i=2}^{k}\frac{\T(x+z+\e_{i}(\textbf{a})+1,a_{2i-2}+\od_{i}(\textbf{a})-1,\od_{i}(\textbf{a}))}{\T(x+y+\e_{i}(\textbf{a})+1,a_{2i-2}+\od_{i}(\textbf{a})-1,\od_{i}(\textbf{a}))}\frac{\T(y+\e_{i}(\textbf{a})+1,a_{2i-2}+\od_{i}(\textbf{a})-1,\od_{i}(\textbf{a}))}{\T(z+\e_{i}(\textbf{a})+1,a_{2i-2}+\od_{i}(\textbf{a})-1,\od_{i}(\textbf{a}))}\notag\\
&\times \prod_{i=2}^{k}\frac{\T(x+y+\s_{2i-1}(\textbf{a})+\s_{2k}(\textbf{a})+2,a_{2i-2}+\od_{i}(\textbf{a})-1,\od_{i}(\textbf{a}))}{\T(x+z+\s_{2i-1}(\textbf{a})+\s_{2k}(\textbf{a})+2,a_{2i-2}+\od_{i}(\textbf{a})-1,\od_{i}(\textbf{a}))}\notag\\
&\times \prod_{i=2}^{k}\frac{\T(z+\s_{2i-1}(\textbf{a})+\s_{2k}(\textbf{a})+2,a_{2i-2}+\od_{i}(\textbf{a})-1,\od_{i}(\textbf{a}))}{\T(y+\s_{2i-1}(\textbf{a})+\s_{2k}(\textbf{a})+2,a_{2i-2}+\od_{i}(\textbf{a})-1,\od_{i}(\textbf{a}))}.
\end{align}
%\begin{equation}
%\M(H^{(5)}_{x,y,z}(a_1,a_2,\dotsc,a_{2k-1}))=\M(H^{(5)}_{x,y,z}(a_1,a_2,\dotsc,a_{2k-1},0)).
%\end{equation}
\end{thm}

\begin{thm}\label{main6}
Assume that $x,y,z,k$ are non-negative integers and  $\textbf{a}=(a_1,a_2,\dotsc,a_k)$ is a sequence of $k$ non-negative integers. The number of tilings of the defected halved hexagon  $H^{(6)}_{x,y,z}(a_1,a_2,\dotsc,a_k)$ is given by
\begin{align}
\M&(H^{(6)}_{x,y,z}(a,b))=\frac{\Pn'_{y,y+2a,b}\Pn_{z+b-1,z+b-1,a}\K(a,b,x,y+z)}{\Pn_{y+z+b-1,y+z+b-1,a}}\frac{\V(2a+2b+3,y+z-2,y)}{\V(2x+2a+2b+3,y+z-2,y)}\notag\\
&\frac{\T(x+b+1,y+z+2a-1,y)\T(2a+b+1,y+b-1,b)\T(z+1,y+b-1,b)}{\T(b+1,y+z+2a-1,y)\T(x+2a+b+1,y+b-1,b)\T(x+z+1,y+b-1,b)},
\end{align}
and for $k\geq 2$
\begin{align}
\M&(H^{(6)}_{x,y,z}(a_1,a_2,\dotsc,a_{2k}))=2^{a_1}\M(H^{(6)}_{x,y,z}(\Od(\textbf{a}),\E(\textbf{a})))\frac{\Q'(0,a_1,a_2,\dotsc,a_{2k},y)\K(a_1,\dotsc,a_{2k}+z)}{\Pn'_{y,y+2\Od(\textbf{a}),\E(\textbf{a})}\Pn_{z+\E(\textbf{a})-1,z+\E(\textbf{a})-1,\Od(\textbf{a})}}\notag\\
&\times \prod_{i=2}^{k}\frac{\T(x+z+\e_{i}(\textbf{a})+1,a_{2i-2}+\od_{i}(\textbf{a})-1,\od_{i}(\textbf{a}))}{\T(x+y+\e_{i}(\textbf{a})+1,a_{2i-2}+\od_{i}(\textbf{a})-1,\od_{i}(\textbf{a}))}\frac{\T(y+\e_{i}(\textbf{a})+1,a_{2i-2}+\od_{i}(\textbf{a})-1,\od_{i}(\textbf{a}))}{\T(z+\e_{i}(\textbf{a})+1,a_{2i-2}+\od_{i}(\textbf{a})-1,\od_{i}(\textbf{a}))}\notag\\
&\times \prod_{i=2}^{k}\frac{\T(x+y+\s_{2i-1}(\textbf{a})+\s_{2k}(\textbf{a})+1,a_{2i-2}+\od_{i}(\textbf{a})-1,\od_{i}(\textbf{a}))}{\T(x+z+\s_{2i-1}(\textbf{a})+\s_{2k}(\textbf{a})+1,a_{2i-2}+\od_{i}(\textbf{a})-1,\od_{i}(\textbf{a}))}\notag\\
&\times \prod_{i=2}^{k}\frac{\T(z+\s_{2i-1}(\textbf{a})+\s_{2k}(\textbf{a})+1,a_{2i-2}+\od_{i}(\textbf{a})-1,\od_{i}(\textbf{a}))}{\T(y+\s_{2i-1}(\textbf{a})+\s_{2k}(\textbf{a})+1,a_{2i-2}+\od_{i}(\textbf{a})-1,\od_{i}(\textbf{a}))}.
\end{align}
%and
%\begin{equation}
%\M(H^{(6)}_{x,y,z}(a_1,a_2,\dotsc,a_{2k-1}))=\M(H^{(6)}_{x,y,z}(a_1,a_2,\dotsc,a_{2k-1},0)).
%\end{equation}
\end{thm}

\begin{thm}\label{main7}
Assume that $x,y,z,k$ are non-negative integers and  $\textbf{a}=(a_1,a_2,\dotsc,a_k)$ is a sequence of $k$ non-negative integers. The number of tilings of the defected halved hexagon  $H^{(7)}_{x,y,z}(a_1,a_2,\dotsc,a_k)$ is given by
\begin{align}
\M&(H^{(7)}_{x,y,z}(a,b))=\frac{\Pn_{y,y+2a-1,b}\Pn'_{z+b,z+b,a}\Q'(a,b,x,y+z)}{\Pn'_{y+z+b,y+z+b,a}}\frac{\V(2a+2b+1,y+z,y)}{\V(2x+2a+2b+1,y+z,y)}\notag\\
&\frac{\T(x+b+1,y+z+2a-1,y)\T(2a+b+1,y+b-1,b)\T(z+1,y+b-1,b)}{\T(b+1,y+z+2a-1,y)\T(x+2a+b+1,y+b-1,b)\T(x+z+1,y+b-1,b)},
\end{align}
and for $k\geq 2$
\begin{align}
\M&(H^{(7)}_{x,y,z}(a_1,a_2,\dotsc,a_{2k}))=\M(H^{(7)}_{x,y,z}(\Od(\textbf{a}),\E(\textbf{a})))\frac{\K(0,a_1,a_2,\dotsc,a_{2k},y)\Q'(a_1,\dotsc,a_{2k}+z)}{\Pn_{y,y+2\Od(\textbf{a})-1,\E(\textbf{a})}\Pn'_{z+\E(\textbf{a}),z+\E(\textbf{a}),\Od(\textbf{a})}}\notag\\
&\times \prod_{i=2}^{k}\frac{\T(x+z+\e_{i}(\textbf{a})+1,a_{2i-2}+\od_{i}(\textbf{a})-1,\od_{i}(\textbf{a}))}{\T(x+y+\e_{i}(\textbf{a})+1,a_{2i-2}+\od_{i}(\textbf{a})-1,\od_{i}(\textbf{a}))}\frac{\T(y+\e_{i}(\textbf{a})+1,a_{2i-2}+\od_{i}(\textbf{a})-1,\od_{i}(\textbf{a}))}{\T(z+\e_{i}(\textbf{a})+1,a_{2i-2}+\od_{i}(\textbf{a})-1,\od_{i}(\textbf{a}))}\notag\\
&\times \prod_{i=2}^{k}\frac{\T(x+y+\s_{2i-1}(\textbf{a})+\s_{2k}(\textbf{a})+1,a_{2i-2}+\od_{i}(\textbf{a})-1,\od_{i}(\textbf{a}))}{\T(x+z+\s_{2i-1}(\textbf{a})+\s_{2k}(\textbf{a})+1,a_{2i-2}+\od_{i}(\textbf{a})-1,\od_{i}(\textbf{a}))}\notag\\
&\times \prod_{i=2}^{k}\frac{\T(z+\s_{2i-1}(\textbf{a})+\s_{2k}(\textbf{a})+1,a_{2i-2}+\od_{i}(\textbf{a})-1,\od_{i}(\textbf{a}))}{\T(y+\s_{2i-1}(\textbf{a})+\s_{2k}(\textbf{a})+1,a_{2i-2}+\od_{i}(\textbf{a})-1,\od_{i}(\textbf{a}))}.
\end{align}
%and
%\begin{equation}
%\M(H^{(7)}_{x,y,z}(a_1,a_2,\dotsc,a_{2k-1}))=\M(H^{(7)}_{x,y,z}(a_1,a_2,\dotsc,a_{2k-1},0)).
%\end{equation}
\end{thm}

\begin{thm}\label{main8}
Assume that $x,y,z,k$ are non-negative integers and  $\textbf{a}=(a_1,a_2,\dotsc,a_k)$ is a sequence of $k$ non-negative integers. The number of tilings of the defected halved hexagon  $H^{(8)}_{x,y,z}(a_1,a_2,\dotsc,a_k)$ is given by
\begin{align}
\M&(H^{(8)}_{x,y,z}(a,b))=\frac{\Pn_{y,y+2a-2,b}\Pn'_{z+b-1,y+b-1,a}\K'(a,b,x,y+z)}{\Pn'_{y+z+b-1,y+z+b-1,a}}\frac{\V(2a+2b+1,y+z-1,y)}{\V(2x+2a+2b+1,y+z-1,y)}\notag\\
&\frac{\T(x+b+1,y+z+2a-2,y)\T(2a+b,y+b-1,b)\T(z+1,y+b-1,b)}{\T(b+1,y+z+2a-2,y)\T(x+2a+b,y+b-1,b)\T(x+z+1,y+b-1,b)},
\end{align}
and for $k\geq 2$ 
\begin{align}
\M&(H^{(8)}_{x,y,z}(a_1,a_2,\dotsc,a_{2k}))=\M(H^{(8)}_{x,y,z}(\Od(\textbf{a}),\E(\textbf{a})))\frac{\Q(0,a_1-1,a_2,\dotsc,a_{2k},y)\K'(a_1,\dotsc,a_{2k}+z)}{\Pn_{y,y+2\Od(\textbf{a})-2,\E(\textbf{a})}\Pn'_{z+\E(\textbf{a})-1,y+\E(\textbf{a})-1,\Od(\textbf{a})}}\notag\\
&\times \prod_{i=2}^{k}\frac{\T(x+z+\e_{i}(\textbf{a})+1,a_{2i-2}+\od_{i}(\textbf{a})-1,\od_{i}(\textbf{a}))}{\T(x+y+\e_{i}(\textbf{a})+1,a_{2i-2}+\od_{i}(\textbf{a})-1,\od_{i}(\textbf{a}))}\frac{\T(y+\e_{i}(\textbf{a})+1,a_{2i-2}+\od_{i}(\textbf{a})-1,\od_{i}(\textbf{a}))}{\T(z+\e_{i}(\textbf{a})+1,a_{2i-2}+\od_{i}(\textbf{a})-1,\od_{i}(\textbf{a}))}\notag\\
&\times \prod_{i=2}^{k}\frac{\T(x+y+\s_{2i-1}(\textbf{a})+\s_{2k}(\textbf{a}),a_{2i-2}+\od_{i}(\textbf{a})-1,\od_{i}(\textbf{a}))}{\T(x+z+\s_{2i-1}(\textbf{a})+\s_{2k}(\textbf{a}),a_{2i-2}+\od_{i}(\textbf{a})-1,\od_{i}(\textbf{a}))}\notag\\
&\times \prod_{i=2}^{k}\frac{\T(z+\s_{2i-1}(\textbf{a})+\s_{2k}(\textbf{a}),a_{2i-2}+\od_{i}(\textbf{a})-1,\od_{i}(\textbf{a}))}{\T(y+\s_{2i-1}(\textbf{a})+\s_{2k}(\textbf{a}),a_{2i-2}+\od_{i}(\textbf{a})-1,\od_{i}(\textbf{a}))}.
\end{align}
%and
%\begin{equation}
%\M(H^{(8)}_{x,y,z}(a_1,a_2,\dotsc,a_{2k-1}))=\M(H^{(8)}_{x,y,z}(a_1,a_2,\dotsc,a_{2k-1},0)).
%\end{equation}
\end{thm}

Assume that $x,y,z$ are non-negative integers, and   $\textbf{a}=(a_1,a_2,\dotsc,a_n)$ is a sequence of nonnegative integers as usual. Consider a symmetric hexagon of side-lengths $y+2\Od(\textbf{a})-a_1,x+2\E(\textbf{b}),y+2\Od(\textbf{a})-a_1,y+2\E(\textbf{a}),x+2\E(\textbf{a}),y+2\E(\textbf{a})$. We remove at level $z$ (from the bottom) a symmetric array of triangles of sides $a_n,a_{n-1},\dotsc,a_2,a_1,a_2,\dotsc,a_{n-1},a_n$ ordered from left to right, so that the middle triangle is an up-pointing triangle of side $a_1$. Denote by $\mathcal{S}_{x,y,z}(\textbf{a})=\mathcal{S}_{x,y,z}(a_1,a_2,\dotsc,a_n)$ the resulting region (see Figure \ref{middlearrayn}
 for examples).

We note that, by the symmetry, if  the total length of the array, $2(\E(\textbf{a})+\Od(\textbf{a}))-a_1$, and the length of the base of the hexagon, $x+2\Od(\textbf{a})-a_1$, have the same parity, then $z$ must be even. Thus, $z$ and $x$ \emph{always} have the same parity.

\begin{thm}\label{main9}
Assume that $x,y,z$ are non-negative integers and $\textbf{a}=(a_1,a_2,\dotsc,a_n)$ is a sequence of positive integers.  If $2\E(\textbf{a})-1\leq z \leq 2y+2\E(\textbf{a})+1$, then the number of tilings of $\mathcal{S}_{x,y,z}(\textbf{a})$ is always given by  a simple product formula as follows.

(1) If $x$ is even (so $z$ is even) and $a_1$ is even, then
\begin{align}
\M(\mathcal{S}_{x,y,z}(\textbf{a}))&=2^{y+a_2+a_3+\dotsc+a_n}\M\left( H^{(2)}_{\frac{x}{2}+\E(\textbf{a}),y-\frac{z}{2}+\E(\textbf{a}) ,\frac{z}{2}-\E(\textbf{a})}\left(\frac{a_1}{2},a_2,\dotsc,a_n\right)\right)\notag\\
&\times\M\left( H^{(3)}_{\frac{x}{2}+\E(\textbf{a}),y-\frac{z}{2}+\E(\textbf{a}) ,\frac{z}{2}-\E(\textbf{a})}\left(\frac{a_1}{2},a_2,\dotsc,a_n\right)\right).
\end{align}

(2) If $x$ is odd (so $z$ is odd)  and $a_1$ is even, then
\begin{align}
\M(\mathcal{S}_{x,y,z}(\textbf{a}))&=2^{y+a_2+a_3+\dotsc+a_n}\M\left( H^{(2)}_{\frac{x-1}{2}+\E(\textbf{a}),y-\frac{z}{2}+\E(\textbf{a}) ,\frac{z}{2}-\E(\textbf{a})+1}\left(\frac{a_1}{2},a_2,\dotsc,a_n\right)\right)\notag\\
&\times\M\left( H^{(3)}_{\frac{x+1}{2}+\E(\textbf{a}),y-\frac{z-1}{2}+\E(\textbf{a})-1 ,\frac{z-1}{2}-\E(\textbf{a})}\left(\frac{a_1}{2},a_2,\dotsc,a_n\right)\right).
\end{align}

(3) If $x$ is even (so $z$ is even) and $a_1$ is odd, then
\begin{align}
\M(\mathcal{S}_{x,y,z}(\textbf{a}))&=2^{y+a_2+a_3+\dotsc+a_n}\M\left( H^{(5)}_{\frac{x}{2}+\E(\textbf{a}),y-\frac{z}{2}+\E(\textbf{a}) ,\frac{z}{2}-\E(\textbf{a})}\left(\frac{a_1-1}{2},a_2,\dotsc,a_n\right)\right)\notag\\
&\times\M\left( H^{(8)}_{\frac{x}{2}+\E(\textbf{a}),y-\frac{z}{2}+\E(\textbf{a}),\frac{z}{2}-\E(\textbf{a})}\left(\frac{a_1+1}{2},a_2,\dotsc,a_n\right)\right).
\end{align}
(4) If $x$ is odd  (so $z$ is odd) and $a_1$ is odd, then
\begin{align}
\M(\mathcal{S}_{x,y,z}(\textbf{a}))&=2^{y+a_2+a_3+\dotsc+a_n}\M\left( H^{(5)}_{\frac{x+1}{2}+\E(\textbf{a}),y-\frac{z-1}{2}+\E(\textbf{a})-1 ,\frac{z-1}{2}-\E(\textbf{a})}\left(\frac{a_1-1}{2},a_2,\dotsc,a_n\right)\right)\notag\\
&\times\M\left( H^{(8)}_{\frac{x-1}{2}+\E(\textbf{a}),y-\frac{z-1}{2}+\E(\textbf{a}),\frac{z-1}{2}-\E(\textbf{a})+1}\left(\frac{a_1+1}{2},a_2,\dotsc,a_n\right)\right).
\end{align}
\end{thm}

Note that when $z<2\E(\textbf{a})-1$ or $z>2y+2\E(\textbf{a})+1$, then $\M(\mathcal{S}_{x,y,z}(\textbf{a}))=0$. This will be explained later in Remark \ref{remark1} of  Section \ref{Mainproofs}.

\section{Preliminaries}\label{Background}
Let $G$ be a simple graph without loop. A \emph{perfect matching} (or simple \emph{matching} in this paper) of $G$ is a collection vertex-disjoint edges that cover all vertices of $G$. We use the notation $\M(G)$ for the number of matchings of $G$. The tilings of a region $R$ can be identified with the matchings of its \emph{dual graph} $G$ (the graph whose vertices are unit triangles in $R$ and whose edges connect precisely two unit triangles sharing an edge). In the weighted case, each edge of the dual graph $G$ has the same weight as the corresponding lozenge in the region $R$.  In this case $M(G)$ is the weighted number matchings of $G$, i.e. the sum of weights of all matchings in $G$, where the weight of a matching is the product of all weight of its constituent edges.

If a region admits a tiling, then it has the same number of up- and down-pointing unit triangles. We call the regions satisfying the latter balancing condition \emph{balanced} regions.

\begin{lem}[Region-splitting Lemma]\label{RS}
Let $R$ be a balanced region. Assume that a subregion $S$ of $R$ satisfies the following two conditions:
\begin{enumerate}
\item[(i)] \text{\rm{(Separating Condition)}} There is only one type of unit triangle (up-pointing or down-pointing) running along each side of the border between $S$ and $R-S$.

\item[(ii)] \text{\rm{(Balancing Condition)}} $S$ is balanced.
\end{enumerate}
Then
\begin{equation}\label{GSeq}
\M(R)=\M(S)\, \M(R-S).
\end{equation}
\end{lem}
\begin{proof}
Assume there is a tiling of $R$ which contains boundary-crossing lozenges between $S$ and $R - S$ (i.e., lozenges which consist of a unit triangle from the boundary of $S$ and a unit triangle from the boundary of $R - S$). Since there is only one type of unit triangle on each side of the boundary between $S$ and $R - S$, and since $S$ and $R-S$ are balanced, the regions obtained by removing such boundary-crossing lozenges would no longer be balanced, and hence would have no tilings. Therefore, there can not be any boundary-crossing lozenges, and $S$ and $R - S$ must be tiled independently, giving the factorization (\ref{GSeq}).
\end{proof}

One of the main ingredients of our proofs is the following powerful theorem by Eric H. Kuo \cite{Kuo} that is usually mentioned as \emph{Kuo condensation}.
\begin{thm}[Theorem 5.1 in \cite{Kuo}]\label{Kuothm}
Assume that $G=(V_1,V_2,E)$ is a weighted bipartite planar graph with two vertex classes $V_1$ and $V_2$ of the same cardinality. Assume in addition that $u,v,w,s$ are four vertices appearing on a cyclic order on a face of $G$, such that $u,w \in V_1$ and $v,s\in V_2$. Then
\begin{align}\label{Kuoeq}
\M(G)\M(G-\{u,v,w,s\})=\M(G-\{u,v\})\M(G-\{w,s\})+\M(G-\{u,s\})\M(G-\{v,w\}).
\end{align}
\end{thm}

The next lemma is often called \emph{Ciucu's factorization theorem} (Theorem 1.2 in \cite{Ciucu3}), that allows us write the number of matchings of  a symmetric graph as the product of the matching numbers of two smaller graphs.

\begin{lem}[Ciucu's Factorization Theorem]\label{ciucuthm}
Let $G=(V_1,V_2,E)$ be a weighted bipartite planar graph with a vertical symmetry axis $\ell$. Assume that $a_1,b_1,a_2,b_2,\dots,a_k,b_k$ are all the vertices of $G$ on $\ell$ appearing in this order from top to bottom\footnote{It is easy to see that if $G$ admits a perfect matching, then $G$ has an even number of vertices on $\ell$.}. Assume in addition that the vertices of $G$ on $\ell$ form a cut set of $G$ (i.e. the removal of those vertices separates $G$ into two vertex-disjoint graphs). We reduce the weights of all edges of $G$ lying on $\ell$ by half and keep the other edge-weights unchanged. Next, we color two vertex classes of $G$ by black and white, without loss of generality, assume that $a_1$ is black. Finally, we remove all edges on the left of $\ell$ which are adjacent to a black $a_i$ or a white $b_j$; we also remove the edges  on the right of $\ell$ which are adjacent to a white $a_i$ or a black $b_j$. This way, $G$ is divided into two disjoint weighted graphs $G^+$ and $G^-$ (on the left and on the right of $\ell$, respectively). Then
\begin{equation}
\M(G)=2^{k}\M(G^+)\M(G^-).
\end{equation}
\end{lem}

\begin{figure}\centering
\setlength{\unitlength}{3947sp}%
\begingroup\makeatletter\ifx\SetFigFont\undefined%
\gdef\SetFigFont#1#2#3#4#5{%
  \reset@font\fontsize{#1}{#2pt}%
  \fontfamily{#3}\fontseries{#4}\fontshape{#5}%
  \selectfont}%
\fi\endgroup%
\resizebox{12cm}{!}{
\begin{picture}(0,0)%
\includegraphics{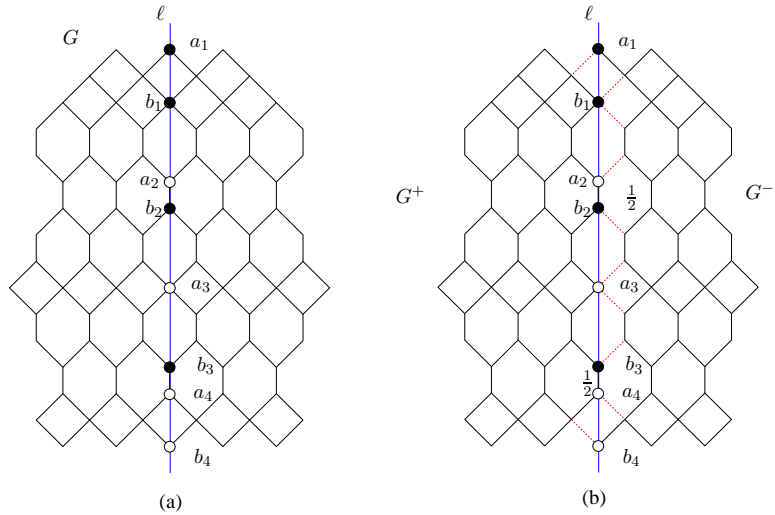}%
\end{picture}%

\begin{picture}(8018,4626)(-296,-4264)
%  METADATA <id>91</id>
\put(6496,-1478){\makebox(0,0)[rb]{\smash{{\SetFigFont{11}{13.2}{\familydefault}{\mddefault}{\updefault}{$\frac{1}{2}$}%
}}}}
%  METADATA <id>92</id>
\put(6097,162){\makebox(0,0)[rb]{\smash{{\SetFigFont{11}{13.2}{\familydefault}{\mddefault}{\updefault}{$\ell$}%
}}}}
%  METADATA <id>93</id>
\put(6496,-75){\makebox(0,0)[rb]{\smash{{\SetFigFont{11}{13.2}{\familydefault}{\mddefault}{\updefault}{$a_1$}%
}}}}
%  METADATA <id>94</id>
\put(6085,-618){\makebox(0,0)[rb]{\smash{{\SetFigFont{11}{13.2}{\familydefault}{\mddefault}{\updefault}{$b_1$}%
}}}}
%  METADATA <id>95</id>
\put(6055,-1309){\makebox(0,0)[rb]{\smash{{\SetFigFont{11}{13.2}{\familydefault}{\mddefault}{\updefault}{$a_2$}%
}}}}
%  METADATA <id>96</id>
\put(6085,-1571){\makebox(0,0)[rb]{\smash{{\SetFigFont{11}{13.2}{\familydefault}{\mddefault}{\updefault}{$b_2$}%
}}}}
%  METADATA <id>97</id>
\put(6508,-2241){\makebox(0,0)[rb]{\smash{{\SetFigFont{11}{13.2}{\familydefault}{\mddefault}{\updefault}{$a_3$}%
}}}}
%  METADATA <id>98</id>
\put(6542,-2970){\makebox(0,0)[rb]{\smash{{\SetFigFont{11}{13.2}{\familydefault}{\mddefault}{\updefault}{$b_3$}%
}}}}
%  METADATA <id>99</id>
\put(6529,-3207){\makebox(0,0)[rb]{\smash{{\SetFigFont{11}{13.2}{\familydefault}{\mddefault}{\updefault}{$a_4$}%
}}}}
%  METADATA <id>100</id>
\put(6517,-3792){\makebox(0,0)[rb]{\smash{{\SetFigFont{11}{13.2}{\familydefault}{\mddefault}{\updefault}{$b_4$}%
}}}}
%  METADATA <id>101</id>
\put(4610,-1474){\makebox(0,0)[rb]{\smash{{\SetFigFont{11}{13.2}{\familydefault}{\mddefault}{\updefault}{$G^+$}%
}}}}
%  METADATA <id>102</id>
\put(7707,-1440){\makebox(0,0)[rb]{\smash{{\SetFigFont{11}{13.2}{\familydefault}{\mddefault}{\updefault}{$G^-$}%
}}}}
%  METADATA <id>104</id>
\put(6101,-3131){\makebox(0,0)[rb]{\smash{{\SetFigFont{11}{13.2}{\familydefault}{\mddefault}{\updefault}{$\frac{1}{2}$}%
}}}}
%  METADATA <id>263</id>
\put(2303,156){\makebox(0,0)[rb]{\smash{{\SetFigFont{11}{13.2}{\familydefault}{\mddefault}{\updefault}{$\ell$}%
}}}}
%  METADATA <id>264</id>
\put(2702,-81){\makebox(0,0)[rb]{\smash{{\SetFigFont{11}{13.2}{\familydefault}{\mddefault}{\updefault}{$a_1$}%
}}}}
%  METADATA <id>265</id>
\put(2291,-624){\makebox(0,0)[rb]{\smash{{\SetFigFont{11}{13.2}{\familydefault}{\mddefault}{\updefault}{$b_1$}%
}}}}
%  METADATA <id>266</id>
\put(2261,-1315){\makebox(0,0)[rb]{\smash{{\SetFigFont{11}{13.2}{\familydefault}{\mddefault}{\updefault}{$a_2$}%
}}}}
%  METADATA <id>267</id>
\put(2291,-1577){\makebox(0,0)[rb]{\smash{{\SetFigFont{11}{13.2}{\familydefault}{\mddefault}{\updefault}{$b_2$}%
}}}}
%  METADATA <id>268</id>
\put(2714,-2247){\makebox(0,0)[rb]{\smash{{\SetFigFont{11}{13.2}{\familydefault}{\mddefault}{\updefault}{$a_3$}%
}}}}
%  METADATA <id>269</id>
\put(2748,-2976){\makebox(0,0)[rb]{\smash{{\SetFigFont{11}{13.2}{\familydefault}{\mddefault}{\updefault}{$b_3$}%
}}}}
%  METADATA <id>270</id>
\put(2735,-3213){\makebox(0,0)[rb]{\smash{{\SetFigFont{11}{13.2}{\familydefault}{\mddefault}{\updefault}{$a_4$}%
}}}}
%  METADATA <id>271</id>
\put(2723,-3798){\makebox(0,0)[rb]{\smash{{\SetFigFont{11}{13.2}{\familydefault}{\mddefault}{\updefault}{$b_4$}%
}}}}
%  METADATA <id>272</id>
%\put(-20,-1494){\makebox(0,0)[rb]{\smash{{\SetFigFont{11}{13.2}{\familydefault}{\mddefault}{\updefault}{$G^+$}%
%}}}}
%  METADATA <id>275</id>
\put(1546,-61){\makebox(0,0)[rb]{\smash{{\SetFigFont{11}{13.2}{\familydefault}{\mddefault}{\updefault}{$G$}%
}}}}
\end{picture}}
\caption{Ciucu's Factorization Theorem. The edges cut off are illustrated by dotted edges.}\label{Figurefactor}
\end{figure}
\section{Proof of the main theorems}\label{Mainproofs}

\begin{figure}
  \centering
  \setlength{\unitlength}{3947sp}%
\begingroup\makeatletter\ifx\SetFigFont\undefined%
\gdef\SetFigFont#1#2#3#4#5{%
  \reset@font\fontsize{#1}{#2pt}%
  \fontfamily{#3}\fontseries{#4}\fontshape{#5}%
  \selectfont}%
\fi\endgroup%
\resizebox{10cm}{!}{
\begin{picture}(0,0)%
\includegraphics{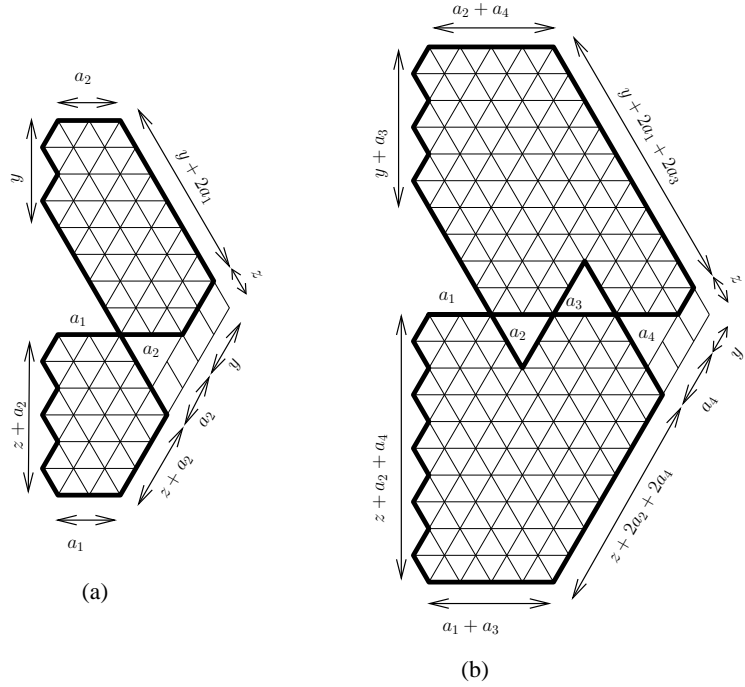}%
\end{picture}%

\begin{picture}(9450,8745)(5183,-9312)
%  METADATA <id>216</id> 
\put(6043,-4748){\makebox(0,0)[lb]{\smash{{\SetFigFont{14}{16.8}{\rmdefault}{\mddefault}{\itdefault}{$a_1$}%
}}}}
%  METADATA <id>217</id> 
\put(6958,-5138){\makebox(0,0)[lb]{\smash{{\SetFigFont{14}{16.8}{\rmdefault}{\mddefault}{\itdefault}{$a_2$}%
}}}}
%  METADATA <id>218</id> 
\put(5458,-6338){\rotatebox{90.0}{\makebox(0,0)[lb]{\smash{{\SetFigFont{14}{16.8}{\rmdefault}{\mddefault}{\itdefault}{$z+a_2$}%
}}}}}
%  METADATA <id>219</id> 
\put(5413,-2978){\rotatebox{90.0}{\makebox(0,0)[lb]{\smash{{\SetFigFont{14}{16.8}{\rmdefault}{\mddefault}{\itdefault}{$y$}%
}}}}}
%  METADATA <id>220</id> 
\put(6103,-1688){\makebox(0,0)[lb]{\smash{{\SetFigFont{14}{16.8}{\rmdefault}{\mddefault}{\itdefault}{$a_2$}%
}}}}
%  METADATA <id>221</id> 
\put(7393,-2603){\rotatebox{300.0}{\makebox(0,0)[lb]{\smash{{\SetFigFont{14}{16.8}{\rmdefault}{\mddefault}{\itdefault}{$y+2a_1$}%
}}}}}
%  METADATA <id>222</id> 
\put(8308,-4103){\rotatebox{300.0}{\makebox(0,0)[lb]{\smash{{\SetFigFont{14}{16.8}{\rmdefault}{\mddefault}{\itdefault}{$z$}%
}}}}}
%  METADATA <id>223</id> 
\put(8113,-5333){\rotatebox{60.0}{\makebox(0,0)[lb]{\smash{{\SetFigFont{14}{16.8}{\rmdefault}{\mddefault}{\itdefault}{$y$}%
}}}}}
%  METADATA <id>224</id> 
\put(7738,-6068){\rotatebox{60.0}{\makebox(0,0)[lb]{\smash{{\SetFigFont{14}{16.8}{\rmdefault}{\mddefault}{\itdefault}{$a_2$}%
}}}}}
%  METADATA <id>225</id> 
\put(7258,-6968){\rotatebox{60.0}{\makebox(0,0)[lb]{\smash{{\SetFigFont{14}{16.8}{\rmdefault}{\mddefault}{\itdefault}{$z+a_2$}%
}}}}}
%  METADATA <id>226</id> 
\put(6013,-7583){\makebox(0,0)[lb]{\smash{{\SetFigFont{14}{16.8}{\rmdefault}{\mddefault}{\itdefault}{$a_1$}%
}}}}
%  METADATA <id>231</id> 
\put(13155,-4936){\makebox(0,0)[lb]{\smash{{\SetFigFont{14}{16.8}{\rmdefault}{\mddefault}{\itdefault}{$a_4$}%
}}}}
%  METADATA <id>240</id> 
\put(9975,-2911){\rotatebox{90.0}{\makebox(0,0)[lb]{\smash{{\SetFigFont{14}{16.8}{\rmdefault}{\mddefault}{\itdefault}{$y+a_3$}%
}}}}}
%  METADATA <id>241</id> 
\put(10703,-8626){\makebox(0,0)[lb]{\smash{{\SetFigFont{14}{16.8}{\rmdefault}{\mddefault}{\itdefault}{$a_1+a_3$}%
}}}}
%  METADATA <id>242</id> 
\put(14400,-5191){\rotatebox{60.0}{\makebox(0,0)[lb]{\smash{{\SetFigFont{14}{16.8}{\rmdefault}{\mddefault}{\itdefault}{$y$}%
}}}}}
%  METADATA <id>243</id> 
\put(14302,-4165){\rotatebox{300.0}{\makebox(0,0)[lb]{\smash{{\SetFigFont{14}{16.8}{\rmdefault}{\mddefault}{\itdefault}{$z$}%
}}}}}
%  METADATA <id>244</id> 
\put(12930,-1816){\rotatebox{300.0}{\makebox(0,0)[lb]{\smash{{\SetFigFont{14}{16.8}{\rmdefault}{\mddefault}{\itdefault}{$y+2a_1+2a_3$}%
}}}}}
%  METADATA <id>245</id> 
\put(10830,-841){\makebox(0,0)[lb]{\smash{{\SetFigFont{14}{16.8}{\rmdefault}{\mddefault}{\itdefault}{$a_2+a_4$}%
}}}}
%  METADATA <id>246</id> 
\put(14025,-5881){\rotatebox{60.0}{\makebox(0,0)[lb]{\smash{{\SetFigFont{14}{16.8}{\rmdefault}{\mddefault}{\itdefault}{$a_4$}%
}}}}}
%  METADATA <id>247</id> 
\put(9975,-7261){\rotatebox{90.0}{\makebox(0,0)[lb]{\smash{{\SetFigFont{14}{16.8}{\rmdefault}{\mddefault}{\itdefault}{$z+a_2+a_4$}%
}}}}}
%  METADATA <id>248</id> 
\put(10658,-4456){\makebox(0,0)[lb]{\smash{{\SetFigFont{14}{16.8}{\rmdefault}{\mddefault}{\itdefault}{$a_1$}%
}}}}
%  METADATA <id>249</id> 
\put(11543,-4921){\makebox(0,0)[lb]{\smash{{\SetFigFont{14}{16.8}{\rmdefault}{\mddefault}{\itdefault}{$a_2$}%
}}}}
%  METADATA <id>250</id> 
\put(12248,-4531){\makebox(0,0)[lb]{\smash{{\SetFigFont{14}{16.8}{\rmdefault}{\mddefault}{\itdefault}{$a_3$}%
}}}}
%  METADATA <id>254</id> 
\put(12885,-7786){\rotatebox{60.0}{\makebox(0,0)[lb]{\smash{{\SetFigFont{14}{16.8}{\rmdefault}{\mddefault}{\itdefault}{$z+2a_2+2a_4$}%
}}}}}
\end{picture}}
  \caption{Splitting up a $H^{(1)}$-type region into two $Q$-type regions in the case of $x=0$.}\label{Halfhex2basex}
\end{figure}

\begin{figure}\centering
\setlength{\unitlength}{3947sp}%
\begingroup\makeatletter\ifx\SetFigFont\undefined%
\gdef\SetFigFont#1#2#3#4#5{%
  \reset@font\fontsize{#1}{#2pt}%
  \fontfamily{#3}\fontseries{#4}\fontshape{#5}%
  \selectfont}%
\fi\endgroup%
\resizebox{10cm}{!}{
\begin{picture}(0,0)%
\includegraphics{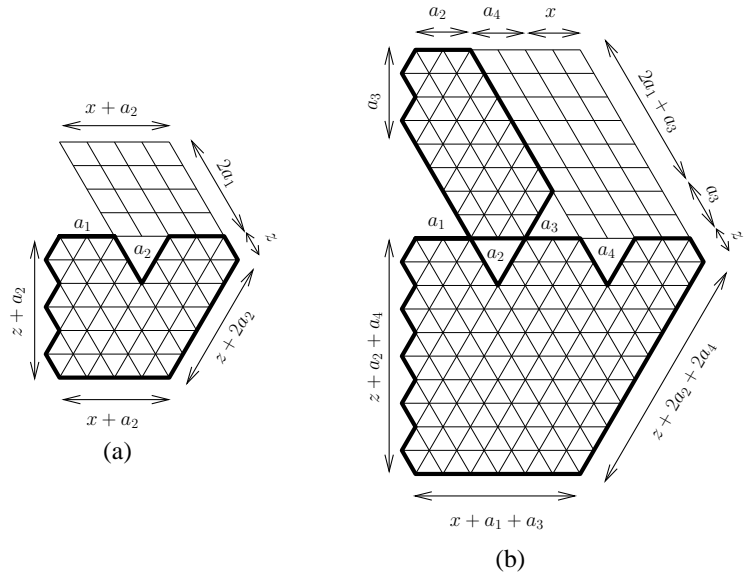}%
\end{picture}%

\begin{picture}(10741,8314)(4625,-8386)
%  METADATA <id>264</id> 
\put(5830,-1770){\makebox(0,0)[lb]{\smash{{\SetFigFont{17}{20.4}{\rmdefault}{\mddefault}{\itdefault}{$x+a_2$}%
}}}}
%  METADATA <id>265</id> 
\put(5665,-3420){\makebox(0,0)[lb]{\smash{{\SetFigFont{17}{20.4}{\rmdefault}{\mddefault}{\itdefault}{$a_1$}%
}}}}
%  METADATA <id>266</id> 
\put(6520,-3795){\makebox(0,0)[lb]{\smash{{\SetFigFont{17}{20.4}{\rmdefault}{\mddefault}{\itdefault}{$a_2$}%
}}}}
%  METADATA <id>267</id> 
\put(5845,-6240){\makebox(0,0)[lb]{\smash{{\SetFigFont{17}{20.4}{\rmdefault}{\mddefault}{\itdefault}{$x+a_2$}%
}}}}
%  METADATA <id>268</id> 
\put(4930,-4980){\rotatebox{90.0}{\makebox(0,0)[lb]{\smash{{\SetFigFont{17}{20.4}{\rmdefault}{\mddefault}{\itdefault}{$z+a_2$}%
}}}}}
%  METADATA <id>269</id> 
\put(7780,-5385){\rotatebox{60.0}{\makebox(0,0)[lb]{\smash{{\SetFigFont{17}{20.4}{\rmdefault}{\mddefault}{\itdefault}{$z+2a_2$}%
}}}}}
%  METADATA <id>270</id> 
\put(8380,-3510){\rotatebox{300.0}{\makebox(0,0)[lb]{\smash{{\SetFigFont{17}{20.4}{\rmdefault}{\mddefault}{\itdefault}{$z$}%
}}}}}
%  METADATA <id>271</id> 
\put(7690,-2505){\rotatebox{300.0}{\makebox(0,0)[lb]{\smash{{\SetFigFont{17}{20.4}{\rmdefault}{\mddefault}{\itdefault}{$2a_1$}%
}}}}}
%  METADATA <id>284</id> 
\put(9976,-1726){\rotatebox{90.0}{\makebox(0,0)[lb]{\smash{{\SetFigFont{17}{20.4}{\rmdefault}{\mddefault}{\itdefault}{$a_3$}%
}}}}}
%  METADATA <id>285</id> 
\put(10711,-376){\makebox(0,0)[lb]{\smash{{\SetFigFont{17}{20.4}{\rmdefault}{\mddefault}{\itdefault}{$a_2$}%
}}}}
%  METADATA <id>287</id> 
\put(10696,-3391){\makebox(0,0)[lb]{\smash{{\SetFigFont{17}{20.4}{\rmdefault}{\mddefault}{\itdefault}{$a_1$}%
}}}}
%  METADATA <id>288</id> 
\put(11536,-3841){\makebox(0,0)[lb]{\smash{{\SetFigFont{17}{20.4}{\rmdefault}{\mddefault}{\itdefault}{$a_2$}%
}}}}
%  METADATA <id>289</id> 
\put(12301,-3436){\makebox(0,0)[lb]{\smash{{\SetFigFont{17}{20.4}{\rmdefault}{\mddefault}{\itdefault}{$a_3$}%
}}}}
%  METADATA <id>290</id> 
\put(13111,-3796){\makebox(0,0)[lb]{\smash{{\SetFigFont{17}{20.4}{\rmdefault}{\mddefault}{\itdefault}{$a_4$}%
}}}}
%  METADATA <id>291</id> 
\put(13681,-1216){\rotatebox{300.0}{\makebox(0,0)[lb]{\smash{{\SetFigFont{17}{20.4}{\rmdefault}{\mddefault}{\itdefault}{$2a_1+a_3$}%
}}}}}
%  METADATA <id>292</id> 
\put(14641,-2836){\rotatebox{300.0}{\makebox(0,0)[lb]{\smash{{\SetFigFont{17}{20.4}{\rmdefault}{\mddefault}{\itdefault}{$a_3$}%
}}}}}
%  METADATA <id>293</id> 
\put(15001,-3481){\rotatebox{300.0}{\makebox(0,0)[lb]{\smash{{\SetFigFont{17}{20.4}{\rmdefault}{\mddefault}{\itdefault}{$z$}%
}}}}}
%  METADATA <id>294</id> 
\put(13981,-6466){\rotatebox{60.0}{\makebox(0,0)[lb]{\smash{{\SetFigFont{17}{20.4}{\rmdefault}{\mddefault}{\itdefault}{$z+2a_2+2a_4$}%
}}}}}
%  METADATA <id>295</id> 
\put(10006,-5926){\rotatebox{90.0}{\makebox(0,0)[lb]{\smash{{\SetFigFont{17}{20.4}{\rmdefault}{\mddefault}{\itdefault}{$z+a_2+a_4$}%
}}}}}
%  METADATA <id>296</id> 
\put(11011,-7681){\makebox(0,0)[lb]{\smash{{\SetFigFont{17}{20.4}{\rmdefault}{\mddefault}{\itdefault}{$x+a_1+a_3$}%
}}}}
%  METADATA <id>304</id> 
\put(12361,-399){\makebox(0,0)[lb]{\smash{{\SetFigFont{17}{20.4}{\rmdefault}{\mddefault}{\itdefault}{$x$}%
}}}}
%  METADATA <id>306</id> 
\put(11484,-384){\makebox(0,0)[lb]{\smash{{\SetFigFont{17}{20.4}{\rmdefault}{\mddefault}{\itdefault}{$a_4$}%
}}}}
\end{picture}}
\caption{Partitioning  a $H^{(1)}$-type region into two $Q$-type regions in the case of $y=0$.}\label{Halfhex2basey}
\end{figure}

\begin{figure}
  \centering
  \setlength{\unitlength}{3947sp}%
\begingroup\makeatletter\ifx\SetFigFont\undefined%
\gdef\SetFigFont#1#2#3#4#5{%
  \reset@font\fontsize{#1}{#2pt}%
  \fontfamily{#3}\fontseries{#4}\fontshape{#5}%
  \selectfont}%
\fi\endgroup%
\resizebox{10cm}{!}{
\begin{picture}(0,0)%
\includegraphics{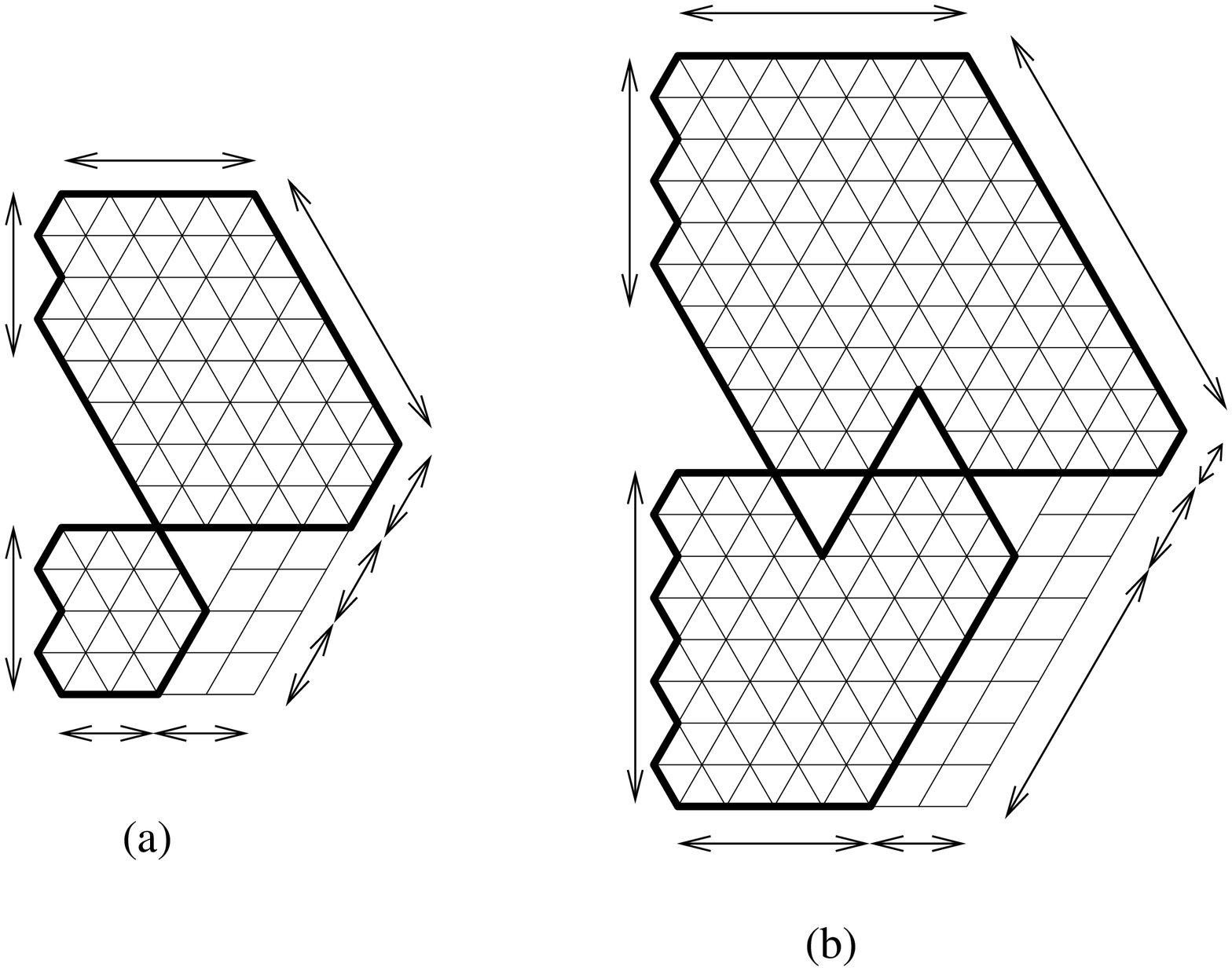}%
\end{picture}%
%
%  Created by WinFIG version 6.2 
%  METADATA <version>1.0</version> 
%

\begin{picture}(10580,8225)(5976,-8761)
%  METADATA <id>448</id> 
\put(6956,-5011){\makebox(0,0)[lb]{\smash{{\SetFigFont{17}{20.4}{\rmdefault}{\mddefault}{\itdefault}{$a_1$}%
}}}}
%  METADATA <id>449</id> 
\put(7856,-5401){\makebox(0,0)[lb]{\smash{{\SetFigFont{17}{20.4}{\rmdefault}{\mddefault}{\itdefault}{$a_2$}%
}}}}
%  METADATA <id>451</id> 
\put(14377,-4965){\makebox(0,0)[lb]{\smash{{\SetFigFont{17}{20.4}{\rmdefault}{\mddefault}{\itdefault}{$a_4$}%
}}}}
%  METADATA <id>452</id> 
\put(6941,-7141){\makebox(0,0)[lb]{\smash{{\SetFigFont{17}{20.4}{\rmdefault}{\mddefault}{\itdefault}{$a_1$}%
}}}}
%  METADATA <id>453</id> 
\put(7796,-7156){\makebox(0,0)[lb]{\smash{{\SetFigFont{17}{20.4}{\rmdefault}{\mddefault}{\itdefault}{$x$}%
}}}}
%  METADATA <id>454</id> 
\put(6281,-6166){\rotatebox{90.0}{\makebox(0,0)[lb]{\smash{{\SetFigFont{17}{20.4}{\rmdefault}{\mddefault}{\itdefault}{$a_2$}%
}}}}}
%  METADATA <id>456</id> 
\put(11317,-6690){\rotatebox{90.0}{\makebox(0,0)[lb]{\smash{{\SetFigFont{17}{20.4}{\rmdefault}{\mddefault}{\itdefault}{$a_2+a_4$}%
}}}}}
%  METADATA <id>457</id> 
\put(12082,-8040){\makebox(0,0)[lb]{\smash{{\SetFigFont{17}{20.4}{\rmdefault}{\mddefault}{\itdefault}{$a_1+a_3$}%
}}}}
%  METADATA <id>458</id> 
\put(13552,-8025){\makebox(0,0)[lb]{\smash{{\SetFigFont{17}{20.4}{\rmdefault}{\mddefault}{\itdefault}{$x$}%
}}}}
%  METADATA <id>461</id> 
\put(11941,-4561){\makebox(0,0)[lb]{\smash{{\SetFigFont{17}{20.4}{\rmdefault}{\mddefault}{\itdefault}{$a_1$}%
}}}}
%  METADATA <id>463</id> 
\put(13552,-4620){\makebox(0,0)[lb]{\smash{{\SetFigFont{17}{20.4}{\rmdefault}{\mddefault}{\itdefault}{$a_3$}%
}}}}
%  METADATA <id>465</id> 
\put(12832,-4965){\makebox(0,0)[lb]{\smash{{\SetFigFont{17}{20.4}{\rmdefault}{\mddefault}{\itdefault}{$a_2$}%
}}}}
%  METADATA <id>466</id> 
\put(14956,-7366){\rotatebox{60.0}{\makebox(0,0)[lb]{\smash{{\SetFigFont{17}{20.4}{\rmdefault}{\mddefault}{\itdefault}{$2a_2+a_4$}%
}}}}}
%  METADATA <id>467</id> 
\put(15976,-5596){\rotatebox{60.0}{\makebox(0,0)[lb]{\smash{{\SetFigFont{17}{20.4}{\rmdefault}{\mddefault}{\itdefault}{$a_4$}%
}}}}}
%  METADATA <id>468</id> 
\put(16366,-4906){\rotatebox{60.0}{\makebox(0,0)[lb]{\smash{{\SetFigFont{17}{20.4}{\rmdefault}{\mddefault}{\itdefault}{$y$}%
}}}}}
%  METADATA <id>469</id> 
\put(11257,-2805){\rotatebox{90.0}{\makebox(0,0)[lb]{\smash{{\SetFigFont{17}{20.4}{\rmdefault}{\mddefault}{\itdefault}{$y+a_3$}%
}}}}}
%  METADATA <id>471</id> 
\put(6251,-3241){\rotatebox{90.0}{\makebox(0,0)[lb]{\smash{{\SetFigFont{17}{20.4}{\rmdefault}{\mddefault}{\itdefault}{$y$}%
}}}}}
%  METADATA <id>473</id> 
\put(9071,-2776){\rotatebox{300.0}{\makebox(0,0)[lb]{\smash{{\SetFigFont{17}{20.4}{\rmdefault}{\mddefault}{\itdefault}{$y+2a_1$}%
}}}}}
%  METADATA <id>474</id> 
\put(7316,-2056){\makebox(0,0)[lb]{\smash{{\SetFigFont{17}{20.4}{\rmdefault}{\mddefault}{\itdefault}{$x+a_2$}%
}}}}
%  METADATA <id>476</id> 
\put(12307,-840){\makebox(0,0)[lb]{\smash{{\SetFigFont{17}{20.4}{\rmdefault}{\mddefault}{\itdefault}{$x+a_2+a_4$}%
}}}}
%  METADATA <id>477</id> 
\put(15127,-1860){\rotatebox{300.0}{\makebox(0,0)[lb]{\smash{{\SetFigFont{17}{20.4}{\rmdefault}{\mddefault}{\itdefault}{$y+2a_1+2a_3$}%
}}}}}
%  METADATA <id>481</id> 
\put(8981,-6661){\rotatebox{60.0}{\makebox(0,0)[lb]{\smash{{\SetFigFont{17}{20.4}{\rmdefault}{\mddefault}{\itdefault}{$a_2$}%
}}}}}
%  METADATA <id>482</id> 
\put(9386,-6001){\rotatebox{60.0}{\makebox(0,0)[lb]{\smash{{\SetFigFont{17}{20.4}{\rmdefault}{\mddefault}{\itdefault}{$a_2$}%
}}}}}
%  METADATA <id>483</id> 
\put(9806,-5191){\rotatebox{60.0}{\makebox(0,0)[lb]{\smash{{\SetFigFont{17}{20.4}{\rmdefault}{\mddefault}{\itdefault}{$y$}%
}}}}}
\end{picture}}
  \caption{The $H^{(1)}$-type region can be divided into two $Q$-type subregions when $z=0$.}\label{Halfhex2basez}
\end{figure}

We first prove Theorem \ref{main1}.

\begin{proof}[Proof of Theorem \ref{main1}]
We first prove (\ref{main1eq1}) by induction on $x+y+z$. The base cases are the situations when at least one of the parameters $x,y,z$  is equal to $0$.

If $x=0$, then we split the region $H^{(1)}_{0,y,z}(a,b)$ into two (balanced) subregions along the right  side of the $a$-hole: the lower subregion is the halved hexagon $\mathcal{P}_{b+z,b+z,a}$,  and the upper subregion is a union of the halved hexagon $\mathcal{P}_{y,y+a_1,a_2}$ and several forced vertical lozenges (see Figure \ref{Halfhex2basex}(a)). By Region-splitting Lemma \ref{RS}, we get
\begin{equation}
\M(H^{(1)}_{0,y,z}(a,b))=\M(\mathcal{P}_{b+z,b+z,a})\M(\mathcal{P}_{y,y+a,b}),
\end{equation}
and (\ref{main1eq1}) follows from Proctor's Theorem \ref{Proctiling}.

If $y=0$, then our region has several forced lozenges on the top as in Figure \ref{Halfhex2basey}(a). By removing these forced lozenges, we obtained an upside-down region $\mathcal{Q}(a,b,x,z)$. Then (\ref{main1eq1}) follows from Lemma \ref{QAR} in this case.

If $z=0$, similar to the case when $x=0$, we also split the region into two halved hexagons and several forced lozenges as in Figure \ref{Halfhex2basez}(a). Thus, our formula is also implied by  Proctor's Theorem \ref{Proctiling}.

\medskip

\begin{figure}\centering
\setlength{\unitlength}{3947sp}%
\begingroup\makeatletter\ifx\SetFigFont\undefined%
\gdef\SetFigFont#1#2#3#4#5{%
  \reset@font\fontsize{#1}{#2pt}%
  \fontfamily{#3}\fontseries{#4}\fontshape{#5}%
  \selectfont}%
\fi\endgroup%
\resizebox{10cm}{!}{
\begin{picture}(0,0)%
\includegraphics{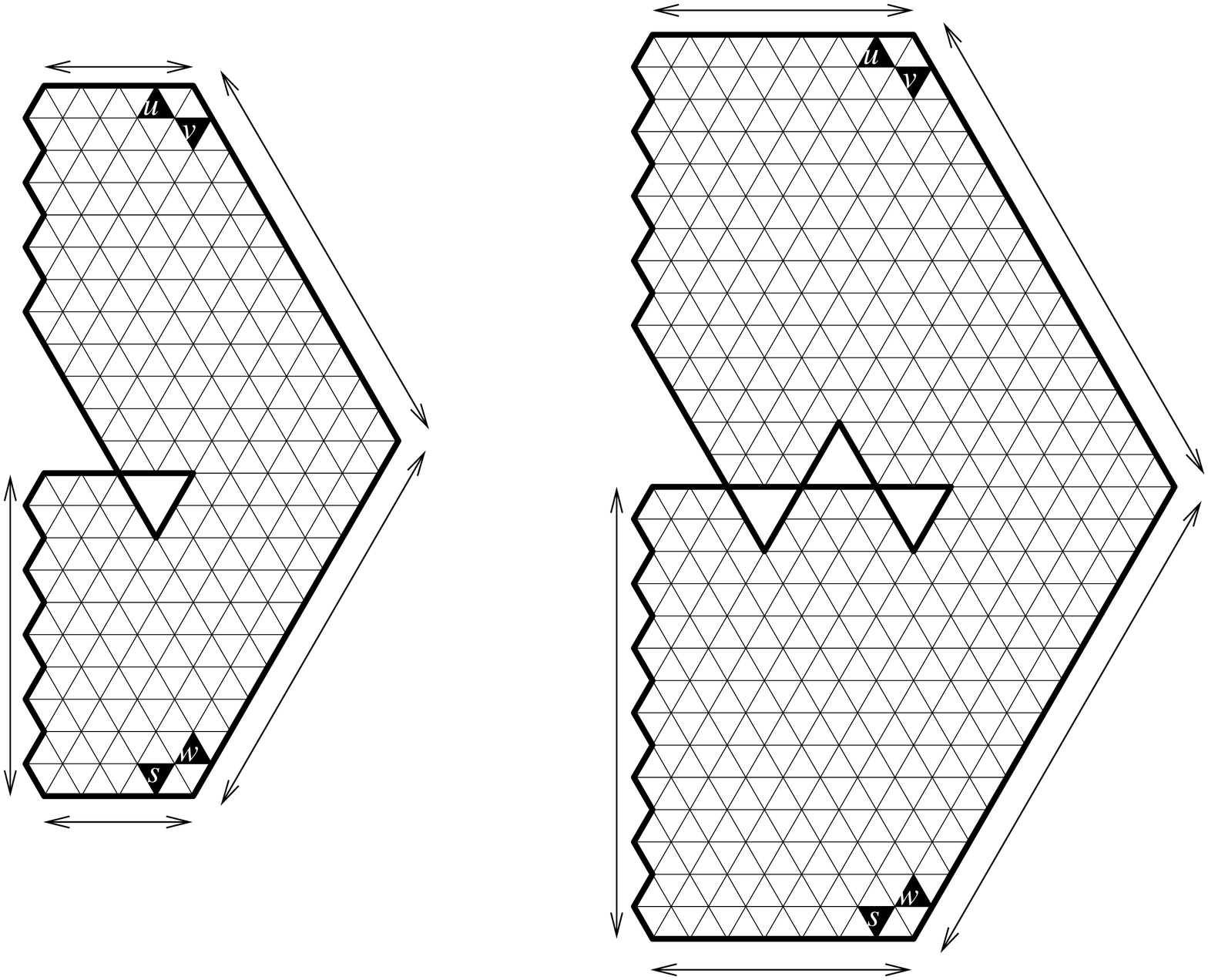}%
\end{picture}%
\begin{picture}(13014,10997)(2616,-10329)
%  METADATA <id>367</id>
\put(3601,-4404){\makebox(0,0)[lb]{\smash{{\SetFigFont{20}{24.0}{\rmdefault}{\mddefault}{\itdefault}{$a$}%
}}}}
%  METADATA <id>368</id>
\put(4546,-5094){\makebox(0,0)[lb]{\smash{{\SetFigFont{20}{24.0}{\rmdefault}{\mddefault}{\itdefault}{$b$}%
}}}}
%  METADATA <id>369</id>
\put(4216,-346){\makebox(0,0)[lb]{\smash{{\SetFigFont{20}{24.0}{\rmdefault}{\mddefault}{\itdefault}{$x+b$}%
}}}}
%  METADATA <id>370</id>
\put(3939,-8656){\makebox(0,0)[lb]{\smash{{\SetFigFont{20}{24.0}{\rmdefault}{\mddefault}{\itdefault}{$x+a$}%
}}}}
%  METADATA <id>371</id>
\put(6109,-1539){\rotatebox{300.0}{\makebox(0,0)[lb]{\smash{{\SetFigFont{20}{24.0}{\rmdefault}{\mddefault}{\itdefault}{$y+z+2a$}%
}}}}}
%  METADATA <id>372</id>
\put(6204,-7456){\rotatebox{60.0}{\makebox(0,0)[lb]{\smash{{\SetFigFont{20}{24.0}{\rmdefault}{\mddefault}{\itdefault}{$y+z+2b$}%
}}}}}
%  METADATA <id>373</id>
\put(2971,-7089){\rotatebox{90.0}{\makebox(0,0)[lb]{\smash{{\SetFigFont{20}{24.0}{\rmdefault}{\mddefault}{\itdefault}{$2z+2b$}%
}}}}}
%  METADATA <id>374</id>
\put(9931,-4704){\makebox(0,0)[lb]{\smash{{\SetFigFont{20}{24.0}{\rmdefault}{\mddefault}{\itdefault}{$a_1$}%
}}}}
%  METADATA <id>375</id>
\put(10839,-5161){\makebox(0,0)[lb]{\smash{{\SetFigFont{20}{24.0}{\rmdefault}{\mddefault}{\itdefault}{$a_2$}%
}}}}
%  METADATA <id>376</id>
\put(11574,-4764){\makebox(0,0)[lb]{\smash{{\SetFigFont{20}{24.0}{\rmdefault}{\mddefault}{\itdefault}{$a_3$}%
}}}}
%  METADATA <id>377</id>
\put(12384,-5154){\makebox(0,0)[lb]{\smash{{\SetFigFont{20}{24.0}{\rmdefault}{\mddefault}{\itdefault}{$a_4$}%
}}}}
%  METADATA <id>378</id>
\put(10600,314){\makebox(0,0)[lb]{\smash{{\SetFigFont{20}{24.0}{\rmdefault}{\mddefault}{\itdefault}{$x+a_2+a_4$}%
}}}}
%  METADATA <id>379</id>
\put(13756,-1216){\rotatebox{300.0}{\makebox(0,0)[lb]{\smash{{\SetFigFont{20}{24.0}{\rmdefault}{\mddefault}{\itdefault}{$y+z+2a_1+2a_3$}%
}}}}}
%  METADATA <id>380</id>
\put(13794,-8829){\rotatebox{60.0}{\makebox(0,0)[lb]{\smash{{\SetFigFont{20}{24.0}{\rmdefault}{\mddefault}{\itdefault}{$y+z+2a_2+2a_4$}%
}}}}}
%  METADATA <id>381</id>
\put(10291,-10314){\makebox(0,0)[lb]{\smash{{\SetFigFont{20}{24.0}{\rmdefault}{\mddefault}{\itdefault}{$x+a_1+a_3$}%
}}}}
%  METADATA <id>382</id>
\put(9309,-8416){\rotatebox{90.0}{\makebox(0,0)[lb]{\smash{{\SetFigFont{20}{24.0}{\rmdefault}{\mddefault}{\itdefault}{$2z+2a_2+2a_4$}%
}}}}}
\end{picture}}
\caption{How to apply Kuo condensation to a $H^{(1)}$-region.}\label{MidarrayKuo1}
\end{figure}

\begin{figure}
  \centering
  \setlength{\unitlength}{3947sp}%
\begingroup\makeatletter\ifx\SetFigFont\undefined%
\gdef\SetFigFont#1#2#3#4#5{%
  \reset@font\fontsize{#1}{#2pt}%
  \fontfamily{#3}\fontseries{#4}\fontshape{#5}%
  \selectfont}%
\fi\endgroup%
\resizebox{13cm}{!}{
\begin{picture}(0,0)%
\includegraphics{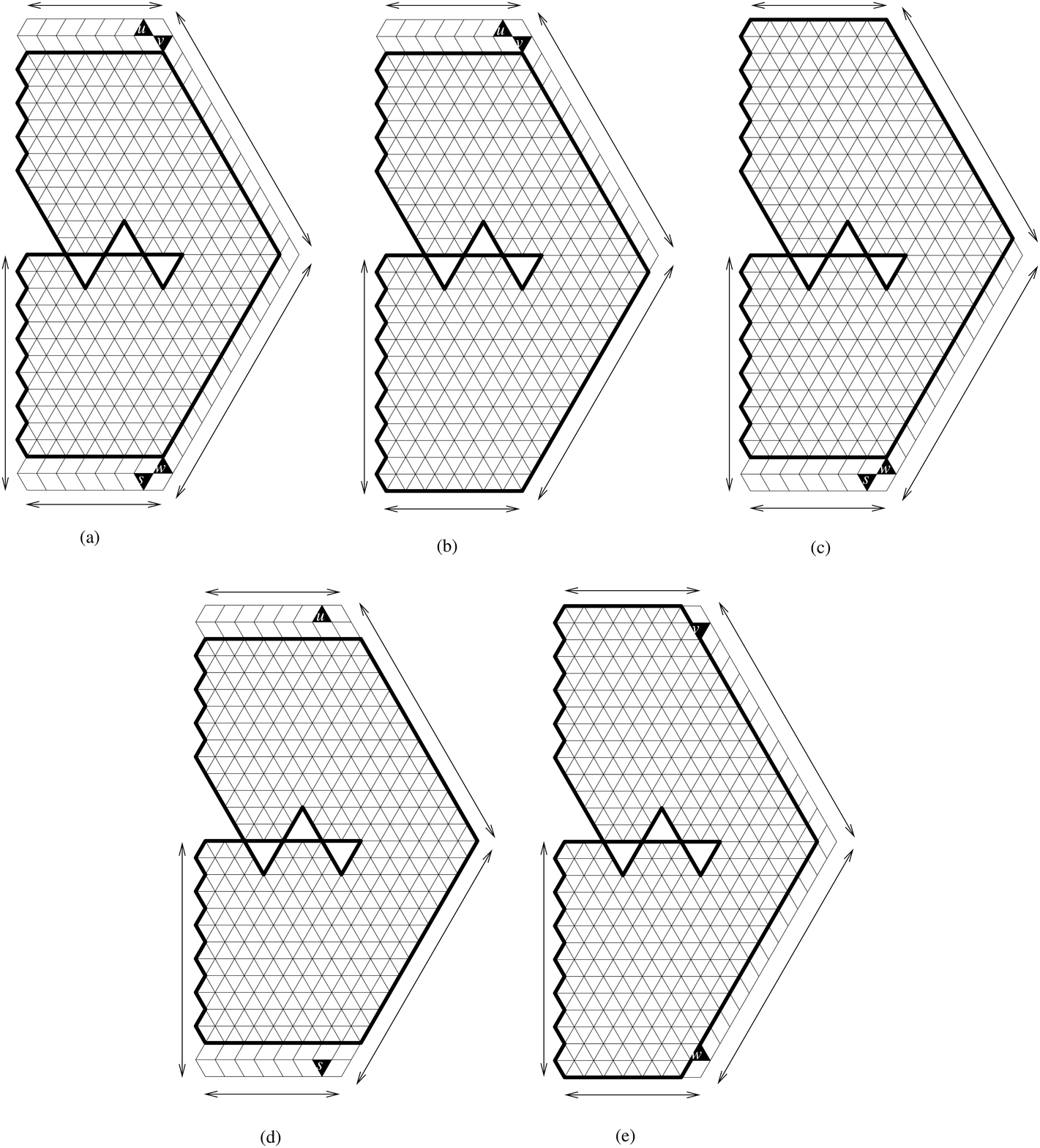}%
\end{picture}%

\begin{picture}(21196,23409)(3250,-22730)
%  METADATA <id>1560</id>
\put(4696,329){\makebox(0,0)[lb]{\smash{{\SetFigFont{20}{24.0}{\rmdefault}{\mddefault}{\itdefault}{$x+a_2+a_4$}%
}}}}
%  METADATA <id>1561</id>
\put(4546,-10171){\makebox(0,0)[lb]{\smash{{\SetFigFont{20}{24.0}{\rmdefault}{\mddefault}{\itdefault}{$x+a_1+a_3$}%
}}}}
%  METADATA <id>1562</id>
\put(7951,-8971){\rotatebox{60.0}{\makebox(0,0)[lb]{\smash{{\SetFigFont{20}{24.0}{\rmdefault}{\mddefault}{\itdefault}{$y+z+2a_2+2a_4$}%
}}}}}
%  METADATA <id>1563</id>
\put(7711,-601){\rotatebox{300.0}{\makebox(0,0)[lb]{\smash{{\SetFigFont{20}{24.0}{\rmdefault}{\mddefault}{\itdefault}{$y+z+2a_1+2a_3$}%
}}}}}
%  METADATA <id>1564</id>
\put(3601,-8536){\rotatebox{90.0}{\makebox(0,0)[lb]{\smash{{\SetFigFont{20}{24.0}{\rmdefault}{\mddefault}{\itdefault}{$2z+2a_2+2a_4$}%
}}}}}
%  METADATA <id>1565</id>
\put(4246,-4681){\makebox(0,0)[lb]{\smash{{\SetFigFont{20}{24.0}{\rmdefault}{\mddefault}{\itdefault}{$a_1$}%
}}}}
%  METADATA <id>1566</id>
\put(5176,-5056){\makebox(0,0)[lb]{\smash{{\SetFigFont{20}{24.0}{\rmdefault}{\mddefault}{\itdefault}{$a_2$}%
}}}}
%  METADATA <id>1567</id>
\put(5941,-4696){\makebox(0,0)[lb]{\smash{{\SetFigFont{20}{24.0}{\rmdefault}{\mddefault}{\itdefault}{$a_3$}%
}}}}
%  METADATA <id>1568</id>
\put(6751,-5056){\makebox(0,0)[lb]{\smash{{\SetFigFont{20}{24.0}{\rmdefault}{\mddefault}{\itdefault}{$a_4$}%
}}}}
%  METADATA <id>1570</id>
\put(11409,-4711){\makebox(0,0)[lb]{\smash{{\SetFigFont{20}{24.0}{\rmdefault}{\mddefault}{\itdefault}{$a_1$}%
}}}}
%  METADATA <id>1571</id>
\put(18684,-4726){\makebox(0,0)[lb]{\smash{{\SetFigFont{20}{24.0}{\rmdefault}{\mddefault}{\itdefault}{$a_1$}%
}}}}
%  METADATA <id>1572</id>
\put(7749,-16441){\makebox(0,0)[lb]{\smash{{\SetFigFont{20}{24.0}{\rmdefault}{\mddefault}{\itdefault}{$a_1$}%
}}}}
%  METADATA <id>1573</id>
\put(14964,-16441){\makebox(0,0)[lb]{\smash{{\SetFigFont{20}{24.0}{\rmdefault}{\mddefault}{\itdefault}{$a_1$}%
}}}}
%  METADATA <id>1575</id>
\put(12354,-5071){\makebox(0,0)[lb]{\smash{{\SetFigFont{20}{24.0}{\rmdefault}{\mddefault}{\itdefault}{$a_2$}%
}}}}
%  METADATA <id>1576</id>
\put(19644,-5071){\makebox(0,0)[lb]{\smash{{\SetFigFont{20}{24.0}{\rmdefault}{\mddefault}{\itdefault}{$a_2$}%
}}}}
%  METADATA <id>1577</id>
\put(15894,-16846){\makebox(0,0)[lb]{\smash{{\SetFigFont{20}{24.0}{\rmdefault}{\mddefault}{\itdefault}{$a_2$}%
}}}}
%  METADATA <id>1578</id>
\put(8709,-16831){\makebox(0,0)[lb]{\smash{{\SetFigFont{20}{24.0}{\rmdefault}{\mddefault}{\itdefault}{$a_2$}%
}}}}
%  METADATA <id>1580</id>
\put(9474,-16471){\makebox(0,0)[lb]{\smash{{\SetFigFont{20}{24.0}{\rmdefault}{\mddefault}{\itdefault}{$a_3$}%
}}}}
%  METADATA <id>1582</id>
\put(10299,-16831){\makebox(0,0)[lb]{\smash{{\SetFigFont{20}{24.0}{\rmdefault}{\mddefault}{\itdefault}{$a_4$}%
}}}}
%  METADATA <id>1583</id>
\put(17499,-16846){\makebox(0,0)[lb]{\smash{{\SetFigFont{20}{24.0}{\rmdefault}{\mddefault}{\itdefault}{$a_4$}%
}}}}
%  METADATA <id>1585</id>
\put(16659,-16471){\makebox(0,0)[lb]{\smash{{\SetFigFont{20}{24.0}{\rmdefault}{\mddefault}{\itdefault}{$a_3$}%
}}}}
%  METADATA <id>1586</id>
\put(13089,-4741){\makebox(0,0)[lb]{\smash{{\SetFigFont{20}{24.0}{\rmdefault}{\mddefault}{\itdefault}{$a_3$}%
}}}}
%  METADATA <id>1587</id>
\put(20409,-4741){\makebox(0,0)[lb]{\smash{{\SetFigFont{20}{24.0}{\rmdefault}{\mddefault}{\itdefault}{$a_3$}%
}}}}
%  METADATA <id>1589</id>
\put(21219,-5086){\makebox(0,0)[lb]{\smash{{\SetFigFont{20}{24.0}{\rmdefault}{\mddefault}{\itdefault}{$a_4$}%
}}}}
%  METADATA <id>1591</id>
\put(10786,-8911){\rotatebox{90.0}{\makebox(0,0)[lb]{\smash{{\SetFigFont{20}{24.0}{\rmdefault}{\mddefault}{\itdefault}{$2z+2a_2+2a_4$}%
}}}}}
%  METADATA <id>1593</id>
\put(18076,-8761){\rotatebox{90.0}{\makebox(0,0)[lb]{\smash{{\SetFigFont{20}{24.0}{\rmdefault}{\mddefault}{\itdefault}{$2z+2a_2+2a_4$}%
}}}}}
%  METADATA <id>1594</id>
\put(7066,-20266){\rotatebox{90.0}{\makebox(0,0)[lb]{\smash{{\SetFigFont{20}{24.0}{\rmdefault}{\mddefault}{\itdefault}{$2z+2a_2+2a_4$}%
}}}}}
%  METADATA <id>1595</id>
\put(14356,-20581){\rotatebox{90.0}{\makebox(0,0)[lb]{\smash{{\SetFigFont{20}{24.0}{\rmdefault}{\mddefault}{\itdefault}{$2z+2a_2+2a_4$}%
}}}}}
%  METADATA <id>1599</id>
\put(11799,-10276){\makebox(0,0)[lb]{\smash{{\SetFigFont{20}{24.0}{\rmdefault}{\mddefault}{\itdefault}{$x+a_1+a_3$}%
}}}}
%  METADATA <id>1600</id>
\put(19029,-10276){\makebox(0,0)[lb]{\smash{{\SetFigFont{20}{24.0}{\rmdefault}{\mddefault}{\itdefault}{$x+a_1+a_3$}%
}}}}
%  METADATA <id>1601</id>
\put(15339,-22036){\makebox(0,0)[lb]{\smash{{\SetFigFont{20}{24.0}{\rmdefault}{\mddefault}{\itdefault}{$x+a_1+a_3$}%
}}}}
%  METADATA <id>1602</id>
\put(8154,-22021){\makebox(0,0)[lb]{\smash{{\SetFigFont{20}{24.0}{\rmdefault}{\mddefault}{\itdefault}{$x+a_1+a_3$}%
}}}}
%  METADATA <id>1604</id>
\put(11829,284){\makebox(0,0)[lb]{\smash{{\SetFigFont{20}{24.0}{\rmdefault}{\mddefault}{\itdefault}{$x+a_2+a_4$}%
}}}}
%  METADATA <id>1605</id>
\put(19104,329){\makebox(0,0)[lb]{\smash{{\SetFigFont{20}{24.0}{\rmdefault}{\mddefault}{\itdefault}{$x+a_2+a_4$}%
}}}}
%  METADATA <id>1606</id>
\put(8199,-11446){\makebox(0,0)[lb]{\smash{{\SetFigFont{20}{24.0}{\rmdefault}{\mddefault}{\itdefault}{$x+a_2+a_4$}%
}}}}
%  METADATA <id>1607</id>
\put(15444,-11416){\makebox(0,0)[lb]{\smash{{\SetFigFont{20}{24.0}{\rmdefault}{\mddefault}{\itdefault}{$x+a_2+a_4$}%
}}}}
%  METADATA <id>1609</id>
\put(15160,-873){\rotatebox{300.0}{\makebox(0,0)[lb]{\smash{{\SetFigFont{20}{24.0}{\rmdefault}{\mddefault}{\itdefault}{$y+z+2a_1+2a_3$}%
}}}}}
%  METADATA <id>1610</id>
\put(22615,-873){\rotatebox{300.0}{\makebox(0,0)[lb]{\smash{{\SetFigFont{20}{24.0}{\rmdefault}{\mddefault}{\itdefault}{$y+z+2a_1+2a_3$}%
}}}}}
%  METADATA <id>1612</id>
\put(15226,-9089){\rotatebox{60.0}{\makebox(0,0)[lb]{\smash{{\SetFigFont{20}{24.0}{\rmdefault}{\mddefault}{\itdefault}{$y+z+2a_2+2a_4$}%
}}}}}
%  METADATA <id>1613</id>
\put(22651,-8969){\rotatebox{60.0}{\makebox(0,0)[lb]{\smash{{\SetFigFont{20}{24.0}{\rmdefault}{\mddefault}{\itdefault}{$y+z+2a_2+2a_4$}%
}}}}}
%  METADATA <id>1615</id>
\put(13899,-5071){\makebox(0,0)[lb]{\smash{{\SetFigFont{20}{24.0}{\rmdefault}{\mddefault}{\itdefault}{$a_4$}%
}}}}
\put(11830,-13068){\rotatebox{300.0}{\makebox(0,0)[lb]{\smash{{\SetFigFont{20}{24.0}{\rmdefault}{\mddefault}{\itdefault}{$y+z+2a_1+2a_3$}%
}}}}}
%  METADATA <id>1642</id>
\put(18916,-12961){\rotatebox{300.0}{\makebox(0,0)[lb]{\smash{{\SetFigFont{20}{24.0}{\rmdefault}{\mddefault}{\itdefault}{$y+z+2a_1+2a_3$}%
}}}}}
\end{picture}}
  \caption{Obtaining a recurrence for the number of tilings of $H^{(1)}$-type region.}\label{MidarrayKuo2}
\end{figure}

For the induction step, we assume $x,y,z$ are all positive and that (\ref{main1eq1}) holds for any $H^{(1)}$-type region with two holes whose sum of the $x$-, $y$- and $z$-parameters is strictly less than $x+y+z$. We need to verify (\ref{main1eq1})  for the region $H^{(1)}_{x,y,z}(a,b)$.

We apply Kuo condensation to the dual graph $G$ of the region $H^{(1)}_{x,y,z}(\textbf{a})$, for a general sequence of nonegative integers $\textbf{a}=(a_1,a_2,\dotsc,a_n)$ with the four vertices $u,v,w,s$ chosen as in Figure \ref{MidarrayKuo1}. In particular, each of the four vertices $u,v,w,s$ corresponds to the black unit triangle of the same label. The $u$- and $v$-unit triangles are the up-pointing and down-pointing unit triangles at the upper-right corner of the region, and the $w$- and $s$-unit triangles are the up-pointing and down-pointing unit triangles at the lower-right corner of the region.

Consider the region corresponding to the graph $G-\{u,v,w,s\}$. The removal of the black unit triangles yields several forced lozenges. By removing these forced lozenges, we get the region
$H^{(1)}_{x,y-1,z-1}(\textbf{a})$ (see the region restricted by the bold contour in Figure \ref{MidarrayKuo2}(a)) and obtain
\begin{equation}
\M(G-\{u,v,w,s\})=\M(H^{(1)}_{x,y-1,z-1}(\textbf{a})).
\end{equation}

Similarly, we have
\begin{equation}
\M(G-\{u,v\})=\M(H^{(1)}_{x,y-1,z}(\textbf{a})),
\end{equation}
\begin{equation}
\M(G-\{w,s\})=\M(H^{(1)}_{x,y,z-1}(\textbf{a})),
\end{equation}
\begin{equation}
\M(G-\{u,s\})=\M(H^{(1)}_{x+1,y-1,z-1}(\textbf{a})),
\end{equation}
and
\begin{equation}
\M(G-\{v,w\})=\M(H^{(1)}_{x-1,y,z}(\textbf{a}))
\end{equation}
(see Figures \ref{MidarrayKuo2}(b)--(e), respectively).
Plugging the above five identities into the equation (\ref{Kuoeq}) in Kuo's Theorem \ref{Kuothm}, we get the following recurrence:
\begin{align}\label{recur1}
\M(H^{(1)}_{x,y,z}(\textbf{a}))\M(H^{(1)}_{x,y-1,z-1}(\textbf{a}))=\M(H^{(1)}_{x,y-1,z}(\textbf{a}))\M(H^{(1)}_{x,y,z-1}(\textbf{a}))+\M(H^{(1)}_{x+1,y-1,z-1}(\textbf{a}))\M(H^{(1)}_{x-1,y,z}(\textbf{a})),
\end{align}
for any sequence of positive integers $\textbf{a}=(a_1,a_2,\dotsc,a_n)$. In particular, the number tilings of the region $H^{(1)}_{x,y,z}(a,b)$ satisfies this recurrence (by setting $n=2$, $a_1=a$, $a_2=b$).
To finish the proof of  (\ref{main1eq1}) we only need to verify that the expression on the right-hand side of  (\ref{main1eq1}) satisfies the same recurrence (\ref{recur1}). Indeed, denote by $\phi_{x,y,z}(a,b)$ this expression.
 We need to verify that
\begin{align}
\phi_{x,y,z}(a,b)\phi_{x,y-1,z-1}(a,b)=\phi_{x,y-1,z}(a,b)\phi_{x,y,z-1}(a,b)+\phi_{x+1,y-1,z-1}(a,b)\phi_{x-1,y,z}(a,b).
\end{align}
It is equivalent to show that
\begin{align}\label{recur1b}
\frac{\phi_{x,y-1,z}(a,b)}{\phi_{x,y-1,z-1}(a,b)}\frac{\phi_{x,y,z-1}(a,b)}{\phi_{x,y,z}(a,b)}+\frac{\phi_{x+1,y-1,z-1}(a,b)}{\phi_{x,y-1,z-1}(a,b)}\frac{\phi_{x-1,y,z}(a,b)}{\phi_{x,y,z}(a,b)}=1.
\end{align}

We have several claims that are direct consequence of Theorem \ref{Proctiling}, the definition of the function $\T$, and Lemma \ref{QAR}:

\begin{claim}\label{claimP}
\begin{align}
\frac{\Pn_{x,x,a}}{\Pn_{x-1,x-1,a}}=\frac{a+x}{x}\frac{(2a+2x-1)!}{(2a+x)!}\frac{x!}{(2x-1)!}.
\end{align}
\end{claim}

\begin{claim}\label{claimT}
\begin{align}
\frac{\T(x,n,m)}{\T(x-1,n,m)}=\frac{(x+n-m)_m}{(x-1)_m}.
\end{align}
\end{claim}

\begin{claim}\label{claimQ} For any sequence $\textbf{t}=(t_1,t_2,\dotsc,t_{2l})$
\begin{align}
\frac{\Q(t_1,\dots,t_{2l}+1)}{\Q(t_1,\dots,t_{2l})}=&\frac{(\s_{2l}(\textbf{t})+1)(2\s_{2l}(\textbf{t})+1)!}{(2\E(\textbf{t})+1)!}\notag\\
&\times\prod_{i=1}^{l}\frac{(\s_{2l}(\textbf{t})-\s_{2i-1}(\textbf{t}))!}{(\s_{2l}(\textbf{t})+\s_{2i-1}(\textbf{t})+1)!}\prod_{i=1}^{l-1}\frac{(\s_{2l}(\textbf{t})+\s_{2i}(\textbf{t})+1)!}{(\s_{2l}(\textbf{t})-\s_{2i}(\textbf{t}))!}.
\end{align}
\end{claim}

\begin{claim}\label{claimQ2}
For any sequence $\textbf{t}=(t_1,t_2,t_3,t_4)$
\begin{align}
&\frac{\Q(t_1,t_2,t_3,t_4)}{\Q(t_1,t_2,t_3-1,t_4)}=\frac{\s_{4}(\textbf{t})}{\s_{3}(\textbf{t})} (2\s_4(\textbf{t})-1)!(2\s_{3}(\textbf{t}))!\notag\\
&\times\frac{(\s_{4}(\textbf{t})-\s_{1}(\textbf{t})-1)!(\s_{3}(\textbf{t})-\s_{2}(\textbf{t})-1)!(\s_{4}(\textbf{t})+\s_{2}(\textbf{t}))!(\s_{3}(\textbf{t})+\s_{1}(\textbf{t}))!}{(\s_{4}(\textbf{t})-\s_{2}(\textbf{t})-1)!(\s_{3}(\textbf{t})-\s_{1}(\textbf{t})-1)!(\s_{4}(\textbf{t})+\s_{1}(\textbf{t}))!(\s_{3}(\textbf{t})+\s_{2}(\textbf{t}))!(s_4(\textbf{t})+s_3(\textbf{t})-1)!(s_4(\textbf{t})+s_3(\textbf{t}))!}.
\end{align}
\end{claim}

We now simplify the first fraction of the first term on the left-hand side of (\ref{recur1b}) as
\begin{align}
\frac{\phi(x,y-1,z)}{\phi(x,y-1,z-1)}&=\frac{\T(x+z+a+b+2,y+a-2,a)}{\T(x+z+a+b+1,y+a-2,a)}\frac{\T(z+1,y+b-2,b)}{\T(z,y+b-2,b)}\notag\\
&\times \frac{\T(z+a+b+1,y+a-2,a)}{\T(z+a+b+2,y+a-2,a)}\frac{\T(x+z,y+b-2,b)}{\T(x+z+1,y+b-2,b)}\notag\\
&\times \frac{\Q(a,b,x,y+z-1)}{\Q(a,b,x,y+z-2)} \frac{\Pn_{z+b,z+b,a}}{\Pn_{z+b-1,z+b-1,a}} \frac{\Pn_{y+z+b-2,y+z+b-2,a}}{\Pn_{y+z+b-1,y+z+b-1,a}}\\
&=\frac{(x+y+z+a+b)_a}{(x+z+a+b+1)_a}\frac{(y+z-1)_b}{(z)_b}\frac{(z+a+b+1)_a}{(y+z+a+b)_a}\frac{(x+z)_b}{(x+y+z-1)_b}\notag\\
&\times\frac{(x+y+z+a+b-1)(2x+2y+2z+2a+2b-3)!}{(2y+2z+2b-3)!}\notag\\
&\times\frac{(y+z-2)!(x+y+z+b-2)!(x+y+z+2a+2b-1)!}{(x+y+z-2)!(x+y+z+2a+b-1)!(2x+y+z+2a+2b-1)!}\notag\\
&\times \frac{a+b+z}{b+z}\frac{(2z+2a+2b-1)!}{(z+b+2a)!}\frac{(b+z)!}{(2z+2b-1)!}\notag\\
&\times \frac{y+z+b-1}{y+z+a+b-1}\frac{(y+z+2a+b-1)!}{(2y+2z+2a+2b-3)!}\frac{(2y+2z+2b-3)!}{(y+z+b-1)!}.
\end{align}
Working similarly for the second fraction, $\frac{\phi_{x,y,z-1}(a,b)}{\phi_{x,y,z}(a,b)}$, of the first term and multiplying the result by the above simplified form of the first fraction, we obtain
\begin{align}
&\frac{\phi_{x,y-1,z}(a,b)}{\phi_{x,y-1,z-1}(a,b)}\frac{\phi_{x,y,z-1}(a,b)}{\phi_{x,y,z}(a,b)}=\notag\\
&\frac{(x+y+z+a+b)}{(x+y+z+2a+b)}\frac{(y+z-1)}{(y+z+b-1)}\frac{(y+z+2a+b)}{(y+z+a+b)}\frac{(x+y+z+b-1)}{(x+y+z-1)}\notag\\
&\times\frac{(x+y+z+a+b-1)}{(x+y+z+a+b)}\frac{(2y+2z+2b-1)(2y+2z+2b-2)}{(2x+2y+2z+2a+2b-2)(2x+2y+2z+2a+2b-1)}\notag\\
&\times\frac{(x+y+z-1)(x+y+z+2a+b)(2x+y+z+2a+2b)}{(y+z-1)(x+y+z+b-1)(x+y+z+2a+2b)}\notag\\
&\times \frac{(y+z+b-1)(y+z+a+b)}{(y+z+b)(y+z+a+b-1)}\frac{(2y+2z+2a+2b-1)(2y+2z+2a+2b-2)(y+z+b)}{(y+z+2a+b)(2y+2z+2b-1)(2y+2z+2b-2)}\notag\\
&=\frac{(2x+y+z+2a+2b)(2y+2z+2a+2b-1)}{(2x+2y+2z+2a+2b-1)(x+y+z+2a+2b)}.\notag\\
\end{align}

 Next, we work on the second term on the left-hand side of (\ref{recur1b}). The first fraction can be written as
\begin{align}
&\frac{\phi_{x-1,y,z}(a,b)}{\phi_{x,y,z}(a,b)}=\frac{\T(x+b,y+a-1,a)}{\T(x+b+1,y+a-1,a)}\frac{\T(x+z+a+b+1,y+a-1,a)}{\T(x+z+a+b+2,y+a-1,a)}\notag\\
&\times \frac{\T(x+2a+b+2,y+b-1,b)}{\T(x+2a+b+1,y+b-1,b)}\frac{\T(x+z+1,y+b-1,b)}{\T(x+z,y+b-1,b)} \frac{\Q(a,b,x-1,y+z)}{\Q(a,b,x,y+z)}\\
&=\frac{(x+b)_a}{(x+y+b)_a}\frac{(x+z+a+b+1)_a}{(x+y+z+a+b+1)_a}\frac{(x+y+2a+b+1)_b}{(x+2a+b+1)_b}\frac{(x+y+z)_b}{(x+z)_b}\notag\\
&\times\frac{(x+a+b)}{(x+y+z+a+b)(2x+2y+2z+2a+2b-1)!(2x+2a+2b)!} \frac{(x+y+z-1)!(x+b-1)!}{(x+y+z+b-1)!(x-1)!}\notag\\
&\times\frac{(x+y+z+2a+b)!(x+2a+2b)!(2a+2b+2x+y+z-1)!(2a+2b+2x+y+z)!}{(x+y+z+2a+2b)!(x+2a+b)!}.
\end{align}
Doing similarly for the second fraction $\frac{\phi_{x-1,y,z}(a,b)}{\phi_{x,y,z}(a,b)}$ and multiplying by the above simplification of the first one, we get
\begin{align}
&\frac{\phi_{x+1,y-1,z-1}(a,b)}{\phi_{x,y-1,z-1}(a,b)}\frac{\phi_{x-1,y,z}(a,b)}{\phi_{x,y,z}(a,b)}=\notag\\
%&=\frac{(x+b)_a}{(x+b+1)_a}\frac{(x+y+z+a+b)_a}{(x+y+z+a+b+1)_a}\notag\\
%&\times \frac{(x+2a+b+2)_b}{(x+2a+b+1)_b}\frac{(x+y+z)_b}{(x+y+z-1)_b}\notag\\
%&\times\frac{(x+a+b)(x+y+z+a+b-1)}{(x+a+b+1)(x+y+z+a+b)}\frac{(2x+2y+2z+2a+2b-3)!(2x+2a+2b+2)!}{(2x+2y+2z+2a+2b-1)!(2x+2a+2b)!}\notag\\
%&\times\frac{(x+y+z-1)!(x+b-1)!(x+y+z+b-2)!(x)!}{(x+y+z-2)!(x+b)!(x+y+z+b-1)!(x-1)!}\frac{(x+y+z+2a+b)!(x+2a+2b)!(x+y+z+2a+2b-1)!(x+2a+b+1)!}{(x+y+z+2a+b-1)!(x+2a+2b+1)!(x+y+z+2a+2b)!(x+2a+b)!}\notag\\
&\frac{(x+b)}{(x+a+b)}\frac{(x+y+z+a+b)}{(x+y+z+2a+b)}\frac{(x+2a+2b+1)}{(x+2a+b+1)}\frac{(x+y+z+b-1)}{(x+y+z-1)}\notag\\
&\times\frac{(x+a+b)(x+y+z+a+b-1)}{(x+a+b+1)(x+y+z+a+b)}\frac{(2x+2a+2b+1)(2x+2a+2b+2)}{(2x+2y+2z+2a+2b-2)(2x+2y+2z+2a+2b-1)}\notag\\
&\times\frac{(x+y+z-1)(x)}{(x+b)(x+y+z+b-1)}\frac{(x+y+z+2a+b)(x+2a+b+1)}{(x+2a+2b+1)(x+y+z+2a+2b)}\notag\\
&=
\frac{x(2x+2a+2b+1)}{(2x+2y+2z+2a+2b-1)(x+y+z+2a+2b)}.
\end{align}
We have now
\begin{align}
\frac{\phi_{x,y-1,z}(a,b)}{\phi_{x,y-1,z-1}(a,b)}\frac{\phi_{x,y,z-1}(a,b)}{\phi_{x,y,z}(a,b)}&+\frac{\phi_{x+1,y-1,z-1}(a,b)}{\phi_{x,y-1,z-1}(a,b)}\frac{\phi_{x-1,y,z}(a,b)}{\phi_{x,y,z}(a,b)}\notag\\
&=\frac{(2x+y+z+2a+2b)(2y+2z+2a+2b-1)}{(2x+2y+2z+2a+2b-1)(x+y+z+2a+2b)}\notag\\
&+\frac{x(2x+2a+2b+1)}{(2x+2y+2z+2a+2b-1)(x+y+z+2a+2b)}=1.
\end{align}
 This finishes the proof of (\ref{main1eq1}).

\bigskip

We now prove (\ref{main1eq2}) by induction on $x+y+z$. The three base cases here are still the cases of $x=0$, $y=0$, and $z=0$.

Similar to the case when $k=1$ above, when $x=0$, we can partition our region into two $\mathcal{Q}$-type regions and several forced lozenges (shown in Figure \ref{Halfhex2basex}(b)). By Region-splitting Lemma \ref{RS}, we have
%\begin{equation}
%\M(H^{(1)}_{0,y,z}(a_1,\dots,a_{2k-1}))=\M(\mathcal{Q}(0,a_1,a_2,\dotsc,a_{2k-1}+y))\M(\mathcal{Q}(a_1,a_2,\dots,a_{2k-1},z))
%\end{equation}
%and
\begin{equation}
\M(H^{(1)}_{0,y,z}(a_1,\dots,a_{2k}))=\M(\mathcal{Q}(0,a_1,a_2,\dotsc,a_{2k},y))\M(\mathcal{Q}(a_1,a_2,\dots,a_{2k}+z)).
\end{equation}
Our tiling formula (\ref{main1eq2}) follows from (\ref{main1eq1}) and Lemma \ref{QAR}.

If $y=0$, we partition our regions into two $\mathcal{Q}$-type regions (and several forced lozenges) along the line containing the bases of the triangular holes (see Figure \ref{Halfhex2basey}(b)). By Region-splitting Lemma, we can also write the number of tilings of our region here as the product of these $\mathcal{Q}$-type regions as follows:
%\begin{equation}
%\M(H^{(1)}_{x,0,z}(a_1,\dots,a_{2k-1}))=\M(\mathcal{Q}(0,a_1,a_2,\dotsc,a_{2k-1}))\M(\mathcal{Q}(a_1,a_2,\dots,a_{2k-1}+x,z)),
%\end{equation}
\begin{equation}
\M(H^{(1)}_{x,0,z}(a_1,\dots,a_{2k}))=\M(\mathcal{Q}(0,a_1,a_2,\dotsc,a_{2k-1}))\M(\mathcal{Q}(a_1,a_2,\dots,a_{2k},x,z)).
\end{equation}
Then  (\ref{main1eq2}) follows again from (\ref{main1eq1}) and Lemma \ref{QAR}.

The last base case is the case when $z=0$. Similar to the cases of $x=0$ and $y=0$, we can break our region into two $\mathcal{Q}$-type regions as in Figure \ref{Halfhex2basez}(b). Region-splitting Lemma gives us the following formula:
%\begin{equation}
%\M(H^{(1)}_{x,y,0}(a_1,\dots,a_{2k-1}))=\M(\mathcal{Q}(0,a_1,a_2,\dotsc,a_{2k-1},x,y))\M(\mathcal{Q}(a_1,a_2,\dots,a_{2k-2})),
%\end{equation}
\begin{equation}
\M(H^{(1)}_{x,0,z}(a_1,\dots,a_{2k}))=\M(\mathcal{Q}(0,a_1,a_2,\dotsc,a_{2k-1},a_{2k}+x,y))\M(\mathcal{Q}(a_1,a_2,\dots,a_{2k})).
\end{equation}
Then (\ref{main1eq2}) was implied by (\ref{main1eq1}) and Lemma \ref{QAR}.

For the induction step, we assume that $x,y,z$ are all positive and that (\ref{main1eq2}) holes for any $H^{(1)}$-type region with at least $3$ holes whose sum of the $x$-, $y$- and $z$-parameters is strictly less than $x+y+z$. We need to show that (\ref{main1eq2}) holds for the general region $H^{(1)}_{x,y,z}(\textbf{a})$, where $\textbf{a}=(a_1,a_2,\dotsc,a_n)$ is a sequence of positive integers with $n\geq 3$.

As proved by using Kuo condensation above,  the tiling number of the region $H^{(1)}_{x,y,z}(\textbf{a})$ satisfies the recurrence  (\ref{recur1}). Therefore, our work is now verifying that the formula in (\ref{main1eq2}) satisfies the same recurrence.

Denote by $\Phi_{x,y,z}(\textbf{a})$ the formula on the right-hand side of (\ref{main1eq2}). We now need to verify that
\begin{align}\label{recur1b}
\frac{\Phi_{x,y-1,z}(\textbf{a})}{\Phi_{x,y-1,z-1}(\textbf{a})}\frac{\Phi_{x,y,z-1}(\textbf{a})}{\Phi_{x,y,z}(\textbf{a})}+\frac{\Phi_{x+1,y-1,z-1}(\textbf{a})}{\Phi_{x,y-1,z-1}(\textbf{a})}\frac{\Phi_{x-1,y,z}(\textbf{a})}{\phi_{x,y,z}(\textbf{a})}=1.
\end{align}
We write $\Phi_{x,y,z}(\textbf{a})=\phi_{x,y,z}(\Od(\textbf{a}),\E(\textbf{a})) \cdot f_{x,y,z}(\textbf{a})$, where $\phi_{x,y,z}(a,b)$ was defined as in the verification of  (\ref{main1eq1}). We need to show that:
\begin{align}\label{recur1c}
&\frac{f_{x,y-1,z}(\textbf{a})f_{x,y,z-1}(\textbf{a})}{f_{x,y-1,z-1}(\textbf{a})f_{x,y,z}(\textbf{a})}\frac{\phi_{x,y-1,z}(\Od(\textbf{a}),\E(\textbf{a}))}{\phi_{x,y-1,z-1}(\Od(\textbf{a}),\E(\textbf{a}))}\frac{\phi_{x,y,z-1}(\Od(\textbf{a}),\E(\textbf{a}))}{\phi_{x,y,z}(\Od(\textbf{a}),\E(\textbf{a}))}\notag\\
&+\frac{f_{x+1,y-1,z-1}(\textbf{a})f_{x-1,y,z}(\textbf{a})}{f_{x,y-1,z-1}(\textbf{a})f_{x,y,z}(\textbf{a})}\frac{\phi_{x+1,y-1,z-1}(\Od(\textbf{a}),\E(\textbf{a}))}{\phi_{x,y-1,z-1}(\Od(\textbf{a}),\E(\textbf{a}))}\frac{\phi_{x-1,y,z}(\Od(\textbf{a}),\E(\textbf{a}))}{\phi_{x,y,z}(\Od(\textbf{a}),\E(\textbf{a}))}=1.
\end{align}

 For given sequence $\textbf{a}=(a_1,a_2,\dotsc,a_n)$, the function $f_{x,y,z}(\textbf{a})$ only depends on the parameters $y$ and $z$ and the sums $x+y$ and $x+z$.  The values of these quadruples for the four occurrences of $f$ in the first fraction in the first terms are:
 
\begin{align}
&(y-1,z,x+y-1,x+z) \text{,  }\, \, \,  \, \,\,\,\,\,\,\,\,\,\,\,\, \,\,\,\,\,\,\,\,\,\,\,\,\,\,\, (y,z-1,z+y,x+z-1) \notag\\
&(y-1,z-1,x+y-1,x+z-1) \text{,  } \,\,\,\,\,\,\,\,\,\, (y,z,x+y,x+z),
\end{align}

where the two quadruples on the top row correspond to the two $f$-functions on the numerator, and the two quadruples on the bottom row correspond to the two $f$-functions on the denominator of our fraction. In particular, the values these parameters in each row form the same set. This means that the product of two $f$-functions on numerator is equal to the product of the ones on the denominator. It means that 
\begin{equation}
\frac{f_{x,y-1,z}(\textbf{a})f_{x,y,z-1}(\textbf{a})}{f_{x,y-1,z-1}(\textbf{a})f_{x,y,z}(\textbf{a})}=1.
\end{equation}
Similarly, we have 
\begin{equation}
\frac{f_{x+1,y-1,z-1}(\textbf{a})f_{x-1,y,z}(\textbf{a})}{f_{x,y-1,z-1}(\textbf{a})f_{x,y,z}(\textbf{a})}=1,
\end{equation}
and (\ref{recur1c}) becomes 
\begin{align}\label{recur1d}
&\frac{\phi_{x,y-1,z}(\Od(\textbf{a}),\E(\textbf{a}))}{\phi_{x,y-1,z-1}(\Od(\textbf{a}),\E(\textbf{a}))}\frac{\phi_{x,y,z-1}(\Od(\textbf{a}),\E(\textbf{a}))}{\phi_{x,y,z}(\Od(\textbf{a}),\E(\textbf{a}))}+\frac{\phi_{x+1,y-1,z-1}(\Od(\textbf{a}),\E(\textbf{a}))}{\phi_{x,y-1,z-1}(\Od(\textbf{a}),\E(\textbf{a}))}\frac{\phi_{x-1,y,z}(\Od(\textbf{a}),\E(\textbf{a}))}{\phi_{x,y,z}(\Od(\textbf{a}),\E(\textbf{a}))}=1.
\end{align}
However, this is exactly a consequence of (\ref{main1eq1}), then (\ref{main1eq2}) follows.

Finally, as mentioned before, any $H^{(1)}$-type region with an odd number of holes can be viewed as a special one with an even number of holes, when the right most hole has size $0$. This means (\ref{main1eqx}) follows.
\end{proof}

The proofs of Theorems \ref{main2}--\ref{main4}  are essentially the same as that of Theorem \ref{main1}, and will be omitted.

\bigskip

We devote the last part of this section to the proof of Theorem \ref{main5}.

\begin{proof}[Proof of Theorem \ref{main5}]
This proof follows the lines in the proof of Theorem \ref{main1}.

We first prove (\ref{main5eq1}) by induction on $x+y+z$. The base cases are still: $x=0$, $y=0$, and $z=0$.

Similar to the case of $H^{(1)}$-type regions, if $x=0$, we can split our regions into two disjoint subregions (and several forced lozenges as in Figure \ref{Halfhex2basex2}). By Region-splitting Lemms \ref{RS}, we get
%\begin{equation}
%\M(H^{(5)}_{0,y,z}(a_1,\dots,a_{2k-1}))=2^{a_1}\M(\mathcal{K}'(0,a_1+1,a_2,\dotsc,a_{2k-1}+y))\M(\mathcal{Q}(a_1,a_2,\dots,a_{2k-1},z)).
%\end{equation}
%Similarly, we get
\begin{equation}
\M(H^{(5 )}_{0,y,z}(a_1,\dots,a_{2k}))=2^{a_1}\M(\mathcal{K}'(0,a_1+1,a_2,\dotsc,a_{2k},y))\M(\mathcal{Q}(a_1,a_2,\dots,a_{2k}+z)).
\end{equation}
Then  (\ref{main5eq1}) follows from the above identity (when $k=1$) and Lemma \ref{QAR}.

If $y=0$, Figure \ref{Halfhex2basey2} and Region-splitting Lemma \ref{RS} tell us
%\begin{equation}
%\M(H^{(5)}_{x,0,z}(a_1,\dots,a_{2k-1}))=2^{a_1}\M(\mathcal{K}'(0,a_1+1,a_2,\dotsc,a_{2k-1}))\M(\mathcal{Q}(a_1,a_2,\dots,a_{2k-1}+x,z)).
%\end{equation}
%Similarly, we get
\begin{equation}
\M(H^{(5 )}_{x,0,z}(a_1,\dots,a_{2k}))=2^{a_1}\M(\mathcal{K}'(0,a_1+1,a_2,\dotsc,a_{2k-1}))\M(\mathcal{Q}(a_1,a_2,\dots,a_{2k},x,z)).
\end{equation}
Finally, when $z=0$,  we have as in Figure \ref{Halfhex2basez2} 
%\begin{equation}
%\M(H^{(5)}_{x,y,0}(a_1,\dots,a_{2k-1}))=2^{a_1}\M(\mathcal{K}'(0,a_1+1,a_2,\dotsc,a_{2k-1},x,y))\M(\mathcal{Q}(a_1,a_2,\dots,a_{2k-2})).
%\end{equation}
%Similarly, we get
\begin{equation}
\M(H^{(5 )}_{x,y,0}(a_1,\dots,a_{2k}))=2^{a_1}\M(\mathcal{K}'(0,a_1+1,a_2,\dotsc,a_{2k-1},a_{2k}+x,y))\M(\mathcal{Q}(a_1,a_2,\dots,a_{2k})).
\end{equation}
Then (\ref{main5eq1}) follows again from  Lemma \ref{QAR}.

\begin{figure}
  \centering
  \setlength{\unitlength}{3947sp}%
\begingroup\makeatletter\ifx\SetFigFont\undefined%
\gdef\SetFigFont#1#2#3#4#5{%
  \reset@font\fontsize{#1}{#2pt}%
  \fontfamily{#3}\fontseries{#4}\fontshape{#5}%
  \selectfont}%
\fi\endgroup%
\resizebox{10cm}{!}{
 \begin{picture}(0,0)%
\includegraphics{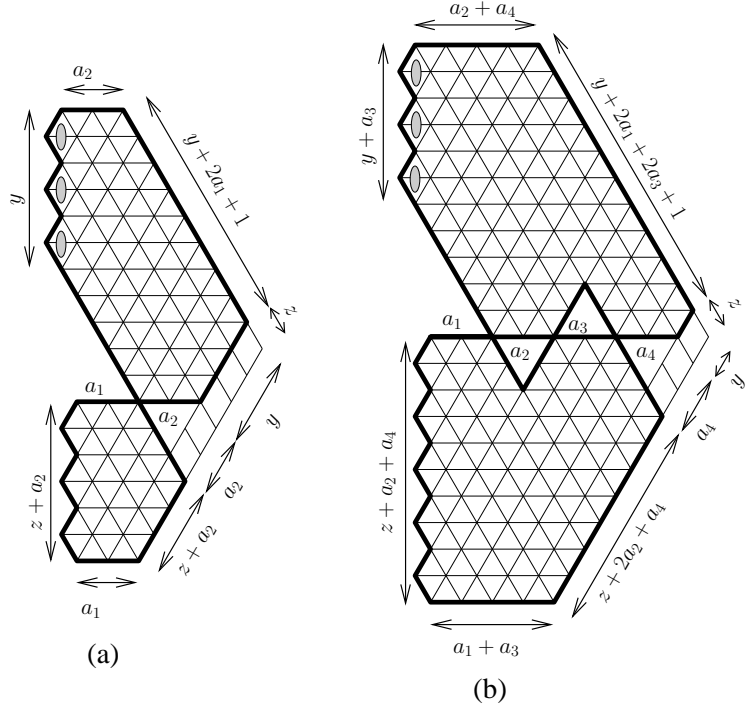}%
\end{picture}%
%
%  Created by WinFIG version 6.2 
%  METADATA <version>1.0</version> 
%

\begin{picture}(9532,9124)(1254,-8743)
%  METADATA <id>310</id> 
\put(4156,-6106){\rotatebox{60.0}{\makebox(0,0)[lb]{\smash{{\SetFigFont{17}{20.4}{\rmdefault}{\mddefault}{\itdefault}{$a_2$}%
}}}}}
%  METADATA <id>303</id> 
\put(10421,-3568){\rotatebox{300.0}{\makebox(0,0)[lb]{\smash{{\SetFigFont{17}{20.4}{\rmdefault}{\mddefault}{\itdefault}{$z$}%
}}}}}
%  METADATA <id>304</id> 
\put(8732,-778){\rotatebox{300.0}{\makebox(0,0)[lb]{\smash{{\SetFigFont{17}{20.4}{\rmdefault}{\mddefault}{\itdefault}{$y+2a_1+2a_3+1$}%
}}}}}
%  METADATA <id>305</id> 
\put(6191,-6538){\rotatebox{90.0}{\makebox(0,0)[lb]{\smash{{\SetFigFont{17}{20.4}{\rmdefault}{\mddefault}{\itdefault}{$z+a_2+a_4$}%
}}}}}
%  METADATA <id>306</id> 
\put(5936,-1843){\rotatebox{90.0}{\makebox(0,0)[lb]{\smash{{\SetFigFont{17}{20.4}{\rmdefault}{\mddefault}{\itdefault}{$y+a_3$}%
}}}}}
%  METADATA <id>264</id> 
\put(2311,-4726){\makebox(0,0)[lb]{\smash{{\SetFigFont{17}{20.4}{\rmdefault}{\mddefault}{\itdefault}{$a_1$}%
}}}}
%  METADATA <id>267</id> 
\put(6821,-3883){\makebox(0,0)[lb]{\smash{{\SetFigFont{17}{20.4}{\rmdefault}{\mddefault}{\itdefault}{$a_1$}%
}}}}
%  METADATA <id>268</id> 
\put(3241,-5116){\makebox(0,0)[lb]{\smash{{\SetFigFont{17}{20.4}{\rmdefault}{\mddefault}{\itdefault}{$a_2$}%
}}}}
%  METADATA <id>271</id> 
\put(7676,-4288){\makebox(0,0)[lb]{\smash{{\SetFigFont{17}{20.4}{\rmdefault}{\mddefault}{\itdefault}{$a_2$}%
}}}}
%  METADATA <id>274</id> 
\put(8411,-3943){\makebox(0,0)[lb]{\smash{{\SetFigFont{17}{20.4}{\rmdefault}{\mddefault}{\itdefault}{$a_3$}%
}}}}
%  METADATA <id>275</id> 
\put(9266,-4288){\makebox(0,0)[lb]{\smash{{\SetFigFont{17}{20.4}{\rmdefault}{\mddefault}{\itdefault}{$a_4$}%
}}}}
%  METADATA <id>277</id> 
\put(2161,-706){\makebox(0,0)[lb]{\smash{{\SetFigFont{17}{20.4}{\rmdefault}{\mddefault}{\itdefault}{$a_2$}%
}}}}
%  METADATA <id>280</id> 
\put(2266,-7561){\makebox(0,0)[lb]{\smash{{\SetFigFont{17}{20.4}{\rmdefault}{\mddefault}{\itdefault}{$a_1$}%
}}}}
%  METADATA <id>282</id> 
\put(1771,-6496){\rotatebox{90.0}{\makebox(0,0)[lb]{\smash{{\SetFigFont{17}{20.4}{\rmdefault}{\mddefault}{\itdefault}{$z+a_2$}%
}}}}}
%  METADATA <id>283</id> 
\put(3586,-7081){\rotatebox{60.0}{\makebox(0,0)[lb]{\smash{{\SetFigFont{17}{20.4}{\rmdefault}{\mddefault}{\itdefault}{$z+a_2$}%
}}}}}
%  METADATA <id>285</id> 
\put(4696,-5191){\rotatebox{60.0}{\makebox(0,0)[lb]{\smash{{\SetFigFont{17}{20.4}{\rmdefault}{\mddefault}{\itdefault}{$y$}%
}}}}}
%  METADATA <id>286</id> 
\put(1501,-2356){\rotatebox{90.0}{\makebox(0,0)[lb]{\smash{{\SetFigFont{17}{20.4}{\rmdefault}{\mddefault}{\itdefault}{$y$}%
}}}}}
%  METADATA <id>287</id> 
\put(4816,-3661){\rotatebox{300.0}{\makebox(0,0)[lb]{\smash{{\SetFigFont{17}{20.4}{\rmdefault}{\mddefault}{\itdefault}{$z$}%
}}}}}
%  METADATA <id>288</id> 
\put(3586,-1486){\rotatebox{300.0}{\makebox(0,0)[lb]{\smash{{\SetFigFont{17}{20.4}{\rmdefault}{\mddefault}{\itdefault}{$y+2a_1+1$}%
}}}}}
%  METADATA <id>298</id> 
\put(6911, 77){\makebox(0,0)[lb]{\smash{{\SetFigFont{17}{20.4}{\rmdefault}{\mddefault}{\itdefault}{$a_2+a_4$}%
}}}}
%  METADATA <id>299</id> 
\put(6956,-8053){\makebox(0,0)[lb]{\smash{{\SetFigFont{17}{20.4}{\rmdefault}{\mddefault}{\itdefault}{$a_1+a_3$}%
}}}}
%  METADATA <id>300</id> 
\put(8891,-7423){\rotatebox{60.0}{\makebox(0,0)[lb]{\smash{{\SetFigFont{17}{20.4}{\rmdefault}{\mddefault}{\itdefault}{$z+2a_2+a_4$}%
}}}}}
%  METADATA <id>301</id> 
\put(10121,-5368){\rotatebox{60.0}{\makebox(0,0)[lb]{\smash{{\SetFigFont{17}{20.4}{\rmdefault}{\mddefault}{\itdefault}{$a_4$}%
}}}}}
%  METADATA <id>302</id> 
\put(10556,-4678){\rotatebox{60.0}{\makebox(0,0)[lb]{\smash{{\SetFigFont{17}{20.4}{\rmdefault}{\mddefault}{\itdefault}{$y$}%
}}}}}
\end{picture}}
  \caption{Spitting up a $H^{(5)}$-type region in the case of $x=0$.}\label{Halfhex2basex2}
\end{figure}

\begin{figure}
  \centering
  \setlength{\unitlength}{3947sp}%
\begingroup\makeatletter\ifx\SetFigFont\undefined%
\gdef\SetFigFont#1#2#3#4#5{%
  \reset@font\fontsize{#1}{#2pt}%
  \fontfamily{#3}\fontseries{#4}\fontshape{#5}%
  \selectfont}%
\fi\endgroup%
\resizebox{10cm}{!}{
\begin{picture}(0,0)%
\includegraphics{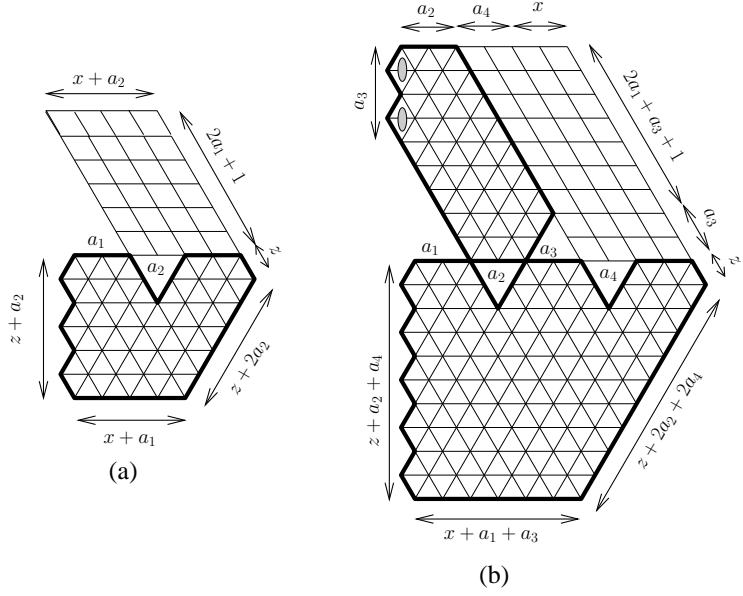}%
\end{picture}%
%
%  Created by WinFIG version 6.2 
%  METADATA <version>1.0</version> 
%

\begin{picture}(10633,8524)(848,-8043)
%  METADATA <id>315</id> 
\put(4649,-3183){\rotatebox{300.0}{\makebox(0,0)[lb]{\smash{{\SetFigFont{17}{20.4}{\rmdefault}{\mddefault}{\itdefault}{$z$}%
}}}}}
%  METADATA <id>298</id> 
\put(7111,-7308){\makebox(0,0)[lb]{\smash{{\SetFigFont{17}{20.4}{\rmdefault}{\mddefault}{\itdefault}{$x+a_1+a_3$}%
}}}}
%  METADATA <id>300</id> 
\put(9931,-6333){\rotatebox{60.0}{\makebox(0,0)[lb]{\smash{{\SetFigFont{17}{20.4}{\rmdefault}{\mddefault}{\itdefault}{$z+2a_2+2a_4$}%
}}}}}
%  METADATA <id>304</id> 
\put(3763,-1544){\rotatebox{300.0}{\makebox(0,0)[lb]{\smash{{\SetFigFont{17}{20.4}{\rmdefault}{\mddefault}{\itdefault}{$2a_1+1$}%
}}}}}
%  METADATA <id>313</id> 
\put(4209,-5276){\rotatebox{60.0}{\makebox(0,0)[lb]{\smash{{\SetFigFont{17}{20.4}{\rmdefault}{\mddefault}{\itdefault}{$z+2a_2$}%
}}}}}
%  METADATA <id>259</id> 
\put(2098,-3134){\makebox(0,0)[lb]{\smash{{\SetFigFont{17}{20.4}{\rmdefault}{\mddefault}{\itdefault}{$a_1$}%
}}}}
%  METADATA <id>262</id> 
\put(6796,-3228){\makebox(0,0)[lb]{\smash{{\SetFigFont{17}{20.4}{\rmdefault}{\mddefault}{\itdefault}{$a_1$}%
}}}}
%  METADATA <id>272</id> 
\put(2938,-3524){\makebox(0,0)[lb]{\smash{{\SetFigFont{17}{20.4}{\rmdefault}{\mddefault}{\itdefault}{$a_2$}%
}}}}
%  METADATA <id>275</id> 
\put(7726,-3618){\makebox(0,0)[lb]{\smash{{\SetFigFont{17}{20.4}{\rmdefault}{\mddefault}{\itdefault}{$a_2$}%
}}}}
%  METADATA <id>279</id> 
\put(8461,-3258){\makebox(0,0)[lb]{\smash{{\SetFigFont{17}{20.4}{\rmdefault}{\mddefault}{\itdefault}{$a_3$}%
}}}}
%  METADATA <id>280</id> 
\put(9301,-3603){\makebox(0,0)[lb]{\smash{{\SetFigFont{17}{20.4}{\rmdefault}{\mddefault}{\itdefault}{$a_4$}%
}}}}
%  METADATA <id>281</id> 
\put(1888,-854){\makebox(0,0)[lb]{\smash{{\SetFigFont{17}{20.4}{\rmdefault}{\mddefault}{\itdefault}{$x+a_2$}%
}}}}
%  METADATA <id>284</id> 
\put(6736,147){\makebox(0,0)[lb]{\smash{{\SetFigFont{17}{20.4}{\rmdefault}{\mddefault}{\itdefault}{$a_2$}%
}}}}
%  METADATA <id>285</id> 
\put(7516,147){\makebox(0,0)[lb]{\smash{{\SetFigFont{17}{20.4}{\rmdefault}{\mddefault}{\itdefault}{$a_4$}%
}}}}
%  METADATA <id>286</id> 
\put(8356,177){\makebox(0,0)[lb]{\smash{{\SetFigFont{17}{20.4}{\rmdefault}{\mddefault}{\itdefault}{$x$}%
}}}}
%  METADATA <id>287</id> 
\put(5986,-1218){\rotatebox{90.0}{\makebox(0,0)[lb]{\smash{{\SetFigFont{17}{20.4}{\rmdefault}{\mddefault}{\itdefault}{$a_3$}%
}}}}}
%  METADATA <id>290</id> 
\put(9631,-723){\rotatebox{300.0}{\makebox(0,0)[lb]{\smash{{\SetFigFont{17}{20.4}{\rmdefault}{\mddefault}{\itdefault}{$2a_1+a_3+1$}%
}}}}}
%  METADATA <id>291</id> 
\put(10756,-2658){\rotatebox{300.0}{\makebox(0,0)[lb]{\smash{{\SetFigFont{17}{20.4}{\rmdefault}{\mddefault}{\itdefault}{$a_3$}%
}}}}}
%  METADATA <id>292</id> 
\put(11116,-3318){\rotatebox{300.0}{\makebox(0,0)[lb]{\smash{{\SetFigFont{17}{20.4}{\rmdefault}{\mddefault}{\itdefault}{$z$}%
}}}}}
%  METADATA <id>293</id> 
\put(6196,-6003){\rotatebox{90.0}{\makebox(0,0)[lb]{\smash{{\SetFigFont{17}{20.4}{\rmdefault}{\mddefault}{\itdefault}{$z+a_2+a_4$}%
}}}}}
%  METADATA <id>295</id> 
\put(1153,-4559){\rotatebox{90.0}{\makebox(0,0)[lb]{\smash{{\SetFigFont{17}{20.4}{\rmdefault}{\mddefault}{\itdefault}{$z+a_2$}%
}}}}}
%  METADATA <id>296</id> 
\put(2323,-5939){\makebox(0,0)[lb]{\smash{{\SetFigFont{17}{20.4}{\rmdefault}{\mddefault}{\itdefault}{$x+a_1$}%
}}}}
\end{picture}}
  \caption{The base case $y=0$ for the $H^{(5)}$-type regions.}\label{Halfhex2basey2}
\end{figure}

\begin{figure}
  \centering
  \setlength{\unitlength}{3947sp}%
\begingroup\makeatletter\ifx\SetFigFont\undefined%
\gdef\SetFigFont#1#2#3#4#5{%
  \reset@font\fontsize{#1}{#2pt}%
  \fontfamily{#3}\fontseries{#4}\fontshape{#5}%
  \selectfont}%
\fi\endgroup%
\resizebox{10cm}{!}{
\begin{picture}(0,0)%
\includegraphics{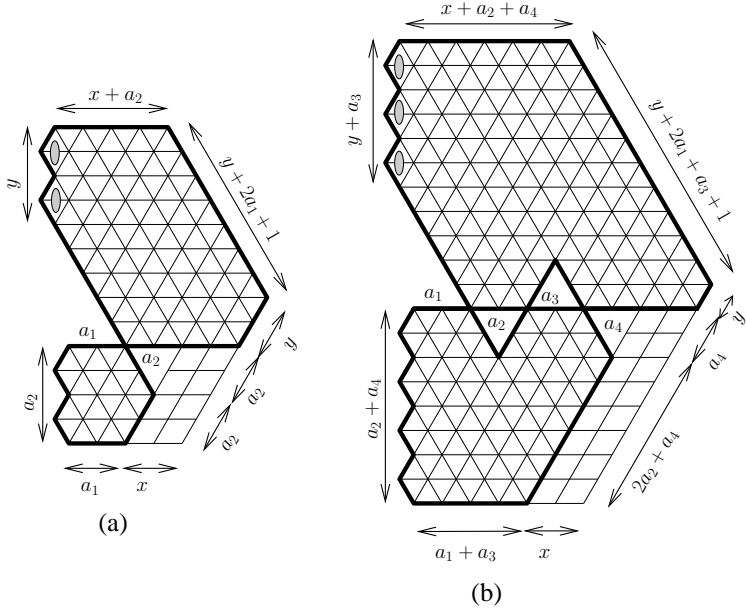}%
\end{picture}%
%
%  Created by WinFIG version 6.2 
%  METADATA <version>1.0</version> 
%

\begin{picture}(10399,8559)(1146,-8139)
%  METADATA <id>301</id> 
\put(2278,-6501){\makebox(0,0)[lb]{\smash{{\SetFigFont{17}{20.4}{\rmdefault}{\mddefault}{\itdefault}{$a_1$}%
}}}}
%  METADATA <id>305</id> 
\put(1633,-5391){\rotatebox{90.0}{\makebox(0,0)[lb]{\smash{{\SetFigFont{17}{20.4}{\rmdefault}{\mddefault}{\itdefault}{$a_2$}%
}}}}}
%  METADATA <id>306</id> 
\put(4303,-5961){\rotatebox{60.0}{\makebox(0,0)[lb]{\smash{{\SetFigFont{17}{20.4}{\rmdefault}{\mddefault}{\itdefault}{$a_2$}%
}}}}}
%  METADATA <id>307</id> 
\put(4663,-5316){\rotatebox{60.0}{\makebox(0,0)[lb]{\smash{{\SetFigFont{17}{20.4}{\rmdefault}{\mddefault}{\itdefault}{$a_2$}%
}}}}}
%  METADATA <id>261</id> 
\put(2248,-4341){\makebox(0,0)[lb]{\smash{{\SetFigFont{17}{20.4}{\rmdefault}{\mddefault}{\itdefault}{$a_1$}%
}}}}
%  METADATA <id>264</id> 
\put(7003,-3814){\makebox(0,0)[lb]{\smash{{\SetFigFont{17}{20.4}{\rmdefault}{\mddefault}{\itdefault}{$a_1$}%
}}}}
%  METADATA <id>265</id> 
\put(3118,-4716){\makebox(0,0)[lb]{\smash{{\SetFigFont{17}{20.4}{\rmdefault}{\mddefault}{\itdefault}{$a_2$}%
}}}}
%  METADATA <id>268</id> 
\put(7873,-4219){\makebox(0,0)[lb]{\smash{{\SetFigFont{17}{20.4}{\rmdefault}{\mddefault}{\itdefault}{$a_2$}%
}}}}
%  METADATA <id>271</id> 
\put(8608,-3859){\makebox(0,0)[lb]{\smash{{\SetFigFont{17}{20.4}{\rmdefault}{\mddefault}{\itdefault}{$a_3$}%
}}}}
%  METADATA <id>272</id> 
\put(9478,-4249){\makebox(0,0)[lb]{\smash{{\SetFigFont{17}{20.4}{\rmdefault}{\mddefault}{\itdefault}{$a_4$}%
}}}}
%  METADATA <id>273</id> 
\put(1393,-2301){\rotatebox{90.0}{\makebox(0,0)[lb]{\smash{{\SetFigFont{17}{20.4}{\rmdefault}{\mddefault}{\itdefault}{$y$}%
}}}}}
%  METADATA <id>276</id> 
\put(6088,-1714){\rotatebox{90.0}{\makebox(0,0)[lb]{\smash{{\SetFigFont{17}{20.4}{\rmdefault}{\mddefault}{\itdefault}{$y+a_3$}%
}}}}}
%  METADATA <id>277</id> 
\put(2383,-1071){\makebox(0,0)[lb]{\smash{{\SetFigFont{17}{20.4}{\rmdefault}{\mddefault}{\itdefault}{$x+a_2$}%
}}}}
%  METADATA <id>279</id> 
\put(7228,116){\makebox(0,0)[lb]{\smash{{\SetFigFont{17}{20.4}{\rmdefault}{\mddefault}{\itdefault}{$x+a_2+a_4$}%
}}}}
%  METADATA <id>280</id> 
\put(7153,-7419){\makebox(0,0)[lb]{\smash{{\SetFigFont{17}{20.4}{\rmdefault}{\mddefault}{\itdefault}{$a_1+a_3$}%
}}}}
%  METADATA <id>281</id> 
\put(8578,-7419){\makebox(0,0)[lb]{\smash{{\SetFigFont{17}{20.4}{\rmdefault}{\mddefault}{\itdefault}{$x$}%
}}}}
%  METADATA <id>292</id> 
\put(5170,-4546){\rotatebox{60.0}{\makebox(0,0)[lb]{\smash{{\SetFigFont{17}{20.4}{\rmdefault}{\mddefault}{\itdefault}{$y$}%
}}}}}
%  METADATA <id>293</id> 
\put(11338,-4179){\rotatebox{60.0}{\makebox(0,0)[lb]{\smash{{\SetFigFont{17}{20.4}{\rmdefault}{\mddefault}{\itdefault}{$y$}%
}}}}}
%  METADATA <id>294</id> 
\put(6358,-5779){\rotatebox{90.0}{\makebox(0,0)[lb]{\smash{{\SetFigFont{17}{20.4}{\rmdefault}{\mddefault}{\itdefault}{$a_2+a_4$}%
}}}}}
%  METADATA <id>295</id> 
\put(10003,-6534){\rotatebox{60.0}{\makebox(0,0)[lb]{\smash{{\SetFigFont{17}{20.4}{\rmdefault}{\mddefault}{\itdefault}{$2a_2+a_4$}%
}}}}}
%  METADATA <id>296</id> 
\put(11008,-4824){\rotatebox{60.0}{\makebox(0,0)[lb]{\smash{{\SetFigFont{17}{20.4}{\rmdefault}{\mddefault}{\itdefault}{$a_4$}%
}}}}}
%  METADATA <id>297</id> 
\put(10123,-1114){\rotatebox{300.0}{\makebox(0,0)[lb]{\smash{{\SetFigFont{17}{20.4}{\rmdefault}{\mddefault}{\itdefault}{$y+2a_1+a_3+1$}%
}}}}}
%  METADATA <id>299</id> 
\put(4213,-1911){\rotatebox{300.0}{\makebox(0,0)[lb]{\smash{{\SetFigFont{17}{20.4}{\rmdefault}{\mddefault}{\itdefault}{$y+2a_1+1$}%
}}}}}
%  METADATA <id>300</id> 
\put(3043,-6501){\makebox(0,0)[lb]{\smash{{\SetFigFont{17}{20.4}{\rmdefault}{\mddefault}{\itdefault}{$x$}%
}}}}
\end{picture}}
  \caption{Applying the Region-splitting Lemma \ref{RS} to a $H^{(5)}$-type region in the case of $z=0$.}\label{Halfhex2basez2}
\end{figure}

For the induction step, we assume that $x,y,z\geq 1$ and that (\ref{main5eq1}) holds for any $H^{(5)}$-type region with two holes whose sum of the $x$-, $y$-, and $z$-parameters is strictly less than $x+y+z$. We need to show   (\ref{main5eq1})  for any region $H^{(5)}_{x,y,z}(a,b)$.

We apply Kuo condensation to the dual graph $G$ of the region $H^{(5)}_{x,y,z}(\textbf{a})$ for a general sequence of positive integers $\textbf{a}=(a_1,a_2,\dotsc,a_n)$, based on Figure \ref{Kuomidhole3}. The figure tells us that the product of the tiling numbers of the two region in the top row equals the product of the tiling numbers of the two regions in the middle row plus the product of the tiling numbers of the two regions in the bottom row. In particular, we have
\begin{align}\label{recur2}
\M(H^{(5)}_{x,y,z}(\textbf{a}))\M(H^{(5)}_{x,y-1,z-1}(\textbf{a}))=\M(H^{(5)}_{x,y-1,z}(\textbf{a}))\M(H^{(5)}_{x,y,z-1}(\textbf{a}))+\M(H^{(5)}_{x+1,y-1,z-1}(\textbf{a}))\M(H^{(5)}_{x-1,y,z}(\textbf{a})),
\end{align}
for any sequence $\textbf{a}=(a_1,a_2,\dotsc,a_k)$.
To verify    (\ref{main5eq1}) we only need to show that the expression on the right-hand side of its satisfies the same recurrence.

We denote by $\psi_{x,y,z}(a,b)$ the product on the right-hand side of  (\ref{main5eq1}), we will show that
\begin{align}\label{recur2b}
\psi_{x,y,z}(a,b)\psi_{x,y-1,z-1}(a,b)=\psi_{x,y-1,z}(a,b)\psi_{x,y,z-1}(a,b)+\psi_{x+1,y-1,z-1}(a,b)\psi_{x-1,y,z}(a,b),
\end{align}
or
\begin{align}\label{recur2c}
\frac{\psi_{x,y-1,z}(a,b)}{\psi_{x,y-1,z-1}(a,b)}\frac{\psi_{x,y,z-1}(a,b)}{\psi_{x,y,z}(a,b)}+\frac{\psi_{x+1,y-1,z-1}(a,b)}{\psi_{x,y-1,z-1}(a,b)}\frac{\psi_{x-1,y,z}(a,b)}{\psi_{x,y,z}(a,b)}=1.
\end{align}
We have several additional claims (that follow directly from the definition of the function $\T$ and $\V$) as follows:
\begin{claim}\label{claimV}
\begin{equation}\frac{\V(x,n,m)}{\V(x,n-1,m))}=[x+2n-2m]_{m}.\end{equation}
\end{claim}
\begin{claim}\label{claimV2}
\begin{equation}\frac{\V(x,n,m)}{\V(x-2,n,m))}=\frac{[x+2n-2m]_{m}}{[x-2]_m}.\end{equation}
\end{claim}
\begin{claim}\label{claimT2}
\begin{equation}\frac{\T(x,n,m)}{\T(x,n-1,m))}=(x+n-m)_{m}.\end{equation}
\end{claim}
We now have the first fraction in the first term on the left-hand side of (\ref{recur2c}) simplified as
\begin{align}
\frac{\psi_{x,y-1,z}(a,b)}{\psi_{x,y-1,z-1}(a,b)}&=\frac{\Pn_{z+b,z+b,a}}{\Pn_{z+b-1,z+b-1,a}}\frac{\Pn_{y+z+b-2,y+z+b-2,a}}{\Pn_{y+z+b-1,y+z+b-1,a}}\notag\\
&\times \frac{\V(2a+2b+3,y+z-2,y-1)}{\V(2a+2b+3,y+z-3,y-1)}\frac{\V(2x+2a+2b+3,y+z-3,y-1)}{\V(2x+2a+2b+3,y+z-2,y-1)}\notag\\
&\times \frac{\T(x+b+1,y+z+2a-1,y-1)}{\T(x+b+1,y+z+2a-2,y-1)}\frac{\T(b+1,y+z+2a-2,y-1)}{\T(b+1,y+z+2a-1,y-1)}\notag\\
&\times \frac{\T(z+1,y+b-2,b)}{\T(z,y+b-2,b)} \frac{\T(x+z,y+b-2,b)}{\T(x+z+1,y+b-2,b)} \frac{Q(a,b,x,y+z-1)}{Q(a,b,x,y+z-2)}\\
&=\frac{\Pn_{z+b,z+b,a}}{\Pn_{z+b-1,z+b-1,a}}\notag\\
&\times \frac{y+z+b-1}{y+z+a+b-1}\frac{(y+z+2a+b-1)!}{(2y+2z+2a+2b-3)!}\frac{(2y+2z+2b-3)!}{(y+z+b-1)!}\notag\\
&\times \frac{[2z+2a+2b+1]_{y-1}}{[2x+2z+2a+2b+1]_{y-1}} \frac{(x+z+2a+b+1)_{y-1}}{(z+2a+b+1)_{y-1}}\frac{(y+z-1)_b}{(z)_b}\frac{(x+z)_b}{(x+y+z-1)_b}\notag\\
&\times\frac{(x+y+z+a+b-1)(2x+2y+2z+2a+2b-3)!}{(2y+2z+2b-3)!}\notag\\
&\times\frac{(y+z-2)!(x+y+z+b-2)!(x+y+z+2a+2b-1)!}{(x+y+z-2)!(x+y+z+2a+b-1)!(2x+y+z+2a+2b-1)!}\\
\end{align}
Working the same for the fraction $\frac{\psi_{x,y,z-1}(a,b)}{\psi_{x,y,z}(a,b)}$ and multiplying up the two simplifications, we get
\begin{align}
&\frac{\psi_{x,y-1,z}(a,b)}{\psi_{x,y-1,z-1}(a,b)}\frac{\psi_{x,y,z-1}(a,b)}{\psi_{x,y,z}(a,b)}=\notag\\
%&=\frac{(y+z+b-1)(y+z+a+b)}{(y+z+a+b-1)(y+z+b)}\frac{(2y+2z+2a+2b-1)!}{(2y+2z+2a+2b-3)!}\notag\\
%&\frac{(y+z+2a+b-1)!}{(y+z+2a+b)!}\frac{(y+z+b)!}{(y+z+b-1)!}\frac{(2y+2z+2b-3)!}{(2y+2z+2b-1)!}\notag\\
%&\times \frac{[2z+2a+2b+1]_{y-1}}{[2z+2a+2b+1]_{y}}\frac{[2x+2z+2a+2b+1]_{y}}{[2x+2z+2a+2b+1]_{y-1}} \notag\\
%&\times \frac{(x+z+2a+b+1)_{y-1}}{(x+z+2a+b+1)_{y}}\frac{(z+2a+b+1)_{y}}{(z+2a+b+1)_{y-1}}\notag\\
%&\times \frac{(y+z-1)_b}{(y+z)_b} \frac{(x+y+z)_b}{(x+y+z-1)_b}\notag\\
%&\times\frac{(x+y+z+a+b-1)}{(x+y+z+a+b)}\frac{(2y+2z+2b-1)!(2x+2y+2z+2a+2b-3)!}{(2y+2z+2b-3)!(2x+2y+2z+2a+2b-1)!}\notag\\
%&\times\frac{(y+z-2)!(x+y+z+b-2)!(x+y+z+2a+2b-1)!}{(y+z-1)!(x+y+z+b-1)!(x+y+z+2a+2b)!}\notag\\
%&\times\frac{(x+y+z-1)!(x+y+z+2a+b)!(2x+y+z+2a+2b)!}{(x+y+z-2)!(x+y+z+2a+b-1)!(2x+y+z+2a+2b-1)!}\\
&\frac{(y+z+b-1)(y+z+a+b)(2y+2z+2a+2b-2)(2y+2z+2a+2b-1)(y+z+b)}{(y+z+a+b-1)(y+z+b)(y+z+2a+b)(2y+2z+2b-2)(2y+2z+2b-1)}\notag\\
&\times \frac{(2x+2y+2z+2a+2b-1)}{(2y+2z+2a+2b-1)}\frac{(y+z+2a+b)}{(x+y+z+2a+b)}\notag\\
&\times \frac{(y+z-1)}{(y+z+b-1)} \frac{(x+y+z+b-1)}{(x+y+z-1)}\frac{(x+y+z+a+b-1)}{(x+y+z+a+b)}\notag\\
&\times\frac{(2y+2z+2b-2)(2y+2z+2b-1)}{(2x+2y+2z+2a+2b-2)(2x+2y+2z+2a+2b-1)}\notag\\
&\times\frac{(x+y+z-1)(x+y+z+2a+b)(2x+y+z+2a+2b)}{(y+z-1)(x+y+z+b-1)(x+y+z+2a+2b)}\notag\\
&=\frac{(y+z+a+b)(2x+y+z+2a+2b)}{(x+y+z+a+b)(x+y+z+2a+2b)}.
\end{align}

Next, we work on the second term in (\ref{recur2c}). We get the following simplifications:
\begin{align}
&\frac{\psi_{x-1,y,z}(a,b)}{\psi_{x,y,z}(a,b)}=\frac{\V(2x+2a+2b+3,y+z-1,y)}{\V(2x+2a+2b+1,y+z-1,y)}\notag\\
&\times \frac{\T(x+b,y+z+2a,y)}{\T(x+b+1,y+z+2a,y)}\frac{\T(x+2a+b+2,y+b-1,b)}{\T(x+2a+b+1,y+b-1,b)} \frac{\T(x+z+1,y+b-1,b)}{\T(x+z,y+b-1,b)}\notag\\
&\times \frac{Q(a,b,x-1,y+z)}{Q(a,b,x,y+z)}\\
&=\frac{[2x+2z+2a+2b+1]_y}{[2x+2a+2b+1]_y}\frac{(x+b)_y}{(x+z+2a+b+1)_y}\frac{(x+y+2a+b+1)_b}{(x+2a+b+1)_b}\frac{(x+y+z)_b}{(x+z)_b}\notag\\
&\times\frac{(x+a+b)}{(x+y+z+a+b)}\frac{1}{(2x+2y+2z+2a+2b-1)!(2x+2a+2b)!} \frac{(x+y+z-1)!(x+b-1)!}{(x+y+z+b-1)!(x-1)!}\notag\\
&\times\frac{(x+y+z+2a+b)!(x+2a+2b)!(2x+y+z+2a+2b-1)!(2x+y+z+2a+2b)!}{(x+y+z+2a+2b)!(x+2a+b)!}
\end{align}
and
\begin{align}
&\frac{\psi_{x+1,y-1,z-1}(a,b)}{\psi_{x,y-1,z-1}(a,b)}\frac{\psi_{x-1,y,z}(a,b)}{\psi_{x,y,z}(a,b)}\notag\\
%&=\frac{[2x+2a+2b+3]_{y-1}}{[2x+2a+2b+1]_y}\frac{[2x+2z+2a+2b+1]_y}{[2x+2z+2a+2b+1]_{y-1}}\notag\\
%&\times \frac{(x+z+2a+b+1)_{y-1}}{(x+z+2a+b+1)_y}\frac{(x+b)_y}{(x+b+1)_{y-1}}\frac{(x+2a+b+2)_b}{(x+2a+b+1)_b}\frac{(x+y+z)_b}{(x+y+z-1)_b}\notag\\
%&\times\frac{(x+a+b)(x+y+z+a+b-1)}{(x+a+b+1)(x+y+z+a+b)}\frac{(2x+2y+2z+2a+2b-3)!(2x+2a+2b+2)!}{(2x+2y+2z+2a+2b-1)!(2x+2a+2b)!}\notag\\
%&\times\frac{(x+y+z-1)!(x+b-1)!}{(x+y+z-2)!(x+b)!}\frac{(x+y+z+b-2)!(x)!}{(x+y+z+b-1)!(x-1)!}\notag\\
%&\frac{(x+y+z+2a+b)!(x+2a+2b)!}{(x+y+z+2a+b-1)!(x+2a+2b+1)!}\frac{(x+y+z+2a+2b-1)!(x+2a+b+1)!}{(x+y+z+2a+2b)!(x+2a+b)!}\notag\\
&=\frac{(2x+2y+2z+2a+2b-1)(x+b)}{(2x+2a+2b+1)(x+y+z+2a+b)}\frac{(x+2a+2b+1)}{(x+2a+b+1)}\frac{(x+y+z+b-1)}{(x+y+z-1)}\notag\\
&\times\frac{(x+a+b)(x+y+z+a+b-1)}{(x+a+b+1)(x+y+z+a+b)}\frac{(2x+2a+2b+1)(2x+2a+2b+2)}{(2x+2y+2z+2a+2b-2)(2x+2y+2z+2a+2b-1)}\notag\\
&\times\frac{(x+y+z-1)}{(x+b)}\frac{(x)}{(x+y+z+b-1)}\frac{(x+y+z+2a+b)}{(x+2a+2b+1)}\frac{(x+2a+b+1)}{(x+y+z+2a+2b)}\notag\\
&=\frac{x(x+a+b)}{(x+y+z+a+b)(x+y+z+2a+2b)}.
\end{align}
It is easy to see that
\begin{align}\
&\frac{\psi_{x,y-1,z}(a,b)}{\psi_{x,y-1,z-1}(a,b)}\frac{\psi_{x,y,z-1}(a,b)}{\psi_{x,y,z}(a,b)}+\frac{\psi_{x+1,y-1,z-1}(a,b)}{\psi_{x,y-1,z-1}(a,b)}\frac{\psi_{x-1,y,z}(a,b)}{\psi_{x,y,z}(a,b)}=\notag\\
&\frac{(y+z+a+b)(2x+y+z+2a+2b)}{(x+y+z+a+b)(x+y+z+2a+2b)}+\frac{x(x+a+b)}{(x+y+z+a+b)(x+y+z+2a+2b)}=1
\end{align}
and (\ref{recur2c}) follows.

\begin{figure}
  \centering
  \setlength{\unitlength}{3947sp}%
\begingroup\makeatletter\ifx\SetFigFont\undefined%
\gdef\SetFigFont#1#2#3#4#5{%
  \reset@font\fontsize{#1}{#2pt}%
  \fontfamily{#3}\fontseries{#4}\fontshape{#5}%
  \selectfont}%
\fi\endgroup%
\resizebox{10cm}{!}{
\begin{picture}(0,0)%
\includegraphics{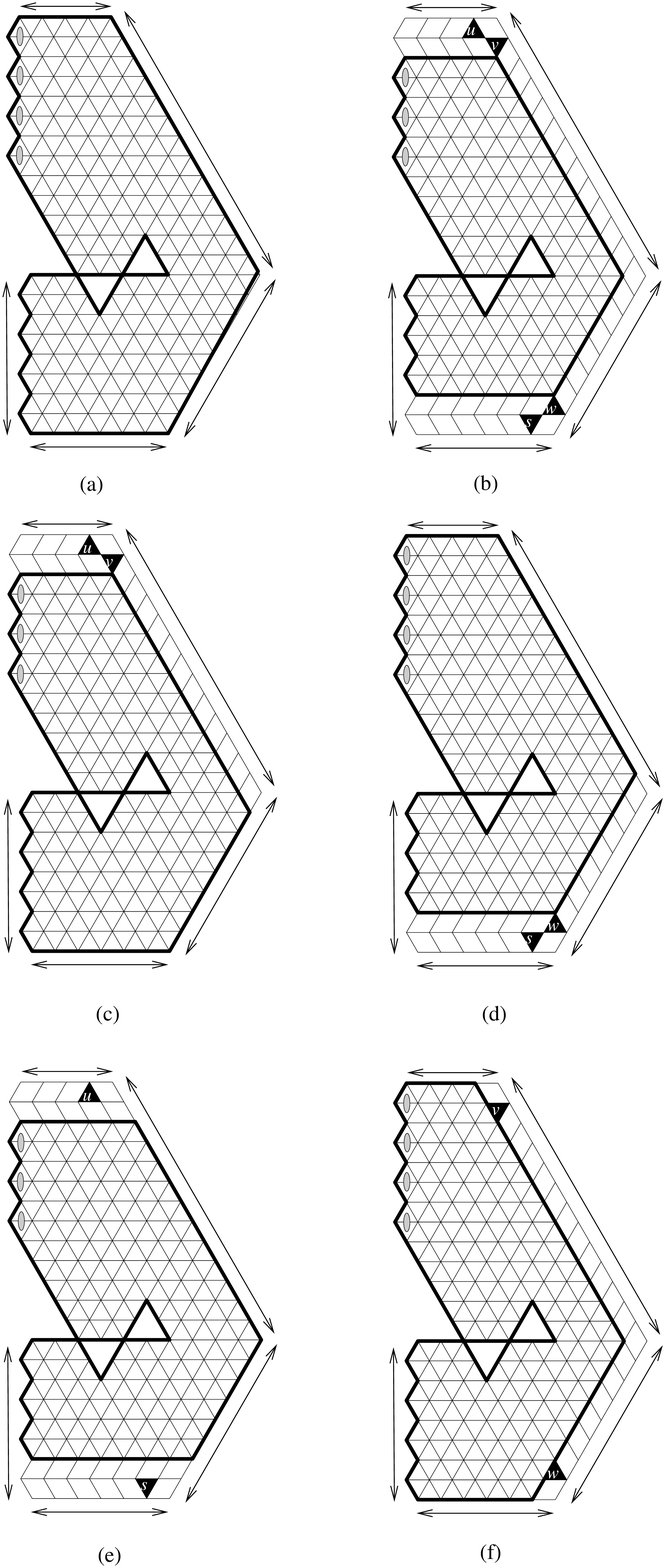}%
\end{picture}%
\begin{picture}(11599,27079)(2537,-26897)
%  METADATA <id>2162</id>
\put(9444,-25256){\rotatebox{90.0}{\makebox(0,0)[lb]{\smash{{\SetFigFont{17}{20.4}{\rmdefault}{\mddefault}{\itdefault}{$2z+2a_2$}%
}}}}}
%  METADATA <id>2161</id>
\put(10514,-26226){\makebox(0,0)[lb]{\smash{{\SetFigFont{17}{20.4}{\rmdefault}{\mddefault}{\itdefault}{$x+a_1+a_3$}%
}}}}
%  METADATA <id>2160</id>
\put(13184,-25336){\rotatebox{60.0}{\makebox(0,0)[lb]{\smash{{\SetFigFont{17}{20.4}{\rmdefault}{\mddefault}{\itdefault}{$y+z+2a_2$}%
}}}}}
%  METADATA <id>2159</id>
\put(12134,-19236){\rotatebox{300.0}{\makebox(0,0)[lb]{\smash{{\SetFigFont{17}{20.4}{\rmdefault}{\mddefault}{\itdefault}{$y+z+2a_1+2a_3+1$}%
}}}}}
%  METADATA <id>2158</id>
\put(10274,-18286){\makebox(0,0)[lb]{\smash{{\SetFigFont{17}{20.4}{\rmdefault}{\mddefault}{\itdefault}{$x+a_2$}%
}}}}
%  METADATA <id>2157</id>
\put(11714,-22936){\makebox(0,0)[lb]{\smash{{\SetFigFont{17}{20.4}{\rmdefault}{\mddefault}{\itdefault}{$a_3$}%
}}}}
%  METADATA <id>2156</id>
\put(10964,-23266){\makebox(0,0)[lb]{\smash{{\SetFigFont{17}{20.4}{\rmdefault}{\mddefault}{\itdefault}{$a_2$}%
}}}}
%  METADATA <id>2155</id>
\put(10084,-22896){\makebox(0,0)[lb]{\smash{{\SetFigFont{17}{20.4}{\rmdefault}{\mddefault}{\itdefault}{$a_1$}%
}}}}
%  METADATA <id>2004</id>
\put(2885,-25236){\rotatebox{90.0}{\makebox(0,0)[lb]{\smash{{\SetFigFont{17}{20.4}{\rmdefault}{\mddefault}{\itdefault}{$2z+2a_2$}%
}}}}}
%  METADATA <id>2003</id>
\put(3955,-26206){\makebox(0,0)[lb]{\smash{{\SetFigFont{17}{20.4}{\rmdefault}{\mddefault}{\itdefault}{$x+a_1+a_3$}%
}}}}
%  METADATA <id>2002</id>
\put(6625,-25316){\rotatebox{60.0}{\makebox(0,0)[lb]{\smash{{\SetFigFont{17}{20.4}{\rmdefault}{\mddefault}{\itdefault}{$y+z+2a_2$}%
}}}}}
%  METADATA <id>2001</id>
\put(5575,-19216){\rotatebox{300.0}{\makebox(0,0)[lb]{\smash{{\SetFigFont{17}{20.4}{\rmdefault}{\mddefault}{\itdefault}{$y+z+2a_1+2a_3+1$}%
}}}}}
%  METADATA <id>2000</id>
\put(3715,-18266){\makebox(0,0)[lb]{\smash{{\SetFigFont{17}{20.4}{\rmdefault}{\mddefault}{\itdefault}{$x+a_2$}%
}}}}
%  METADATA <id>1999</id>
\put(5155,-22916){\makebox(0,0)[lb]{\smash{{\SetFigFont{17}{20.4}{\rmdefault}{\mddefault}{\itdefault}{$a_3$}%
}}}}
%  METADATA <id>1998</id>
\put(4405,-23246){\makebox(0,0)[lb]{\smash{{\SetFigFont{17}{20.4}{\rmdefault}{\mddefault}{\itdefault}{$a_2$}%
}}}}
%  METADATA <id>1997</id>
\put(3525,-22876){\makebox(0,0)[lb]{\smash{{\SetFigFont{17}{20.4}{\rmdefault}{\mddefault}{\itdefault}{$a_1$}%
}}}}
%  METADATA <id>1846</id>
\put(9444,-15941){\rotatebox{90.0}{\makebox(0,0)[lb]{\smash{{\SetFigFont{17}{20.4}{\rmdefault}{\mddefault}{\itdefault}{$2z+2a_2$}%
}}}}}
%  METADATA <id>1845</id>
\put(10514,-16911){\makebox(0,0)[lb]{\smash{{\SetFigFont{17}{20.4}{\rmdefault}{\mddefault}{\itdefault}{$x+a_1+a_3$}%
}}}}
%  METADATA <id>1844</id>
\put(13184,-16021){\rotatebox{60.0}{\makebox(0,0)[lb]{\smash{{\SetFigFont{17}{20.4}{\rmdefault}{\mddefault}{\itdefault}{$y+z+2a_2$}%
}}}}}
%  METADATA <id>1843</id>
\put(12134,-9921){\rotatebox{300.0}{\makebox(0,0)[lb]{\smash{{\SetFigFont{17}{20.4}{\rmdefault}{\mddefault}{\itdefault}{$y+z+2a_1+2a_3+1$}%
}}}}}
%  METADATA <id>1842</id>
\put(10274,-8971){\makebox(0,0)[lb]{\smash{{\SetFigFont{17}{20.4}{\rmdefault}{\mddefault}{\itdefault}{$x+a_2$}%
}}}}
%  METADATA <id>1839</id>
\put(10084,-13581){\makebox(0,0)[lb]{\smash{{\SetFigFont{17}{20.4}{\rmdefault}{\mddefault}{\itdefault}{$a_1$}%
}}}}
%  METADATA <id>1688</id>
\put(2885,-15921){\rotatebox{90.0}{\makebox(0,0)[lb]{\smash{{\SetFigFont{17}{20.4}{\rmdefault}{\mddefault}{\itdefault}{$2z+2a_2$}%
}}}}}
%  METADATA <id>1687</id>
\put(3955,-16891){\makebox(0,0)[lb]{\smash{{\SetFigFont{17}{20.4}{\rmdefault}{\mddefault}{\itdefault}{$x+a_1+a_3$}%
}}}}
%  METADATA <id>1686</id>
\put(6625,-16001){\rotatebox{60.0}{\makebox(0,0)[lb]{\smash{{\SetFigFont{17}{20.4}{\rmdefault}{\mddefault}{\itdefault}{$y+z+2a_2$}%
}}}}}
%  METADATA <id>1685</id>
\put(5575,-9901){\rotatebox{300.0}{\makebox(0,0)[lb]{\smash{{\SetFigFont{17}{20.4}{\rmdefault}{\mddefault}{\itdefault}{$y+z+2a_1+2a_3+1$}%
}}}}}
%  METADATA <id>1684</id>
\put(3715,-8951){\makebox(0,0)[lb]{\smash{{\SetFigFont{17}{20.4}{\rmdefault}{\mddefault}{\itdefault}{$x+a_2$}%
}}}}
%  METADATA <id>1683</id>
\put(5155,-13601){\makebox(0,0)[lb]{\smash{{\SetFigFont{17}{20.4}{\rmdefault}{\mddefault}{\itdefault}{$a_3$}%
}}}}
%  METADATA <id>1682</id>
\put(4405,-13931){\makebox(0,0)[lb]{\smash{{\SetFigFont{17}{20.4}{\rmdefault}{\mddefault}{\itdefault}{$a_2$}%
}}}}
%  METADATA <id>1681</id>
\put(3525,-13561){\makebox(0,0)[lb]{\smash{{\SetFigFont{17}{20.4}{\rmdefault}{\mddefault}{\itdefault}{$a_1$}%
}}}}
%  METADATA <id>1056</id>
\put(9420,-7131){\rotatebox{90.0}{\makebox(0,0)[lb]{\smash{{\SetFigFont{17}{20.4}{\rmdefault}{\mddefault}{\itdefault}{$2z+2a_2$}%
}}}}}
%  METADATA <id>1055</id>
\put(10490,-8101){\makebox(0,0)[lb]{\smash{{\SetFigFont{17}{20.4}{\rmdefault}{\mddefault}{\itdefault}{$x+a_1+a_3$}%
}}}}
%  METADATA <id>1054</id>
\put(13160,-7211){\rotatebox{60.0}{\makebox(0,0)[lb]{\smash{{\SetFigFont{17}{20.4}{\rmdefault}{\mddefault}{\itdefault}{$y+z+2a_2$}%
}}}}}
%  METADATA <id>1053</id>
\put(12110,-1111){\rotatebox{300.0}{\makebox(0,0)[lb]{\smash{{\SetFigFont{17}{20.4}{\rmdefault}{\mddefault}{\itdefault}{$y+z+2a_1+2a_3+1$}%
}}}}}
%  METADATA <id>1052</id>
\put(10250,-161){\makebox(0,0)[lb]{\smash{{\SetFigFont{17}{20.4}{\rmdefault}{\mddefault}{\itdefault}{$x+a_2$}%
}}}}
%  METADATA <id>1051</id>
\put(11690,-4811){\makebox(0,0)[lb]{\smash{{\SetFigFont{17}{20.4}{\rmdefault}{\mddefault}{\itdefault}{$a_3$}%
}}}}
%  METADATA <id>1050</id>
\put(10940,-5141){\makebox(0,0)[lb]{\smash{{\SetFigFont{17}{20.4}{\rmdefault}{\mddefault}{\itdefault}{$a_2$}%
}}}}
%  METADATA <id>1049</id>
\put(10060,-4771){\makebox(0,0)[lb]{\smash{{\SetFigFont{17}{20.4}{\rmdefault}{\mddefault}{\itdefault}{$a_1$}%
}}}}
%  METADATA <id>740</id>
\put(2861,-7111){\rotatebox{90.0}{\makebox(0,0)[lb]{\smash{{\SetFigFont{17}{20.4}{\rmdefault}{\mddefault}{\itdefault}{$2z+2a_2$}%
}}}}}
%  METADATA <id>739</id>
\put(3931,-8081){\makebox(0,0)[lb]{\smash{{\SetFigFont{17}{20.4}{\rmdefault}{\mddefault}{\itdefault}{$x+a_1+a_3$}%
}}}}
%  METADATA <id>738</id>
\put(6601,-7191){\rotatebox{60.0}{\makebox(0,0)[lb]{\smash{{\SetFigFont{17}{20.4}{\rmdefault}{\mddefault}{\itdefault}{$y+z+2a_2$}%
}}}}}
%  METADATA <id>737</id>
\put(5551,-1091){\rotatebox{300.0}{\makebox(0,0)[lb]{\smash{{\SetFigFont{17}{20.4}{\rmdefault}{\mddefault}{\itdefault}{$y+z+2a_1+2a_3+1$}%
}}}}}
%  METADATA <id>736</id>
\put(3691,-141){\makebox(0,0)[lb]{\smash{{\SetFigFont{17}{20.4}{\rmdefault}{\mddefault}{\itdefault}{$x+a_2$}%
}}}}
%  METADATA <id>735</id>
\put(5131,-4791){\makebox(0,0)[lb]{\smash{{\SetFigFont{17}{20.4}{\rmdefault}{\mddefault}{\itdefault}{$a_3$}%
}}}}
%  METADATA <id>734</id>
\put(4381,-5121){\makebox(0,0)[lb]{\smash{{\SetFigFont{17}{20.4}{\rmdefault}{\mddefault}{\itdefault}{$a_2$}%
}}}}
%  METADATA <id>733</id>
\put(3501,-4751){\makebox(0,0)[lb]{\smash{{\SetFigFont{17}{20.4}{\rmdefault}{\mddefault}{\itdefault}{$a_1$}%
}}}}
%  METADATA <id>1841</id>
\put(11714,-13621){\makebox(0,0)[lb]{\smash{{\SetFigFont{17}{20.4}{\rmdefault}{\mddefault}{\itdefault}{$a_3$}%
}}}}
%  METADATA <id>1840</id>
\put(10964,-13951){\makebox(0,0)[lb]{\smash{{\SetFigFont{17}{20.4}{\rmdefault}{\mddefault}{\itdefault}{$a_2$}%
}}}}
\end{picture}}
  \caption{Applying Kuo condensation to a $H^{(5)}$-type region.}\label{Kuomidhole3}
\end{figure}

We also prove (\ref{main5eq2}) by induction on $x+y+z$. The base cases are also the situations when at least one of the three parameters $x,y,z$ is equal to $0$. This cases follows from the equations when proving the base cases for (\ref{main5eq1}).

The induction step here is exactly the same as that in the proof of  (\ref{main5eq1}). We only need to show that the formula on the right-hand side of (\ref{main5eq2}) also  satisfies the recurrence (\ref{recur2}).
Similar to the proof of Theorem \ref{main1}, this verification follows from (\ref{recur2c}).
\end{proof}

Theorems \ref{main6}--\ref{main8} can be treated similarly to Theorem \ref{main5} in a completely analogous manner, and we also omit the proofs of these theorems.

We finally prove Theorem \ref{main9}.

\begin{figure}\centering
\setlength{\unitlength}{3947sp}%
\begingroup\makeatletter\ifx\SetFigFont\undefined%
\gdef\SetFigFont#1#2#3#4#5{%
  \reset@font\fontsize{#1}{#2pt}%
  \fontfamily{#3}\fontseries{#4}\fontshape{#5}%
  \selectfont}%
\fi\endgroup%
\resizebox{13cm}{!}{
\begin{picture}(0,0)%
\includegraphics{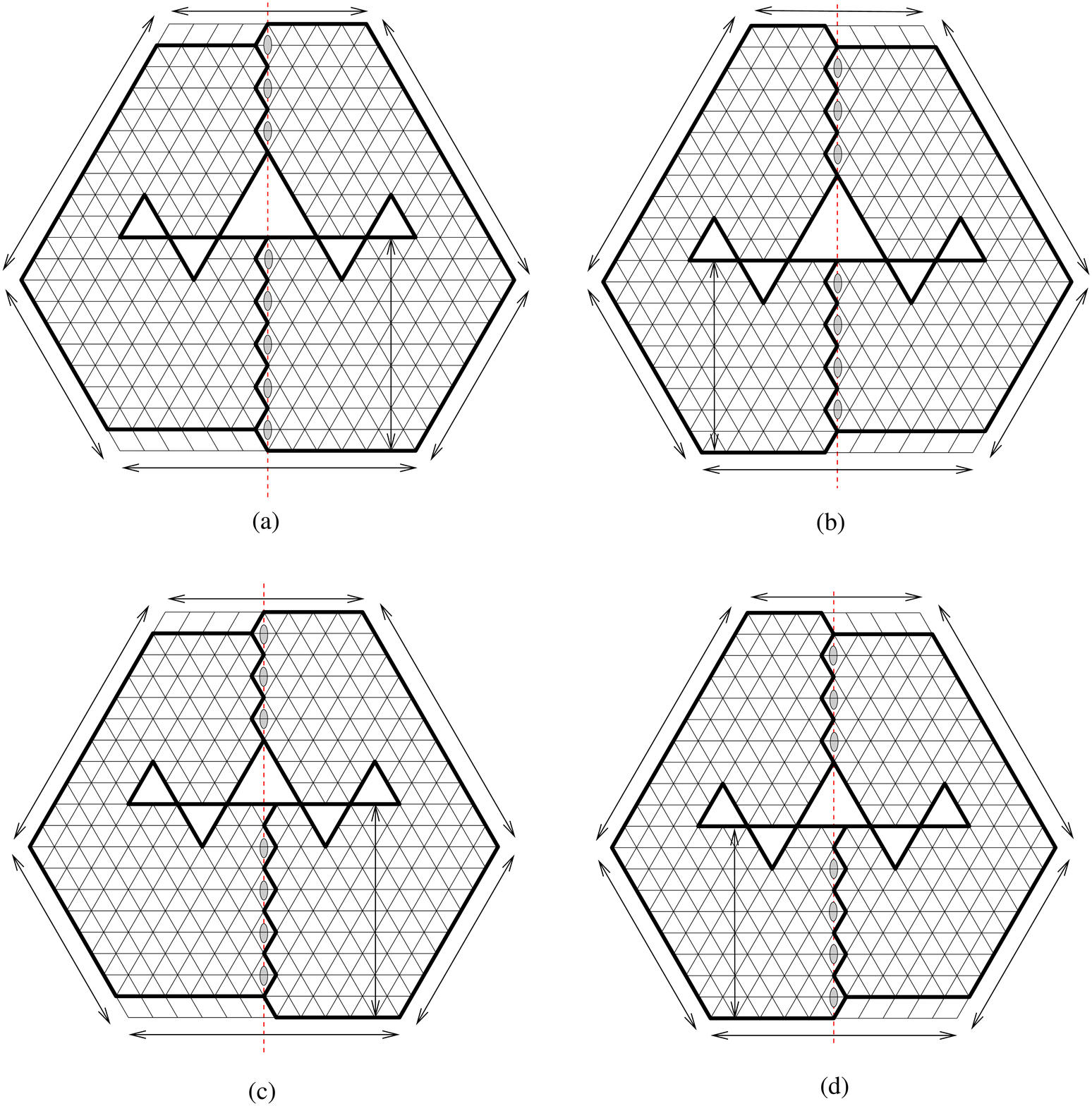}%
\end{picture}%
\begin{picture}(17250,17855)(1476,-17997)
%  METADATA <id>574</id>
\put(5545,-3942){\makebox(0,0)[lb]{\smash{{\SetFigFont{20}{24.0}{\rmdefault}{\mddefault}{\itdefault}{$a_1$}%
}}}}
%  METADATA <id>575</id>
\put(6700,-4497){\makebox(0,0)[lb]{\smash{{\SetFigFont{20}{24.0}{\rmdefault}{\mddefault}{\itdefault}{$a_2$}%
}}}}
%  METADATA <id>576</id>
\put(7480,-4122){\makebox(0,0)[lb]{\smash{{\SetFigFont{20}{24.0}{\rmdefault}{\mddefault}{\itdefault}{$a_3$}%
}}}}
%  METADATA <id>577</id>
\put(4383,-4512){\makebox(0,0)[lb]{\smash{{\SetFigFont{20}{24.0}{\rmdefault}{\mddefault}{\itdefault}{$a_2$}%
}}}}
%  METADATA <id>578</id>
\put(3580,-4114){\makebox(0,0)[lb]{\smash{{\SetFigFont{20}{24.0}{\rmdefault}{\mddefault}{\itdefault}{$a_3$}%
}}}}
%  METADATA <id>597</id>
\put(5111,-501){\makebox(0,0)[lb]{\smash{{\SetFigFont{20}{24.0}{\rmdefault}{\mddefault}{\itdefault}{$x+2a_2$}%
}}}}
%  METADATA <id>598</id>
\put(4831,-8231){\makebox(0,0)[lb]{\smash{{\SetFigFont{20}{24.0}{\rmdefault}{\mddefault}{\itdefault}{$x+a_1+2a_3$}%
}}}}
%  METADATA <id>599</id>
\put(8241,-1651){\rotatebox{300.0}{\makebox(0,0)[lb]{\smash{{\SetFigFont{20}{24.0}{\rmdefault}{\mddefault}{\itdefault}{$y+a_1+2a_3$}%
}}}}}
%  METADATA <id>600</id>
\put(8981,-7201){\rotatebox{60.0}{\makebox(0,0)[lb]{\smash{{\SetFigFont{20}{24.0}{\rmdefault}{\mddefault}{\itdefault}{$y+2a_2$}%
}}}}}
%  METADATA <id>602</id>
\put(1581,-5881){\rotatebox{300.0}{\makebox(0,0)[lb]{\smash{{\SetFigFont{20}{24.0}{\rmdefault}{\mddefault}{\itdefault}{$y+2a_2$}%
}}}}}
%  METADATA <id>604</id>
\put(2131,-3361){\rotatebox{60.0}{\makebox(0,0)[lb]{\smash{{\SetFigFont{20}{24.0}{\rmdefault}{\mddefault}{\itdefault}{$y+a_1+2a_3$}%
}}}}}
%  METADATA <id>606</id>
\put(7731,-5851){\makebox(0,0)[lb]{\smash{{\SetFigFont{20}{24.0}{\rmdefault}{\mddefault}{\itdefault}{$z$}%
}}}}
%  METADATA <id>609</id>
\put(12480,-5912){\makebox(0,0)[lb]{\smash{{\SetFigFont{20}{24.0}{\rmdefault}{\mddefault}{\itdefault}{$z$}%
}}}}
%  METADATA <id>611</id>
\put(14585,-4332){\makebox(0,0)[lb]{\smash{{\SetFigFont{20}{24.0}{\rmdefault}{\mddefault}{\itdefault}{$a_1$}%
}}}}
%  METADATA <id>613</id>
\put(13405,-4892){\makebox(0,0)[lb]{\smash{{\SetFigFont{20}{24.0}{\rmdefault}{\mddefault}{\itdefault}{$a_2$}%
}}}}
%  METADATA <id>615</id>
\put(15725,-4902){\makebox(0,0)[lb]{\smash{{\SetFigFont{20}{24.0}{\rmdefault}{\mddefault}{\itdefault}{$a_2$}%
}}}}
%  METADATA <id>617</id>
\put(12575,-4492){\makebox(0,0)[lb]{\smash{{\SetFigFont{20}{24.0}{\rmdefault}{\mddefault}{\itdefault}{$a_3$}%
}}}}
%  METADATA <id>618</id>
\put(16495,-4492){\makebox(0,0)[lb]{\smash{{\SetFigFont{20}{24.0}{\rmdefault}{\mddefault}{\itdefault}{$a_3$}%
}}}}
%  METADATA <id>676</id>
\put(14205,-532){\makebox(0,0)[lb]{\smash{{\SetFigFont{20}{24.0}{\rmdefault}{\mddefault}{\itdefault}{$x+2a_2$}%
}}}}
%  METADATA <id>678</id>
\put(17068,-1751){\rotatebox{300.0}{\makebox(0,0)[lb]{\smash{{\SetFigFont{20}{24.0}{\rmdefault}{\mddefault}{\itdefault}{$y+a_1+2a_3$}%
}}}}}
%  METADATA <id>680</id>
\put(11332,-3539){\rotatebox{60.0}{\makebox(0,0)[lb]{\smash{{\SetFigFont{20}{24.0}{\rmdefault}{\mddefault}{\itdefault}{$y+a_1+2a_3$}%
}}}}}
%  METADATA <id>682</id>
\put(10916,-5950){\rotatebox{300.0}{\makebox(0,0)[lb]{\smash{{\SetFigFont{20}{24.0}{\rmdefault}{\mddefault}{\itdefault}{$y+2a_2$}%
}}}}}
%  METADATA <id>684</id>
\put(17820,-7089){\rotatebox{60.0}{\makebox(0,0)[lb]{\smash{{\SetFigFont{20}{24.0}{\rmdefault}{\mddefault}{\itdefault}{$y+2a_2$}%
}}}}}
%  METADATA <id>686</id>
\put(13800,-8242){\makebox(0,0)[lb]{\smash{{\SetFigFont{20}{24.0}{\rmdefault}{\mddefault}{\itdefault}{$x+a_1+2a_3$}%
}}}}
%  METADATA <id>688</id>
\put(13688,-17199){\makebox(0,0)[lb]{\smash{{\SetFigFont{20}{24.0}{\rmdefault}{\mddefault}{\itdefault}{$x+a_1+2a_3$}%
}}}}
%  METADATA <id>689</id>
\put(4556,-17230){\makebox(0,0)[lb]{\smash{{\SetFigFont{20}{24.0}{\rmdefault}{\mddefault}{\itdefault}{$x+a_1+2a_3$}%
}}}}
%  METADATA <id>691</id>
\put(8750,-16106){\rotatebox{60.0}{\makebox(0,0)[lb]{\smash{{\SetFigFont{20}{24.0}{\rmdefault}{\mddefault}{\itdefault}{$y+2a_2$}%
}}}}}
%  METADATA <id>692</id>
\put(17578,-16116){\rotatebox{60.0}{\makebox(0,0)[lb]{\smash{{\SetFigFont{20}{24.0}{\rmdefault}{\mddefault}{\itdefault}{$y+2a_2$}%
}}}}}
%  METADATA <id>694</id>
\put(11360,-12676){\rotatebox{60.0}{\makebox(0,0)[lb]{\smash{{\SetFigFont{20}{24.0}{\rmdefault}{\mddefault}{\itdefault}{$y+a_1+2a_3$}%
}}}}}
%  METADATA <id>695</id>
\put(2072,-12746){\rotatebox{60.0}{\makebox(0,0)[lb]{\smash{{\SetFigFont{20}{24.0}{\rmdefault}{\mddefault}{\itdefault}{$y+a_1+2a_3$}%
}}}}}
%  METADATA <id>697</id>
\put(8098,-10708){\rotatebox{300.0}{\makebox(0,0)[lb]{\smash{{\SetFigFont{20}{24.0}{\rmdefault}{\mddefault}{\itdefault}{$y+a_1+2a_3$}%
}}}}}
%  METADATA <id>699</id>
\put(1726,-14867){\rotatebox{300.0}{\makebox(0,0)[lb]{\smash{{\SetFigFont{20}{24.0}{\rmdefault}{\mddefault}{\itdefault}{$y+2a+2$}%
}}}}}
%  METADATA <id>700</id>
\put(11154,-15097){\rotatebox{300.0}{\makebox(0,0)[lb]{\smash{{\SetFigFont{20}{24.0}{\rmdefault}{\mddefault}{\itdefault}{$y+2a+2$}%
}}}}}
%  METADATA <id>702</id>
\put(16906,-10838){\rotatebox{300.0}{\makebox(0,0)[lb]{\smash{{\SetFigFont{20}{24.0}{\rmdefault}{\mddefault}{\itdefault}{$y+a_1+2a_3$}%
}}}}}
%  METADATA <id>704</id>
\put(14053,-9729){\makebox(0,0)[lb]{\smash{{\SetFigFont{20}{24.0}{\rmdefault}{\mddefault}{\itdefault}{$x+2a_2$}%
}}}}
%  METADATA <id>705</id>
\put(5001,-9805){\makebox(0,0)[lb]{\smash{{\SetFigFont{20}{24.0}{\rmdefault}{\mddefault}{\itdefault}{$x+2a_2$}%
}}}}
%  METADATA <id>709</id>
\put(12828,-15179){\makebox(0,0)[lb]{\smash{{\SetFigFont{20}{24.0}{\rmdefault}{\mddefault}{\itdefault}{$z$}%
}}}}
%  METADATA <id>713</id>
\put(7520,-15159){\makebox(0,0)[lb]{\smash{{\SetFigFont{20}{24.0}{\rmdefault}{\mddefault}{\itdefault}{$z$}%
}}}}
%  METADATA <id>715</id>
\put(14493,-13414){\makebox(0,0)[lb]{\smash{{\SetFigFont{20}{24.0}{\rmdefault}{\mddefault}{\itdefault}{$a_1$}%
}}}}
%  METADATA <id>716</id>
\put(5485,-13044){\makebox(0,0)[lb]{\smash{{\SetFigFont{20}{24.0}{\rmdefault}{\mddefault}{\itdefault}{$a_1$}%
}}}}
%  METADATA <id>718</id>
\put(4535,-13499){\makebox(0,0)[lb]{\smash{{\SetFigFont{20}{24.0}{\rmdefault}{\mddefault}{\itdefault}{$a_2$}%
}}}}
%  METADATA <id>719</id>
\put(6495,-13499){\makebox(0,0)[lb]{\smash{{\SetFigFont{20}{24.0}{\rmdefault}{\mddefault}{\itdefault}{$a_2$}%
}}}}
%  METADATA <id>720</id>
\put(13563,-13849){\makebox(0,0)[lb]{\smash{{\SetFigFont{20}{24.0}{\rmdefault}{\mddefault}{\itdefault}{$a_2$}%
}}}}
%  METADATA <id>721</id>
\put(15483,-13829){\makebox(0,0)[lb]{\smash{{\SetFigFont{20}{24.0}{\rmdefault}{\mddefault}{\itdefault}{$a_2$}%
}}}}
%  METADATA <id>723</id>
\put(3715,-13079){\makebox(0,0)[lb]{\smash{{\SetFigFont{20}{24.0}{\rmdefault}{\mddefault}{\itdefault}{$a_3$}%
}}}}
%  METADATA <id>724</id>
\put(7225,-13069){\makebox(0,0)[lb]{\smash{{\SetFigFont{20}{24.0}{\rmdefault}{\mddefault}{\itdefault}{$a_3$}%
}}}}
%  METADATA <id>725</id>
\put(12753,-13439){\makebox(0,0)[lb]{\smash{{\SetFigFont{20}{24.0}{\rmdefault}{\mddefault}{\itdefault}{$a_3$}%
}}}}
%  METADATA <id>726</id>
\put(16243,-13439){\makebox(0,0)[lb]{\smash{{\SetFigFont{20}{24.0}{\rmdefault}{\mddefault}{\itdefault}{$a_3$}%
}}}}
\end{picture}}
\caption{Apply Ciucu's factorization theorem to a symmetric hexagon with an array of holes on the symmetric axis.}\label{middlearrayn}
\end{figure}

\begin{proof}[Proof of Theorem \ref{main9}]
Apply Ciucu's Factorization Theorem \ref{ciucuthm} to the dual graph $G$ of $\mathcal{S}=\mathcal{S}_{x,y,z}(\textbf{a})$, we obtain
\begin{equation}\label{main9eq1}
\M(\mathcal{S})=\M(G)=2^{y+a_2+\dotsc+a_n}\M(G^+)\M(G^-).
\end{equation}
We consider first the case $x$ and $a_1$ are both even. Note that we also have $z$ even in this case ($z$ and $x$ always have the same parity). It turns out that $G^-$ is the dual graph of the region
\[H^{(3)}_{\frac{x}{2}+\E(\textbf{a}),y-\frac{z}{2}+\E(\textbf{a}) ,\frac{z}{2}-\E(\textbf{a})}\left(\frac{a_1}{2},a_2,\dotsc,a_n\right)\]
(see Figure \ref{middlearrayn}(a)). Thus,
\begin{equation}\label{main9eq2}
\M(G^-)=H^{(3)}_{\frac{x}{2}+\E(\textbf{a}),y-\frac{z}{2}+\E(\textbf{a}) ,\frac{z}{2}-\E(\textbf{a})}\left(\frac{a_1}{2},a_2,\dotsc,a_n\right).
\end{equation}
 The graph $G^+$ after removed several forced edges (the edges corresponding to the forced lozenges in Figure \ref{middlearrayn}(a))  is the dual graph of the region
\[H^{(2)}_{\frac{x}{2}+\E(\textbf{a}),y-\frac{z}{2}+\E(\textbf{a}) ,\frac{z}{2}-\E(\textbf{a})}\left(\frac{a_1}{2},a_2,\dotsc,a_n\right).\]
Since the removal of these forced edges does not affect to the number of matchings, we have 
\begin{equation}\label{main9eq3}
\M(G^+)=H^{(2)}_{\frac{x}{2}+\E(\textbf{a}),y-\frac{z}{2}+\E(\textbf{a}) ,\frac{z}{2}-\E(\textbf{a})}\left(\frac{a_1}{2},a_2,\dotsc,a_n\right).
\end{equation}
Then part (1) follows from (\ref{main9eq1})--(\ref{main9eq3}).

Parts (2), (3), and (4) can be similarly treated by using Ciucu's Factorization Theorem, based on Figures \ref{middlearrayn}(b), (c), and (d), respectively.
\end{proof}

\begin{rmk}\label{remark1}
If the parameter $z$ in the region $\mathcal{S}=\mathcal{S}_{x,y,z}(\textbf{a})$ is too small or too large, then the region does not have a tiling.  We can still apply Ciucu's Factorization Theorem to the dual graph $G$ of the region $\mathcal{S}$ for any $z$, and still get the equality (\ref{main9eq1}). However,  if $z<2\E(\textbf{a})-1$ or $z>2y+2\E(\textbf{a})+1$, then at least one of the component graphs $G^+$ and $G^-$ has no perfect matching. To see this, we consider the regions corresponding to them. Let us work in detail here for  the case of even $x$ and $a_1$ (the other cases can be treated in the same manner). If $z<2\E(\textbf{a})-1$, then the region corresponding to $G^-$ is still an $H^{(3)}$-type region, denoted by $R$, however, it has a negative $y$-parameter. In this case, divide the region into two subregions, denoted $R_1$ and $R_2$ for the upper and lower ones, along the line containing the bases of the triangular holes. By the same arguments in the proof of the Region-Splitting Lemma \ref{RS}, we can show that a tiling of $R$ (if exists) can be written as union of two independent tilings of $R_1$ and $R_2$. However, it is easy to see that $R_2$ has no tiling (since it is not balanced). Similar in the case of $z>2y+2\E(\textbf{a})+1$, the upper subregion $R_1$ has no tiling. In summary, $R$ has no tiling, so does $\mathcal{S}$.
\end{rmk}

%\section{Verify that our tiling formula satisfy our recurrence}\label{Recurrence}

%\section{Tilings of hexagons with three arrays of holes}\label{Threearrays}

\end{document}